\documentclass[11pt,twoside]{amsart}
\usepackage{hyperref, tikz, fullpage}
\usepackage{placeins, enumerate,longtable}
\usepackage{amsthm, amssymb,amsmath,latexsym}
\usepackage[utf8]{inputenc}
\usepackage[T1]{fontenc}
\usepackage{enumitem,color}
        
\usepackage{caption}
\usepackage{lscape}
\newtheorem{theorem}{Theorem}[section]
\newtheorem{prop}[theorem]{Proposition}

\numberwithin{equation}{section}
\newenvironment{psmallmatrix}
  {\left(\begin{smallmatrix}}
  {\end{smallmatrix}\right)}

\newcommand{\Reg}{\mathrm{Regular}}

\newcommand{\F}{\mathbb{F}}
\newcommand{\Fq}{\mathbf{F}_q}

\newcommand{\bmf}[1]{\mathbf{#1}}

\newcommand{\lbd}{\lambda}

\newcommand{\Zu}[1]{Z_{U_{#1}(\Fq)}}

\newcommand{\ZT}[1]{Z_{GT_{#1}(\Fq)}}
\newcommand{\Tq}[1]{GT_{#1}(\Fq)}
\newcommand{\tq}[1]{T_{#1}(\Fq)}

\newcommand{\ZU}[1]{Z_{UT_{#1}(\Fq)}}
\begin{document}
\title{Branching rules and commuting probabilities for Triangular and Unitriangular matrices}
\author{Dilpreet Kaur}
\email{dilpreetmaths@gmail.com}
\address{Department of Mathematics,  Indian Institute of Technology (IIT) Jodhpur, NH 65, Surpura Bypass Rd, Karwar, Rajasthan 342037}
\author{Uday Bhaskar Sharma}
\email{udaybsharmaster@gmail.com}
\address{Tata Institute of Fundamental Research, Dr. Homi Bhabha Road, Navy Nagar, Colaba, Mumbai 400005, India} 
\author{Anupam Singh}
\email{anupamk18@gmail.com}
\address{Indian Institute of Science Education and Research (IISER) Pune,  Dr. Homi Bhabha Road Pashan, Pune 411008, India}
\today
\thanks{The third named author would like to acknowledge support of SERB grant CRG/2019/000271 during this project.}


\begin{abstract}
This paper concerns the enumeration of simultaneous conjugacy classes of $k$-tuples of commuting matrices in the upper triangular group $GT_n(\mathbf F_q)$ and unitriangular group $UT_m(\mathbf F_q)$ over the finite field $\mathbf F_q$ of odd characteristic. This is done for $n=2,3,4$ and $m=3,4,5$, by computing the branching rules. Further, using the branching matrix thus computed, we explicitly get the commuting probabilities $cp_k$ for $k\leq 5$ in each case.  
\end{abstract}
\maketitle


\section{Introduction}
Simultaneous conjugacy of commuting $k$-tuples in a group is understood by computing its branching matrix. In~\cite{Sh1} and~\cite{SS}, the branching table/matrix of finite general linear, unitary and symplectic groups of small rank is computed. In this paper, we continue the work for certain solvable groups, namely, upper triangular matrices. Since, this work is continuation of that in~\cite{SS}, we refer a reader to the same for definition of branching and other related notation. We work with the groups of upper-triangular invertible matrices, $GT_n(\Fq)$, and the groups of upper unitriangular matrices $UT_n(\Fq)$, over a finite field $\Fq$ of odd characteristic. We compute the branching matrix for $GT_2(\Fq)$ (Theorem~\ref{TheoremGT2}), $GT_3(\Fq)$ (Theorem~\ref{TheoremGT3}), $GT_4(\Fq)$ (Theorem~\ref{TheoremGT4}), $UT_3(\Fq)$ (Theorem~\ref{TheoremUT3}), $UT_4(\Fq)$ (Theorem~\ref{TheoremUT4}) and $UT_5(\Fq)$ (Theorem~\ref{TheoremUT5}). 

Further, for a group $G$, the relation of branching matrix $B_G$ to commuting probabilities $cp_k(G)$ was explored in \cite[Theorem 1.1]{SS}. This relation is further explored in the survey article~\cite{SS2}, where commuting probabilities $cp_k(G)$ up to $k\leq 5$ is computed for $G=GL_2(\Fq)$, $GL_3(\Fq)$, $U_2(\Fq)$ and $U_3(\Fq)$. It was also proved that $cp_k(GL_2(\F_q)) = cp_k(U_2(\Fq))$ for all $k$ even though the branching matrices of the two groups are not same (see Proposition 3.3~\cite{SS2}). In~\cite{GR} (see Theorem 12) bounds for commuting probability $cp_2$, when $G$ is a solvable group or $p$-group, is computed. Using the branching matrix we compute the commuting probabilities $cp_k$, up to $k\leq 5$, for each of the groups $GT_n(\Fq)$ and $UT_n(\Fq)$ for which we have branching matrix (see Section~\ref{SCPk}).
 
For this work, we need {\it conjugacy class types} or {\it $z$-classes} (as defined in~\cite{SS} and also dealt in~\cite{Bh}). This is defined as follows: two matrices are said to be of the same conjugacy class type/$z$-class, if the centralizers of two elements are conjugate. However, a further weaker version is enough for our purpose here. We say that two matrices are of same type if their centralizers are isomorphic. This helps us reduce the size of computation (and size of branching matrix) and causes no loss of generality. Throughout, we assume $q$ is odd. We hope our computation throws some light on the subject of commuting probability and will help us understand the groups better.

\subsection*{Acknowledgments} 
The authors would like to thank Amritanshu Prasad, IMSc Chennai, for his interest in this work.

%

\section{Branching rules for $GT_2(\Fq)$}\label{SGT2}
There are four conjugacy class types in $GL_2(\Fq)$ given by the following partitions (as in~\cite{Sh1}) $(1,1)_2,~(2)_1,~(1)_1(1)_1,~\text{and\ }(1)_2$. We use this to get the same for $GT_2(\Fq)$. Since we are looking at $GT_2(\Fq)$, the last one, $(1)_2$ doesn't exist in $GT_2(\Fq)$. In this paper, we shall not use the partition based nomenclature for the conjugacy class types. Instead we use alphanumeric nomenclature as follows (similar to the pattern in~\cite{SS}).

\begin{center}
\begin{tabular}{|c|c|c|c|} \hline 
Canonical Form & No. of Classes & Centralizer & Name of Class Type\\ \hline
 $\begin{matrix}\begin{psmallmatrix} a & 0 \\ 0 & a \end{psmallmatrix},\\~a\in \Fq^*\end{matrix}$ & $q-1$ &$GT_2(\Fq)$& $C$ \\ \hline
 $\begin{matrix}\begin{psmallmatrix} a & 1 \\ 0 & a \end{psmallmatrix},\\~a\in \Fq^*.\end{matrix}$ & $q-1$ &$\left\{\begin{psmallmatrix}x_0&x_1\\&x_0\end{psmallmatrix} \mid x_0\in \Fq^*\right\}$ & $R_1$\\ \hline
 $\begin{matrix}\begin{psmallmatrix} a & 0 \\ 0 & b \end{psmallmatrix},\\~a,b\in \Fq^*,~a\neq b\end{matrix}$ & $(q-1)(q-2)$ &$\left\{\begin{psmallmatrix}x_0&\\&z_0\end{psmallmatrix} \mid x_0,z_0 \in \Fq^* \right\}$ & $R_2$\\ \hline
\end{tabular}
\end{center}
\vskip2mm

\begin{theorem}\label{TheoremGT2}
The branching rules are summarized in the table below given by the branching matrix:
$$B_{GT_{2}(\Fq)} = \begin{pmatrix} q-1 & 0 & 0 \\ q-1 & q(q-1)&0 \\ (q-1)(q-2)& 0 & (q-1)^2
\end{pmatrix}.$$
\end{theorem}

We mention the branching rules below.
\begin{prop}
For an upper triangular matrix of type $C$, the branching rules are as mentioned in the table above.
\end{prop}
\begin{proof}
 The result follows, as this type is central.
\end{proof}

\begin{prop}
For matrices of any of the two regular types:
\begin{itemize}
\item A matrix of type $R_1$ has $q(q-1)$ branches of type $R_1$, and
\item A matrix of type $R_2$ has $(q-1)^2$ branches of type $R_2$.
\end{itemize}
\end{prop}
\begin{proof}
The centralizer of a matrix of any of the above mentioned regular types is commutative, hence each element of the centralizer is a branch. \end{proof}

\begin{proof}[Proof of Theorem~\ref{TheoremGT2}]
The branching rules stated in the above propositions, are summarised in the the branching matrix, as mentioned in the statement of the thoerem.
\end{proof}

\section{Branching in $GT_3(q)$}

Now, we compute the branching table for $GT_3(\Fq)$. The table for the conjugacy classes and their types are as follows:

\begin{center}
\begin{tabular}{|c|c|c|c|}\hline
Class Representative & Number of Classes & Centralizer size & Name of  Type\\ \hline
$aI_3$, $a \neq 0$ & $q-1$ & $(q-1)^3q^3$ & $C$\\ \hline
$\begin{matrix}\begin{psmallmatrix} a& 1&\\&a&\\&& a\end{psmallmatrix}, \begin{psmallmatrix} a& &\\&a&1\\&& a\end{psmallmatrix} ,\\ a\neq 0\end{matrix}$ & $2(q-1)$ &$(q-1)^2q^2$ & $A_1$ \\ \hline
$\begin{psmallmatrix}a& & 1\\ &a& \\ & & a\end{psmallmatrix}, a \neq 0$ & $q-1$ & $(q-1)^2q^3$ & $A_2$ \\ \hline
$\begin{matrix}\begin{psmallmatrix}a& & \\ &a& \\ & & b\end{psmallmatrix},\begin{psmallmatrix}a& & \\ &b& \\ & & a\end{psmallmatrix},\\ \begin{psmallmatrix}b& & \\ &a& \\ & & a\end{psmallmatrix}, 0\neq a\neq b\neq 0 \end{matrix}$ & $3(q-1)(q-2)$ & $(q-1)^3q$ & $B_1$ \\ \hline
$\begin{psmallmatrix} a& 1&\\&a&1\\&& a\end{psmallmatrix}, a \neq 0$ & $q-1$ & $(q-1)q^2$ & $R_1$ \\ \hline  
$\begin{matrix}\begin{psmallmatrix}a&1 & \\ &a& \\ & & b\end{psmallmatrix},\begin{psmallmatrix}a& &1 \\ &b& \\ & & a\end{psmallmatrix},\\ \begin{psmallmatrix}b& & \\ &a& 1\\ & & a\end{psmallmatrix}, a\neq b \end{matrix}$ & $3(q-1)(q-2)$ &$(q-1)^2q$ & $R_2$ \\ \hline
$\begin{matrix}\begin{psmallmatrix} a& &\\&b&\\&& c\end{psmallmatrix}\\ a\neq b \neq c \neq a \end{matrix}$  & $(q-1)(q-2)(q-3)$ & $(q-1)^3$ & $R_3$ \\ \hline 
\end{tabular}
\end{center}
\vskip2mm
The branching rules are described by the branching matrix as follows. 
\begin{theorem}\label{TheoremGT3}
The branching matrix for the group $GT_3(\Fq)$ with types written in the order $\{C, A_1, A_2, B_1, R_1, R_2, R_3\}$ is $B_{GT_3(\Fq)}$
$$= \begin{psmallmatrix}
     q-1 & 0 & 0 & 0 & 0 & 0 & 0  \\
     2(q-1) & q(q-1) & 0 & 0 & 0 &0 & 0 \\
     q-1 & 0 & q(q-1) & 0 & 0&0 &0\\
     3(q-1)(q-2) & 0 & 0 & (q-1)^2 & 0&0 &0\\
     q-1 & q(q-1) & q^2-1 & 0 & (q-1)q^2 & 0 &0\\
     3(q-1)(q-2)& q(q-1)(q-2) & q(q-1)(q-2) & (q-1)^2 & 0 & (q-1)^2q& 0\\
     (q-1)(q-2)(q-3) & 0 & 0 & (q-1)^2(q-2) &0 & 0&(q-1)^3
    \end{psmallmatrix}.
$$
\end{theorem}
 
\begin{prop}\label{T3C}
 For an upper triangular matrix of type $C$, the branches are as in the second column of the table in the the opening paragraph of this section.
\end{prop}
\begin{proof}
 The result follows, since the matrices of type $C$ are central.
\end{proof}

\begin{prop}\label{T3A1}
 An upper triangular matrix of type $A_1$ has $q(q-1)$ branches of type $A_1$, $q(q-1)$ branches of type $R_1$, and $q(q-1)(q-2)$ branches of type $R_2$.
\end{prop}
\begin{proof}
 Let $A = \begin{pmatrix}a&1&\\&a&\\&&a\end{pmatrix}$, a matrix of type $A_1$. The centralizer of $A$ is: $\ZT{3}(A) = \left\{\begin{pmatrix}x_0&x_1&x_2\\&x_0&\\&&z_0\end{pmatrix}\mid x_0, z_0\neq 0 \right\}$. Let $X = \begin{pmatrix}x_0&x_1&x_2\\&x_0&\\&&z_0\end{pmatrix}$, be an invertible member of $\ZT{3}(A)$. Let $B = \begin{pmatrix}a_0&a_1&a_2\\&a_0&\\&&c_0\end{pmatrix}$, and $B' = \begin{pmatrix}a_0&a'_1&a'_2\\&a_0&\\&&c_0\end{pmatrix} = XBX^{-1}$. Thus equating $XB = B'X$ leads us to the following equations:
 
 \begin{eqnarray}
 a'_1 &=& a_1\\
 x_0a_2 + x_2c_0 &=& x_2a_0 + z_0a'_2 \label{ET3A1}
 \end{eqnarray}
 
\noindent{\bf Case: $a_0 = c_0$}. Here, equation~\ref{ET3A1} becomes $x_0a_2 = z_0a'_2$. When $a_2 = 0$, then, we have $B$ reduced to $\begin{pmatrix}a_0&a_1&\\&a_0&\\&&a_0\end{pmatrix}$, with $\ZT{3}(A,B) = \ZT{3}(A)$. Thus $(A,B)$ is of type $A_1$, and there are $q(q-1)$ such branches.

When $a_2 \neq 0$, choose $z_0$ so that $a_2 = 1$. Then $B$ is reduced to $\begin{pmatrix}a_0&a_1&1\\&a_0&\\&&a_0\end{pmatrix}$, and $\ZT{3}(A,B) =\left\{\begin{pmatrix}x_0&x_1&x_2\\&x_0&\\&&x_0\end{pmatrix} \right\}$. This subgroup is commutative. Thus $(A,B)$ is of type $R_1$, and there are $q(q-1)$ such branches. There are no further cases to see here.
\noindent{\bf Case: $a_0 \neq c_0$}. In Equation~\ref{ET3A1}, choose $x_2$ so that $a'_2 = 0$. Thus, $B$ is reduced to $\begin{pmatrix}a_0&a_1&\\&a_0&\\&&c_0\end{pmatrix}$, and $\ZT{3}(A,B) =\left\{\begin{pmatrix}x_0&x_1&\\&x_0&\\&&z_0\end{pmatrix} \right\}$. This subgroup is commutative. Thus $(A,B)$ is of type $R_2$, and there are $q^2(q-1) = q^3-q^2$ such branches.

These are all the cases here. Thus, we have a total of $q^2 + q^3-q^2 = q^3$ branches of type $R$.
\end{proof}
\begin{prop}\label{T3A2}
 An upper triangular matrix of type $A_2$ has $q(q-1)$ branches of type $A_2$, and $q^2-1$ branches of type $R_1$, and $q(q-1)(q-2)$ branches of type $R_2$.
\end{prop}
\begin{proof}
 Given $A = \begin{pmatrix}a& & 1\\ &a&\\&&a\end{pmatrix}$, the canonical form of a matrix of type $A_2$. The centralizer of $A$, $\ZT{3}(A)$ is $\left\{ \begin{pmatrix}x_0&x_1&x_2\\&y_0&y_1\\&&x_0\end{pmatrix}\mid x_0,y_0\neq 0\right\}$. Let $X =  \begin{pmatrix}x_0&x_1&x_2\\&y_0&y_1\\&&x_0\end{pmatrix}\in \ZT{3}(A)$. Let $B = \begin{pmatrix}a_0&a_1&a_2\\&b_0&b_1\\&&a_0\end{pmatrix}$, and $B' = \begin{pmatrix}a_0&a'_1&a'_2\\&b_0&b'_1\\&&a_0\end{pmatrix}= XBX^{-1}$. Thus equating $XB = B'X$ gives us the following equations:
 
  \begin{eqnarray}
  x_0a_1+x_1b_0&=& x_1a_0 +y_0a'_1\\
  y_0b_1+y_1a_0&=& x_0b'_1 + y_1b_0\\
  x_0a_2+x_1b_1&=& x_0a'_2 + y_1a'_1
  \end{eqnarray}
 
Using these we reduce $B$ to the mentioned branches. 
\end{proof}
\begin{prop}\label{T3B1}
 An upper triangular matrix of type $B_1$ has $(q-1)^2$ branches of type $B_1$, and $(q-1)^2$ branches of type $R_2$, and $(q-1)^2(q-2)$ branches of type $R_3$.
\end{prop}
\begin{proof}
 One of the canonical forms of an upper triangular matrix of type $B_1$ is $A = \begin{pmatrix}aI_2 & \\ & b\end{pmatrix}$, where $a \neq b \in \Fq^*$. Hence the centralizer of $A$ is $$\ZT{3}(A)  = \left\{\begin{pmatrix}X& \\ & z_0\end{pmatrix} \mid X\in \Tq{2}, z_0 \neq 0 \right\}.$$ Thus the branches of $A$ are of the form $\begin{pmatrix}C & \\ & d\end{pmatrix}$, where $d\neq 0$, and $C$ is a conjugacy class of $\Tq{2}$. Hence, the result.
\end{proof}
\begin{prop}\label{T3R}
For matrices of the $\Reg$ types:
\begin{itemize}
\item A matrix of type $R_1$ has $(q-1)q^2$ branches of type $R_1$.
\item For type $R_2$, there are $(q-1)^2q$ branches of type $R_2$.
\item For type $R_3$, there are $(q-1)^3$ branches of type $R_3$
\end{itemize}
\end{prop}
\begin{proof}
 The result follows, as the centralizers of matrices of any of the $\Reg$ types are commutative.
\end{proof}
\begin{proof}[Proof of Theorem~\ref{TheoremGT3}]
 From the data in Propositions~\ref{T3C} to~\ref{T3R}, the branching rules are summarized to the branching table/matrix described in the statement of the theorem.
\end{proof}

\section{Branching for $GT_4(q)$}
In this section, we discuss the simultaneous conjugacy classes of tuples of commuting matrices of $GT_4(\Fq)$. The conjugacy classes of $GT_4(\Fq)$ is described in Appendix~\ref{CCGT4}. The branching rules are as follows (types written in the order listed in last column of Appendix~\ref{CCGT4}):
\begin{theorem}\label{TheoremGT4}
The branching matrix for $GT_4(\Fq)$ is of size $28$ ($22$ types of $GT_4(\Fq)$ and $6$ new types), which we write as $B_{GT_4(\Fq)} = \left(\mathcal A\mid \mathcal B \mid \mathcal C \right)$ (split in three parts along the columns for convenience of writing) described in Table~\ref{TableA}, ~\ref{TableB} and~\ref{TableC}.
\end{theorem}
\noindent For the convenience, the branching of non-regular types are in part $\mathcal A$, those of regular types in part $\mathcal B$, and those of the new types in part $\mathcal C$. In each of the sub-tables, the regular branches are in blue, and the new types in red. The $0_{r,s}$ denotes the zero matrix of size $r\times s$. Rest of the section is devoted to proof of this.

\begin{landscape}
\begin{table}
\caption{The matrix $\mathcal A$}
\label{TableA}
\small
\begin{center}
$$ \begin{psmallmatrix}C & A_1& A'_1 & A_2 &A_3&A_4&A_5&A_6&A_7&A_8&A_9&B_1&B_2&B_3&B_4&B_5&B_6\\
q - 1 & 0 & 0 & 0 & 0 & 0 & 0 & 0 & 0 & 0 & 0 & 0 & 0 & 0 & 0 & 0 & 0  \\ 
2 q - 2 & q^{2} - q & 0 & 0 & 0 & 0 & 0 & 0 & 0 & 0 & 0 & 0 & 0 & 0 & 0 & 0 & 0 \\ 
q - 1 & 0 & q^{2} - q & 0 & 0 & 0 & 0 & 0 & 0 & 0 & 0 & 0 & 0 & 0 & 0 & 0 & 0 \\ 
2 q - 2 & 0 & 0 & q^{2} - q & 0 & 0 & 0 & 0 & 0 & 0 & 0 & 0 & 0 & 0 & 0 & 0 & 0 \\ 
q - 1 & 0 & 0 & 0 & q^{2} - q & 0 & 0 & 0 & 0 & 0 & 0 & 0 & 0 & 0 & 0 & 0 & 0 \\ 
q - 1 & 0 & 0 & 0 & 0 & q^{2} - q & 0 & 0 & 0 & 0 & 0 & 0 & 0 & 0 & 0 & 0 & 0 \\ 
q - 1 & 0 & q^{2} - q & q^{2} - q & q^{2} - q & 0 & q^{3} - q^{2} & q(q-1)^2 & 0 & 0 & 0 & 0 & 0 & 0 & 0 & 0 & 0 \\ 
q - 1 & 0 & 0 & 0 & 0 & 0 & 0 & q^{2} - q & 0 & 0 & 0 & 0 & 0 & 0 & 0 & 0 & 0 \\ 
2 q - 2 & q^{2} - q & 0 & q^{2} - q & 0 & 0 & 0 & 0 & q^{3} - q^{2} & 0 & 0 & 0 & 0 & 0 & 0 & 0 & 0 \\ 
q - 1 & q^{2} - q & 0 & 0 & 2 q^{2} - 2 q & 0 & 0 & 0 & 0 & q^{3} - q^{2} & 0 & 0 & 0 & 0 & 0 & 0 & 0 \\ 
q - 1 & 0 & 0 & 0 & q^{2} - q & 0 & 0 & 0 & 0 & 0 & q^{3} - q^{2} & 0 & 0 & 0 & 0 & 0 & 0 \\ 
\begin{smallmatrix}(3q-3).\\ (q-2)\end{smallmatrix} & 0 & 0 & 0 & 0 & 0 & 0 & 0 & 0 & 0 & 0 & (q-1)^2 & 0 & 0 & 0 & 0 & 0 \\ 
\begin{smallmatrix}(4q-4).\\ (q-2)\end{smallmatrix} & 0 & 0 & 0 & 0 & 0 & 0 & 0 & 0 & 0 & 0 & 0 & (q-1)^2 & 0 & 0 & 0 & 0 \\ 
\begin{smallmatrix}(8q-8).\\ (q-2)\end{smallmatrix} & \begin{smallmatrix}2(q^2-q).\\ (q-2)\end{smallmatrix} & \begin{smallmatrix}2(q^2-q).\\ (q-2)\end{smallmatrix} & \begin{smallmatrix}(q^2-q).\\ (q-2)\end{smallmatrix} & 0 & 0 & 0 & 0 & 0 & 0 & 0 & 0 & (q-1)^2 & q(q-1)^2 & 0 & 0 & 0 \\ 
\begin{smallmatrix}(4q-4).\\ (q-2)\end{smallmatrix} & 0 & 0 & \begin{smallmatrix}(q^2-q).\\ (q-2)\end{smallmatrix} & \begin{smallmatrix}2(q^2-q).\\ (q-2)\end{smallmatrix} & 0 & 0 & 0 & 0 & 0 & 0 & 0 & (q-1)^2& 0 & q(q-1)^2 & 0 & 0 \\ 
\begin{smallmatrix}(6q-6).\\ (q-2)\end{smallmatrix} & \begin{smallmatrix}(q^2-q).\\ (q-2)\end{smallmatrix}  &  \begin{smallmatrix}(q^2-q).\\ (q-2)\end{smallmatrix} &  \begin{smallmatrix}(q^2-q).\\ (q-2)\end{smallmatrix} &  \begin{smallmatrix}(q^2-q).\\ (q-2)\end{smallmatrix} & 0 & 0 & 0 & 0 & 0 & 0 & 2(q-1)^2 & 0 & 0 & 0 &  \begin{smallmatrix}(q^2-q).\\ (q-2)\end{smallmatrix} & 0 \\ 
 \begin{smallmatrix}(6q-6).\\ (q-2).\\(q-3)\end{smallmatrix} & 0 & 0 & 0 & 0 & 0 & 0 & 0 & 0 & 0 & 0 & \begin{smallmatrix}2 q^{3} - 8 q^{2} +\\ 10 q - 4 \end{smallmatrix}& \begin{smallmatrix}q^{3} - 4 q^{2} +\\ 5 q - 2\end{smallmatrix} & 0 & 0 & 0 & (q-1)^3 \\ 
\color{blue}q - 1 & \color{blue}q^{2} - q & 0 & \color{blue}q^{2} - q & \color{blue}q^{2} - 1 & \color{blue}q^{3} - q^{2} & 0 &\color{blue} q^{3} - q^{2} &\color{blue} q^{3} - q^{2} & \color{blue}q^{3} - q & \color{blue}\begin{smallmatrix}q^{4} - q^{3} -\\ q^{2} + q\end{smallmatrix} & 0 & 0 & 0 & 0 & 0 & 0 \\ 
\color{blue} \begin{smallmatrix}(4q^-4).\\ (q-2)\end{smallmatrix} & \color{blue} \begin{smallmatrix}(2q-2).\\ (q-2)\end{smallmatrix} &\color{blue}  \begin{smallmatrix}(2q-2).\\ (q-2)\end{smallmatrix} &\color{blue}  \begin{smallmatrix}(3q-3).\\ (q-2)\end{smallmatrix} & \color{blue}\begin{smallmatrix}2 q^{3} - 4 q^{2} - \\2 q + 4\end{smallmatrix} & 0 & 0 & 0 & \color{blue}\begin{smallmatrix}(q^3-q^2).\\ (q-2)\end{smallmatrix} &\color{blue}\begin{smallmatrix}(q^3-q^2).\\ (q-2)\end{smallmatrix} & 0 & 0 & \color{blue}(q-1)^2 & \color{blue}q(q-1)^2 \color{blue}& \color{blue}\begin{smallmatrix}q^{3} - q^{2} - \\q + 1\end{smallmatrix} & 0 & 0 \\ 
\color{blue}  \begin{smallmatrix}(3q-3).\\ (q-2)\end{smallmatrix} & \color{blue}\begin{smallmatrix}(q^2-q).\\ (q-2)\end{smallmatrix}  & \color{blue}\begin{smallmatrix}(q^2-q).\\ (q-2)\end{smallmatrix} & \color{blue}\begin{smallmatrix}(q^2-q).\\ (q-2)\end{smallmatrix} &\color{blue} \begin{smallmatrix}(q^2-q).\\ (q-2)\end{smallmatrix} & \color{blue}\begin{smallmatrix}(q^3-q^2).\\ (q-2)\end{smallmatrix} & \color{blue}\begin{smallmatrix}(q^3-q^2).\\ (q-2)\end{smallmatrix} & \color{blue}\begin{smallmatrix}(q^3-q^2).\\ (q-2)\end{smallmatrix} & 0 & 0 & 0 & \color{blue}(q-1)^2 & 0 & 0 & 0 & \color{blue}q(q-1)^2 & 0 \\ 
\color{blue}\begin{smallmatrix}(6q-6).\\ (q-2).\\(q-3)\end{smallmatrix} &\color{blue} \begin{smallmatrix}q^{4} - 6 q^{3} +\\11 q^{2} - 6 q\end{smallmatrix}  & \color{blue}\begin{smallmatrix}q^{4} - 6 q^{3} +\\ 11 q^{2} - 6 q\end{smallmatrix} &\color{blue} \begin{smallmatrix}q^{4} - 6 q^{3} +\\ 11 q^{2} - 6 q \end{smallmatrix}&\color{blue}\begin{smallmatrix} q^{4} - 6 q^{3} + \\11 q^{2} - 6 q\end{smallmatrix} & 0 & 0 & 0 & 0 & 0 & 0 & \color{blue}\begin{smallmatrix}2 q^{3} - 8 q^{2} +\\ 10 q - 4\end{smallmatrix} &\color{blue} \begin{smallmatrix}q^{3} - 4 q^{2} +\\ 5 q - 2\end{smallmatrix} \color{blue}& \color{blue}\begin{smallmatrix}q^{4} - 4 q^{3} + \\5 q^{2} - 2 q\end{smallmatrix} &\color{blue} \begin{smallmatrix}q^{4} - 4 q^{3} +\\ 5 q^{2} - 2 q\end{smallmatrix} & \color{blue}\begin{smallmatrix}q^{4} - 4 q^{3} +\\ 5 q^{2} - 2 q\end{smallmatrix} &\color{blue} \begin{smallmatrix}q^{3} - 3 q^{2} +\\ 3 q - 1\end{smallmatrix} \\ 
\color{blue}4!{q-1\choose 4} & 0 & 0 & 0 & 0 & 0 & 0 & 0 & 0 & 0 & 0 & \color{blue}\begin{smallmatrix}q^{4} - 6q^{3}\\ + 13q^{2} - 12q\\ + 4\end{smallmatrix} &\color{blue}\begin{smallmatrix}q^{4} - 7q^{3} + \\17q^{2} - 17q\\ + 6\end{smallmatrix} & 0 & 0 & 0 &\color{blue} \begin{smallmatrix}q^{4} - 5 q^{3} +\\ 9 q^{2} - 7 q\\ + 2\end{smallmatrix} \\ 
 0 & \color{red}q^{2} - q & 0 & 0 & 0 & \color{red}q(q-1)^2 & 0 & 0 & 0 & 0 & 0 & 0 & 0 & 0 & 0 & 0 & 0 \\ 
0 & 0 & \color{red}2 q^{2} - 2 q & \color{red}q - 1 & 0 & 0 & 0 & 0 & 0 & 0 & 0 & 0 & 0 & 0 & 0 & 0 & 0 \\ 
0 & 0 & 0 & \color{red}q^{2} - q & \color{red}2 q - 2 & 0 & 0 & 0 & 0 & 0 & 0 & 0 & 0 & 0 & 0 & 0 & 0 \\ 
0 & 0 & 0 & 0 & \color{red}q - 1 & 0 & 0 & \color{red}q^{3} - q^{2} & 0 & 0 & 0 & 0 & 0 & 0 & 0 & 0 & 0 \\ 
0 & 0 & 0 & 0 & 0 & \color{red}q^{2} - q & 0 & 0 & 0 & 0 & 0 & 0 & 0 & 0 & 0 & 0 & 0 \\ 
0 & 0 & \color{red}q^{3} + q^{2} - 2 q &\color{red} q^{2} - 1 & 0 & 0 & \color{red}q^{4} - q^{2} & \color{red}q^{3} - q^{2} & 0 & 0 & 0 & 0 & 0 & 0 & 0 & 0 & 0 
    \end{psmallmatrix}
$$    
\end{center}
\end{table}
\end{landscape}

\begin{table}
\caption{The matrix $\mathcal B$}
\label{TableB}
$$\mathcal B=
{\color{blue} \begin{pmatrix}
R_1&R_2&R_3&R_4&R_5 \\ &&&&\\
 &  & 0_{17\times 5} &  &  \\ &&&&\\
q^{4} - q^{3} & 0 & 0 & 0 & 0 \\
0 & q^{4} - 2 q^{3} + q^{2} & 0 & 0 & 0 \\
0 & 0 & q^{4} - 2 q^{3} + q^{2} & 0 & 0 \\
0 & 0 & 0 & q^{4} - 3 q^{3} + 3 q^{2} - q & 0 \\
0 & 0 & 0 & 0 & (q-1)^4 \\ &&&&\\
 &  & 0_{6\times 5} &  & 
\end{pmatrix}}$$ 
\end{table}
\begin{table}
\caption{The matrix $\mathcal C$}
\label{TableC}
\begin{Small}
$$  \mathcal C = \color{red}\begin{pmatrix}
tNT_1& tNT_2 &tNT_3 & tNT_4 &tNT_5 &NR_1\\
 &  &  &  &  &  \\
 &  & 0_{17 \times 6} &  &  &  \\
 &  &  &  &  &  \\
\color{blue}q^{3} - q^{2} & 0 & \color{blue}q^{3} - q^{2} & \color{blue}q^{4} - 2 q^{3} + q^{2} & \color{blue}q^{4} - q^{3} - q^{2} + q & 0 \\
0 & \color{blue}q^{4} - 3 q^{3} + 2 q^{2} & \color{blue}q^{4} - 3 q^{3} + 2 q^{2} & 0 & 0 & 0 \\
\color{blue}q^{4} - 3 q^{3} + 2 q^{2} & 0 & 0 & 0 & 0 & 0 \\
0 & 0 & 0 & 0 & 0 & 0 \\
0 & 0 & 0 & 0 & 0 & 0 \\
q^{3} - q^{2} & 0 & 0 & 0 & 0 & 0 \\
0 & q^{3} - q^{2} & 0 & 0 & 0 & 0 \\
0 & 0 & q^{3} - q^{2} & 0 & 0 & 0 \\
0 & 0 & 0 & q^{3} - q^{2} & 0 & 0 \\
0 & 0 & 0 & 0 & q^{3} - q & 0 \\
0 & q^{4} - q^{2} & q^{3} - q & q^{4} - q^{3} - q^{2} + q & 0 & q^{5} - q^{4}\end{pmatrix}
$$\end{Small}
\end{table}

The first column of $A$ corresponds to the central type $C$ and the entries in the column are number of classes of each type in $GT_4(\Fq)$ which is the column two of table in Appendix~\ref{CCGT4}. For all the regular types $R_1, R_2, R_3, R_4$ and $R_5$, the only branch is that type itself, and the the number of branches is the size of its centralizer which is again listed in Appendix~\ref{CCGT4}. This fully describes the matrix $\mathcal B$. Thus, it only remain to explain the matrix $\mathcal A$ and $\mathcal C$.


\subsection{Branching rules for type $A$}
Let us deal with type $A$ classes as in Section~\ref{CCGT4}.
\begin{prop}\label{BrT4A1}
The branching rules of a matrix of type $A_1$ are:
\begin{center}
\begin{tabular}{c|c||c|c}\hline
Branch & No. Of Branches & Branch & No. of Branches \\ 
\hline
$A_1$ & $q(q-1)$ & $R_1$ & $(q-1)q$\\
$A_7$ & $q(q-1)$ & $R_2$ & $2q(q-1)(q-2)$ \\
$A_8$ & $q(q-1)$ & $R_3$ & $q(q-1)(q-2)$ \\
$B_3$ & $2q(q-1)(q-2)$ & $R_4$ & $q(q-1)(q-2)(q-3)$\\ 
$B_5$ & $q(q-1)(q-2)$ &$tNT_1$ & $q(q-1)$. \\ \hline
 \end{tabular}
\end{center}
Here, a new type appears, called $tNT_1$, whose centralizer is
 $\left\{ 
 \begin{psmallmatrix}x_0&x_1&&x_3\\&x_0&&\\&&z_0&z_1\\&&&z_0\end{psmallmatrix}\mid x_0z_0\neq 0 \right\}$.
\end{prop}
\begin{proof}
 A matrix of type $A_1$ has either of the canonical forms $ %
$. We may consider any one of them. WLOG, we take $A =  %
$. Equating $XB= B'X$ leads us to the following:
\begin{eqnarray*}
 a'_1 &=& a_1\\
%
$. The second equation leads us to various conjugacy classes of $\Tq{2}$. Hence, we take $C$ to be some conjugacy class representative in $\Tq{2}$, and $Z\in \ZT{2}(C)$. This leads us to the following equation:
\begin{equation}\label{E4A1}
 x_0%
.Z
\end{equation}

When $a_0$ is an eigenvalue of $C$: 

When $(a_2,a_3)= (0,0)$: Equation~\ref{E4A1} becomes $%
.(C-a_0I_2) = (0~~0)$.

When $C = a_0I_2$, Equation~\ref{E4A1} is void, and we have $B$ reduced to $%
$, and $\ZT{4}(A,B) = \ZT{4}(A)$. Thus $(A,B)$ is of type $A_1$, and there are $q(q-1)$ such branches.

When $C = %
\right\}$. This centralizer is not isomorphic to the centralizers of te known types. Thus $(A,B)$ is of a new type, which we will call $tNT_1$ and there are $q(q-1)$ such branches.

When $C = %
\right\}$. Thus $(A,B)$ is of type $B_3$, and there are $q(q-1)(q-2)$ such branches.

When $C = %
\right\}$. Thus $(A,B)$ is of type $B_3$, and there are $q(q-1)(q-2)$ such branches.

When $(a_2,a_3)\neq (0,0)$: 

When $C = a_0I_2$, Equation~\ref{E4A1} becomes:
$%
$.

We have from this:
\begin{eqnarray}
a_2&=& \frac{z_0}{x_0}a'_2\\
a_3 &=& \frac{z_1}{z_0}a'_2 + \frac{z_2}{x_0}a'_3
\end{eqnarray}
When $a_2\neq 0$, choose $x_0$ such that $a'_2 = 1$. In the equation below, choose $z_1$ so that $a'_3 = 0$. So, $B$ reduces to 
$%
 \right\}$.  This $(A,B)$ is of type $A_7$, and there are $q(q-1)$ such branches.

When $a_2 =0$, $a_3\neq 0$, choose $z_2$ such that $a'_3 =1 $. Thus $B$ is reduced to $%
 \right\}$.  This $(A,B)$ is of type $A_8$, and there are $q(q-1)$ such branches.

When $C = %
.
$$
We have:
\begin{eqnarray*}
a_2 &=& a'_2\frac{z_0}{x_0}\\
a_3  + \frac{x_2}{x_0} &=& \frac{z_1}{z_0}a'_2 + \frac{z_0}{x_0}a'_3.
\end{eqnarray*}
Choose $x_2$ such that $a'_3 = 0$. As $(a_2,a_3)\neq (0,0)$, and $a_3 = 0$, we have $a_2\neq 0$. Choose $z_0$ such that $a'_2 = 1$. So, $B$  is reduced to 
$%
 \right\}$.  This $(A,B)$ is of type $R_1$, and there are $q(q-1)$ such branches.

When $C = %
.
$$
We have:
\begin{eqnarray*}
a_2 &=& a'_2\frac{z_0}{x_0}\\
a_3  + \frac{x_3}{x_0}(b_0-a_0) &=& \frac{z_0}{x_0}a'_3.
\end{eqnarray*}
As $b_0-a_0\neq 0$, choose $x_3$ so that $a'_3 = 0$. So we are left with $a_2 \neq 0$. Choose $z_0$ such that $a'_2 = 1$. So $B$ is reduced to 
$%
 \right\}$. Thus, $(A,B)$ is of type $R_2$, and there are $q(q-1)(q-2)$ such branches.

When $C = %
.
$$
We have:
\begin{eqnarray*}
a_2 + \frac{x_2}{x_0}(b_0-a_0)&=& a'_2\frac{z_0}{x_0}\\
a_3   &=& \frac{z_0}{x_0}a'_3.
\end{eqnarray*}
As $b_0-a_0\neq 0$, choose $x_2$ so that $a'_2 = 0$. So we are left with $a_3 \neq 0$. Choose $z_2$ such that $a'_3 = 1$. So $B$ is reduced to 
$%
 \right\}$. Thus, $(A,B)$ is of type $R_2$, and there are $q(q-1)(q-2)$ such branches.

Now, we come to the case of $a_0$ not being an eigenvalue of $C$. In Equation~\ref{E4A1}, the matrix $(C-a_0I_2)$ is invertible. So, we can choose $x_2,x_3$ such that both $a_2 = a_3 =0$. Thus, on replacing $a_2$ and $a_3$ with 0 each in Equation~\ref{E4A1}, we get  
$%
$. So $x_2 = x_3 = 0$. Thus, we have:

When $C = b_0I_2$, $b_0\neq a_0$, $B$ is reduced to 
$%
\right\}$. $(A,B)$ is of the type $B_5$, and there are $q(q-1)(q-2)$ such branches.

When $C = %
$, $b_0\neq a_0$, $B$ is reduced to 
$%
\right\}$. $(A,B)$ is of the type $R_3$, and there are $q(q-1)(q-2)$ such branches.

When $C = %
$, $b_0\neq a_0\neq c_0\neq b_0$, $B$ is reduced to 
$%
\right\}$. $(A,B)$ is of the type $R_4$, and there are $q(q-1)(q-2)(q-3)$ such branches. 

We are left with no other cases.

\end{proof}
\begin{prop}\label{BrT4A1P}
The branching rules of a matrix of type $A'_1$ are:
\begin{center}
\begin{tabular}{c|c||c|c}\hline
   Branch & No. of Branches & Branch & No. of Branches\\ \hline
   $A'_1$& $q(q-1)$ & $R_2$ &$2q(q-1)(q-2)$\\
   $A_5$ & $q(q-1)$ & $R_3$ & $q(q-1)(q-2)$\\
   $B_3$ & $2q(q-1)(q-2)$ & $R_4$ & $q(q-1)(q-2)(q-3)$\\ 
   $B_5$ & $q(q-1)(q-2)$ & $NR_1$ & $q(q-1)(q+2)$\\
   $tNT_2$ & $2q(q-1)$ & & \\ \hline 
\end{tabular}
\end{center}
Two new types $NR_1$ and $tNT_2$ appear here. The centralizers of these new types are $\left \{ \begin{psmallmatrix} x_0 I_2 & Y \\ & x_0 I_2  \end{psmallmatrix} \mid Y \in M_2(\Fq), x_0 \neq 0 \right\}$, and $\left\{  \begin{psmallmatrix}x_0& &x_2&x_3\\&x_0&y_1&y_2 \\&& x_0 & \\ &&&y_0 \end{psmallmatrix}\mid x_0y_0 \neq 0 \right\}$, respectively.
\end{prop}
\begin{proof}
 A matrix of type $A'_1$ has the canonical form: $A = %
$. Denote the submatrix $%
$ of $B$ by $C$, and the submatrix $%
$ by $Z$. Then equating $XB = B'X$ leads us to $ZC = C'Z$. Thus, we can take $C$ to be a canonical form in $\Tq{2}$, and $Z \in \ZT{2}(C)$. Thus we have $b'_1 = b_1$, and the following equaitons:
 \begin{eqnarray}
  x_0a_2 + x_2b_0 &=& x_2a_0 + y_0a'_2 \label{E4AP1}\\
  y_0b_2 + y_2c_0 &=& y_2b_0 + z_0b'_2 \label{E4AP2}
 \end{eqnarray}
 
\noindent{\bfseries When $C$ has $b_0$ as an eigenvalue}:

When $(a_2,b_2) = (0,0)$:

When $C = b_0I_2$, Equation~\ref{E4AP1} and \ref{E4AP2} become void, and $B$ becomes, $B = %
$, and $\ZT{4}(A,B) = \ZT{4}(A)$. $(A,B)$ is of type $A'_1$, and there are $q(q-1)$ such branches.

When $C = %
$. Here too, Equations~\ref{E4AP1} and \ref{E4AP2} are void. So, $B$ reduces to 
$%
\right\}$. This $(A,B)$ is of type $A_5$, and there are $q(q-1)$ such branches.

When $C = %
$, $b_0\neq c_0$. Here Equation~\ref{E4AP1} stays void, but \ref{E4AP2} becomes $y_2c_0 = y_2b_0$, thus $y_2 = 0$. So, $B$ reduces to 
$%
\right\}$. $(A,B)$ is of type $B_3$, and there are $q(q-1)(q-2)$ such branches.

When $C = %
$, $b_0\neq a_0$. Here Equation~\ref{E4AP1} becomes $x_2b_0 = x_2a_0$, hence $x_2 = 0$, and Equation~\ref{E4AP2} stays void. So, $B$ reduces to 
$%
\right\}$. $(A,B)$ is of type $B_3$, and there are $q^2(q-1) = q^3 -q^2 $ such branches.

When $(a_2, b_2)\neq (0,0)$:

When $C = b_0I_2$, Equations~\ref{E4AP1} and \ref{E4AP2} become 
\begin{eqnarray*}
x_0a_2 &=& y_0a'_2\\
y_0b_2 &=& z_0b'_2
\end{eqnarray*}

When $a_2\neq 0$, choose $y_0$ such that $a'_2 = 1$. When $b_2 \neq 0$, choose $z_0$ such that $b'_2 =1$. So, $B$ is reduced to $%
\right\}$. This centralizer is not of any known type in $\Tq{4}$, and it is clearly a commutative one. We call this new type $NR_1$. There are $q(q-1)$ such branches.

When $b_2 = 0$, then Equation~\ref{E4AP2} is void, and $B$ is reduced to $%
\right\}$. This centralizer too is not of any known type in $\Tq{4}$, and definitely not of $R_1$, as this one is 6-dimensional. We call this new type $tNT_2$. There are $q^2$ such branches.

When $a_2 = 0$, and $b_2 \neq 0$, choose $z_0$ such that $b'_2 = 1$, and $B$ is reduced to $%
\right\}$. We have another $q^2$ branches of this new type $tNT_2$.

When $C = %
$. Equation~\ref{E4AP1} becomes $x_0a_2 = y_0a'_2$, and Equation~\ref{E4AP2} becomes $y_0b_2 = x_0b'_2$. 

When $a_2 \neq 0$, choose $y_0$ such that $a'_2 = 1$. Now, on substituting $a_2$ by $a'_2 = 1$ in the equation, we get $y_0 = x_0$, and thus $b'_2 = b_2$. $B$ is reduced to 
$%
 \right\}$. $(A,B)$ is of type $NR_1$, and there are $q^2(q-1)$ such branches.

When $a_2 = 0$, we look at $b'_2\neq 0$, Choose $x_0$ such that $b'_2 = 1$. Then $B$ is reduced to 
$%
 \right\}$. $(A,B)$ is of type $R_1$, and there are another $q(q-1)$ such branches.

When $C = %
$. Equation~\ref{E4AP1} becomes $x_0a_2 = y_0a'_2$, and Equation~\ref{E4AP2} becomes $y_0b_2 + y_2c_0 = y_2b_0 + z_0b'_2$. As $b_0\neq c_0$, choose $y_2$ such that $b'_2 = 0$. So we have only one case here $a_2 \neq 0$. Thus, choose $x_0$ such that $a'_2 = 1$. Thus $B$ is reduced to $%
\right\}$. $(A,B)$ is of type $R_2$, and there are $q(q-1)(q-2)$ such branches. 

When $C = %
$. Equation~\ref{E4AP1} becomes $x_0a_2 + x_2b_0 = y_0a'_2+x_2a_0$, and Equation~\ref{E4AP2} becomes $y_0b_2 = z_0b'_2$. As $b_0\neq c_0$, choose $x_2$ such that $a'_2 = 0$. So we have only one case here $b_2 \neq 0$. Thus, choose $z_0$ such that $b'_2 = 1$. Thus $B$ is reduced to $%
\right\}$. $(A,B)$ is of type $R_2$, and there are $q(q-1)(q-2)$ such branches.

Now, the second main case of $b_0$ not being an eigenvalue of $C$, i.e., $b_0\neq a_0$ and $b_0\neq c_0$. Here in Equation~\ref{E4AP1}, choose $x_2$ so that $a'_2 = 0$, and in Equation~\ref{E4AP2} choose $y_2$ so that $b'_2 =0$. 

When $C = a_0I_2$, where $a_0\neq b_0$, $B$ is reduced to
$%
\right\}$. $(A,B)$ is of type $B_5$, and there are $q(q-1)(q-2)$ such branches.

When $C = %
$, where $a_0\neq b_0$, $B$ is reduced to
$%
\right\}$. $(A,B)$ is of type $R_3$, and there are $q(q-1)(q-2)$ such branches.

When $C = %
$, where $a_0,c_0\neq b_0$, and $a_0\neq c_0$, $B$ is reduced to
$%
\right\}$. $(A,B)$ is of type $R$, and there are $q(q-1)(q-2)(q-3)$ such branches.

With this there are no other cases left for us to analyse.

Adding up the branches of type $NR_1$, we have a total of $2q(q-1) + q^2(q-1) = q(q-1)(q+2)$ branches.

\end{proof}
\begin{prop}\label{BrT4A2}
 The branching rules of a matrix of type $A_2$ are given below:
  \begin{center}
   \begin{tabular}{c|c||c|c}\hline
    Branch & No. of Branches & Branch & No. of Branches \\ \hline
    $A_2$ & $q(q-1)$ & $R_2$ & $3(q^2-q)(q-2)$\\
    $A_5$ & $q(q-1)$ & $R_3$ & $(q^2-q)(q-2)$ \\
    $A_7$& $q(q-1)$ & $R_4$ & $q(q-1)(q-2)(q-3)$\\ 
    $B_3$ & $(q^2-q)(q-2)$ & $tNT_2$ & $q-1$\\ 
    $B_4$ & $(q^2-q)(q-2)$& $tNT_3$ & $q(q-1)$ \\
    $B_5$ & $(q^2-q)(q-2)$& $NR_1$ & $q^2-1$.\\
    $R_1$ & $q(q-1)$ &&\\\hline
   \end{tabular}
  \end{center}
 A further new type $tNT_3$ appears here, whose centralizer is 
  $\left\{  \begin{psmallmatrix}x_0&x_1&x_2&x_3\\&y_0&y_1&y_2\\&&x_0&\\&&&x_0\end{psmallmatrix} \mid x_0,y_0\neq 0\right\}$. 
 \end{prop}
 \begin{proof}
 Matrices of this type have two non-similar canonical forms:
 $%
$. Proving this for anyone of them is enough. We consider $A = %
$. Equating $XB = B'X$, we have first $ZC = C'Z$. Thus, we may take $C$ to be a canonical form from $\Tq{2}$, and $Z\in \ZT{2}(C)$. With these, we have the following equations:
 \begin{eqnarray}
  x_0.%
.Z \label{ETA21}\\
  y_0b_1 + y_1a_0 &=& x_0b'_1 + y_1b_0 \label{ETA22}\\
  x_0a_2 + x_1b_1 &=& x_0a'_2 + y_1a'_1\label{ETA23}
 \end{eqnarray}
 We have two main cases, under each of which there are subcases: 
 
 \noindent{\bfseries When $a_0$ is an eigenvalue of $C$}
 {\bfseries When $C = a_0I_2$:}, Equation~\ref{ETA21} becomes:
 \begin{equation}\label{ETA24}%
\end{equation}
 Equation~\ref{ETA22} becomes $y_0b_1 = x_0b'_1$.
 
 When $a_1 = b_1 = 0$: From Equation~\ref{ETA24} we have $x_0a_3 = z_0a'_3$, and from Equation~\ref{ETA23} $a'_2 = a_2$
 
 We have two subcases:
 
 When $a_3 = 0$: $B$ is reduced to 
 $
 %
 $, and $\ZT{4}(A,B) = \ZT{4}(A)$. $(A,B)$ is of type $A_2$, and there are $q(q-1)$ such branches.
 
 When $a_3 \neq 0$, choose $z_0$ so that $a'_3 = 1$. $B$ is reduced to 
  $
  %
\right\}$. Now, this centralizer is not isomorphic to any known centralizer of a matrix in $\Tq{4}$, and neither it is isomorphic to those of the three new types we encountered in the previous propositions. We have a new type $tNT_3$, and there are $q(q-1)$ such branches.
  
  When $a_1 = 0$, and $b_1 \neq 0$.  In Equation~\ref{ETA22} choose $y_0$ such that $b'_1 = 1$. Then Equation~\ref{ETA23} becomes $x_0a_2 + x_1 = x_0a'_2 $. Choose $x_1$ such that $a'_2 = 0$.
  
  Here, when $a_3 = 0$, $B$ is reduced to 
    $
    %
\right\}$. $(A,B)$ is of type $tNT_2$, and there are $q-1$ such branches.
    
  When $a_3 \neq 0$, choose $z_0$ such that $a'_3 = 1$. Thus $B$ is reduced to 
    $
    %
\right\}$. $(A,B)$ is of type $NR_1$, and there are $q-1$ such branches.
    
When $a_1\neq 0$, in Equation~\ref{ETA24}, choose $y_0$ such that $a'_1 = 1$. Thus, on replacing $a_1$ with $a'_1 =1$ in Equation~\ref{ETA24}, we get $y_0 = x_0$. In the same equation, choose $y_2 $ such that $a'_3 = 0$.

From Equation~\ref{ETA22}, we get $b'_1 = b_1$. Equation~\ref{ETA23} becomes $x_0a_2 + x_1b_1 = x_0a'_2 + y_1$. Choose $y_1$ such that $a'_2 = 0$. Here, $B$ is reduced to 
  $
  %
\right\}$.
$(A,B)$ is of type $A_7$, and there are $q(q-1)$ such branches.

{\bfseries When $C = %
.
\end{equation}
Choose $x_1$ such that $a'_3 = 0$. Hence, on replacing $a_3$ by $a'_3 = 0$ in the above equation, we have $x_1 = a'_1y_2$.

Equation~\ref{ETA22} becomes $x_0b'_1 = y_0b_1$. 

When $a_1 = b_1 = 0$, from Equation~\ref{ETA23} $a'_2 = a_2$, $B$ is reduced to 
  $
  %
\right\}$. $(A,B)$ is of type $A_5$. There are $q(q-1)$ such branches.
  
  When $a_1 = 0$, and $b_1\neq 0$, we choose $y_0$ in Equation~\ref{ETA22} so that $b'_1 = 1$. So, Equation~\ref{ETA23} becomes $x_0a_2 = x_0a'_2$, since $x_1 = y_2a_1 = 0$. Hence $a'_2 = a_2$. So $B$ is reduced to 
    $
    %
\right\}$. So, we have another $q(q-1)$ branches of type $NR_1$ here.
    
When $a_1 \neq 0$, in Equation~\ref{ETA21}, choose $y_0$ so that $a'_1 = 1$. So, $x_1 = y_2$, and on replacing $a'_1$ with $a_1$ in the same equation, we have $y_0 = x_0$, and hence from Equation~\ref{ETA22}, $b'_1 = b_1$.

With these, Equation~\ref{ETA23} becomes $x_0a_2 + x_1b_1 = x_0a'_2 + y_1$. Choose $y_1$ so that $a'_2 = 0$. Thus $B$ is reduced to 
  $
  %
\right\}$. By a routine check, one can see that this subgroup is commutative. Thus $(A,B)$ is of type $R_1$, and there are $q(q-1)$ such branches.
  
  {\bfseries When $C = %
\end{equation} And, Equation~\ref{ETA22} becomes $y_0b_1= x_0b'_1$.

  Choose $x_3$ such that $a'_3 = 0$.
  
  When $a_1 = b_1 = 0$, from Equation~\ref{ETA23}, we have $a'_2 = a_2$, and $B$ is reduced to 
    $
    %
\right\}$. This $(A,B)$ is of type $B_4$, and there are $q(q-1)(q-2)$ such branches.
    
  When $a_1 = 0$, and $b_1\neq 0$. In Equation~\ref{ETA22}, choose $y_0$ so that $b'_1 = 1$. And, Equation~\ref{ETA23} becomes $x_0a_2 + x_1 = x_0a'_2$. We choose $x_1$ so that $a'_2 = 0$. Hence, $B$ is reduced to 
    $
    %
\right\}$. By a routine check, one can see that this subgroup is commutative. Thus $(A,B)$ is of type $R_2$, and there are $q(q-1)(q-2)$ such branches.
    
When $a_1 \neq 0$, in Equation~\ref{ETA21}, choose $y_0$ so that $a'_1 = 1$. Hence, on replacing $a_1$ by $a'_1 = 1$ on both sides of Equation~\ref{ETA25}, we get $x_0 = y_0$. Hence, Equation~\ref{ETA22} becomes $x_0b_1 = x_0b'_1$, thus leaving use with $b'_1 = b_1$. Equation~\ref{ETA23} becomes $x_0a_2 + x_1b_1 = y_1 + x_0a'_2$. Thus, choose $y_1$ such that $a'_2 = 0$. Hence $B$ is reduced to 
  $
  %
\right\}$. By a routine check, one can see that this subgroup is commutative. Thus $(A,B)$ is of type $R_2$, and there are $q(q-1)(q-2)$ such branches.
  
{\bfseries When $C =%
$. In this case, Equation~\ref{ETA21} is reduced to
$%
$. Choose $x_1$ so that $a'_1 = 0$.

Equation~\ref{ETA22} becomes $y_0b_1 + y_1a_0 = y_1b_0+ x_0b'_1$. As $a_0\neq b_0$, choose $y_1$ such that $b'_1  =0$. With these, Equation~\ref{ETA23} becomes $x_0a_2 = x_0a'_2$, thus leaving us with $a'_2 = a_2$.

Now, we are left to deal with $a_3$. 

When $a_3 = 0$, $B$ is reduced to 
  $
  %
\right\}$. This subgroup is isomorphic to the subgroup $\left\{%
\right\}$, which is the centralizer of a matrix of type $B_3$. Hence, we have $q(q-1)(q-2)$ branches of type $B_3$.
  
  When $a_3 \neq 0$, choose $z_0$ so that $a'_3 = 1$. Hence $B$ is reduced to 
    $
    %
\right\}$. By a routine check, one can see that this subgroup is commutative. Thus $(A,B)$ is of type $R_2$, and there are $q(q-1)(q-2)$ such branches.
With this we have looked at all the cases, when $a_0$ is an eigenvalue of $C$.

\noindent{When $a_0$ is not an eigenvalue oc $C$:} When $a_0$ is not an eigenvalue of $C$, $C-a_0I_2 = %
$, with $b_0-a_0\neq 0$, and $c_0-a_0\neq 0$. Hence, in equation~\ref{ETA21}, we can choose $%
$ such that $a'_1 = 0$, and $a'_3 = 0$. In Equation~\ref{ETA22} choose $y_1$ so that $b'_1 = 0$. Hence Equation~\ref{ETA23} becomes $x_0a_2 = x_0a'_2$. Therefore $a'_2 = a_2$.

On replacing $a_1$ and $a_3$ by 0 in Equation~\ref{ETA21}, we get $x_1 = x_3 = 0$, and on replacing $b_1$ by 0 in Eqaution~\ref{ETA22}, we get $y_1 = 0$. Now, we can look at the various cases of $C$.

{\bfseries When $C = b_0I_2, b_0 \neq a_0$:} $B$ is reduced to 
$
%
 \right\}$. Thus $(A,B)$ is of type $B_5$, and there are $q(q-1)(q-2)$ such branches.

{\bfseries When $C = %
, b_0 \neq a_0$:} $B$ is reduced to 
$
%
 \right\}$. Thus $(A,B)$ is of type $R_3$, and there are $q(q-1)(q-2)$ such branches.

{\bfseries When $C = %
, b_0 \neq c_0, a_0\notin \{b_0,c_0\}$:} $B$ is reduced to 
$
%
 \right\}$. Thus $(A,B)$ is of type $R_4$, and there are $q(q-1)(q-2)(q-3)$ such branches.

Adding up the branches of type $NR_1$, we have a total of $q-1 + q(q-1) = q^2-1$ branches of type $R_1$.
\end{proof}
\begin{prop}\label{BrT4A3}
 For a matrix of type $A_3$, the branching rules are in the table below:
 \begin{center}
 \begin{tabular}{c|c||c|c}\hline
   Branch & No. of Branches & Branch & No. of Branches \\ \hline
   $A_3$ & $q(q-1)$ & $R_1$ & $q^2-1$ \\ 
   $A_5$ & $q(q-1)$ & $R_2$ & $2(q^2-1)(q-2)$ \\
   $A_8$& $2q(q-1)$ & $R_3$ & $(q^2-q)(q-2)$ \\
   $A_9$ & $q(q-1)$ & $R_4$& $(q^2-q)(q-2)(q-3)$\\
   $B_4$ & $2(q^2-q)(q-1)$ &$tNT_3$ & $2(q-1)$\\
   $B_5$ & $(q^2-q)(q-1)$ & $tNT_4$ & $q-1$. \\ \hline
 \end{tabular}
 \end{center}
 A new type $tNT_4$ appears here, whose centralizer is $\left\{  \begin{psmallmatrix}  x_0&x_1&x_2&x_3\\ &x_0&y_1 & y_2\\ & & x_0 & x_1\\ &&&x_0  \end{psmallmatrix} \mid x_0\neq 0\right\}$.
\end{prop}
\begin{proof}
 A matrix of type $A_3$ has the canonical form $A = %
$.
Then we have $\ZT{4}(A) = \left\{%
$. $XB = B'X$ leads to firstly, $ZC = C'Z$, hence we shall take $C$ to be a canonical conjugacy class representative in $\Tq{2}$, and $Z \in \ZT{2}(C)$. Then we have the following set of equations:
\begin{eqnarray}
 x_0{}^t\overrightarrow{b} + {}^t\overrightarrow{y}.(C-a_0I_2) &=& {}^t\overrightarrow{b}'.Z \label{ETA31}\\
 Z.\overrightarrow{d} + (a_0I_2-C)\overrightarrow{w} &=& x_0\overrightarrow{d}'  \label{ETA32}\\
 x_0a_1 + y_1d_1+y_2d_2 &=& x_0a'_1 + b'_1w_1 + b'_2w_2\label{ETA33}
\end{eqnarray}

\noindent{\bfseries When $a_0$ is an eigenvalue of $C$:}

When $\overrightarrow{b} = \overrightarrow{d} = \overrightarrow{0}$: 

Here, Equation~\ref{ETA31} becomes ${}^t\overrightarrow{y}(C-a_0I_2) = {}^t\overrightarrow{0}$, Equation~\ref{ETA32} becomes $(a_0I_2-C)\overrightarrow{w} = \overrightarrow{0}$, and Equation~\ref{ETA33} becomes $x_0a_1 = x_0a'_1$. Hence we have $a'_1 = a_1$.

{\bfseries When $C = a_0I_2$}: Equations~\ref{ETA31} and \ref{ETA32} are void. Hence $B$ is reduced to
$
\begin{pmatrix}
a_0&&&a_1\\
&a_0&&\\
&&a_0&\\
&&&a_0
\end{pmatrix}
$, and $\ZT{4}(A,B) = \ZT{4}(A)$. $(A,B)$ is of type $A_2$, and there are $q(q-1)$ such branches.

{\bfseries When $C = \begin{pmatrix}a_0&1\\&a_0\end{pmatrix}$:} Here Equation~\ref{ETA31} becomes: $\begin{pmatrix}0&y_1\end{pmatrix} = \begin{pmatrix}0&0\end{pmatrix}$, and Equation~\ref{ETA32} becomes $\begin{pmatrix}-w_2\\0\end{pmatrix}=\begin{pmatrix}0\\0\end{pmatrix}$. Thus $y_1 =0$, and $w_2 = 0$. $B$ is reduced to 
$
\begin{pmatrix}
a_0&&&a_2\\
&a_0&1&\\
&&a_0&\\
&&&a_0
\end{pmatrix}
$, and $\ZT{4}(A,B)=\left\{ \begin{pmatrix}x_0&&y_2&x_1\\ &z_0&z_1&w_1\\&&z_0&\\&&&x_0\end{pmatrix} \right\}$. $(A,B)$ is of type $A_5$, and there are $q(q-1)$ such branches.

{\bfseries When $C = \begin{pmatrix}a_0&\\&c_0\end{pmatrix}, a_0\neq c_0$:} Here Equation~\ref{ETA31} becomes: $\begin{pmatrix}0&y_2(c_0-a_0)\end{pmatrix} = \begin{pmatrix}0&0\end{pmatrix}$, and Equation~\ref{ETA32} becomes $\begin{pmatrix}\\(a_0-c_0)w_2\end{pmatrix}=\begin{pmatrix}0\\0\end{pmatrix}$. Thus $y_2 =0$, and $w_2 = 0$. $B$ is reduced to 
$
\begin{pmatrix}
a_0&&&a_2\\
&a_0&&\\
&&c_0&\\
&&&a_0
\end{pmatrix}
$, and $\ZT{4}(A,B)=\left\{ \begin{pmatrix}x_0&y_1&&x_1\\ &z_0&&w_1\\&&z_2&\\&&&x_0\end{pmatrix} \right\}$. $(A,B)$ is of type $B_4$, and there are $q^(q-1)(q-2)$ such branches.

{\bfseries When $C = \begin{pmatrix}b_0&\\&a_0\end{pmatrix}, a_0\neq b_0$:} Here Equation~\ref{ETA31} becomes: $\begin{pmatrix}(b_0-a_0)y_1&0\end{pmatrix} = \begin{pmatrix}0&0\end{pmatrix}$, and Equation~\ref{ETA32} becomes $\begin{pmatrix}(b_0-a_0)w_1\\0\end{pmatrix}=\begin{pmatrix}0\\0\end{pmatrix}$. Thus $y_1 =0$, and $w_1 = 0$. $B$ is reduced to 
$
\begin{pmatrix}
a_0&&&a_2\\
&b_0&&\\
&&a_0&\\
&&&a_0
\end{pmatrix}
$, and $\ZT{4}(A,B)=\left\{ \begin{pmatrix}x_0&&y_2&x_1\\ &z_0&&\\&&z_2&w_2\\&&&x_0\end{pmatrix} \right\}$. This $(A,B)$ too is of type $B_4$, and there are $q(q-1)(q-2)$ such branches.

When $(\overrightarrow{b}, \overrightarrow{d}) \neq (\overrightarrow{0}, \overrightarrow{0})$.

{\bfseries When $C = a_0I_2$:} Equation~\ref{ETA31} becomes:
\begin{equation}\label{ETA34}
\begin{pmatrix}x_0b_1&x_0b_2\end{pmatrix} = \begin{pmatrix}z_0b'_1 &z_1b'_1 + z_2b_2\end{pmatrix},
\end{equation}
and Equation~\ref{ETA32} becomes:
\begin{equation}\label{ETA35}
\begin{pmatrix}z_0d_1+z_1d_2\\ z_2d_2\end{pmatrix} = \begin{pmatrix}x_0d'_1\\ x_0d'_2\end{pmatrix}.
\end{equation}

When $b_1 = 0$ and $b_2\neq 0$. In Equation~\ref{ETA34} choose $z_2$ so that $b'_2 = 1$. Hence, on replacing $b_2$ by $b'_2 = 1$ in Equation~\ref{ETA34}, we get $x_0 = z_2$. Hence in Equation~\ref{ETA35}, $x_0d'_2 = x_0d_2$. Thus $d_2 = d'_2$. 

Here, if $d_2 = 0$, in Equation~\ref{ETA35}, we have $x_0d'_1 = z_0d_1$. When $d_1 = 0$; Equation~\ref{ETA33} becomes $x_0a_1  = x_0a'_1 + w_2$. Choose $w_2$ so that $a'_1 = 0$. Hence $B$ is reduced to
$
\begin{pmatrix}
a_0&&1&\\
&a_0&&\\
&&a_0&\\
&&&a_0
\end{pmatrix}
$
, and $\ZT{4}(A,B) = \left\{\begin{pmatrix}x_0&y_1&y_2&x_1\\&z_0&z_1&w_1\\&&x_0&\\&&&x_0\end{pmatrix} \right\}$. $(A,B)$ is of type $tNT_3$, and there are $q-1$ such branches.

When $d_1 \neq 0$, choose $z_0$ so that $d'_1 = 1$. Then Equation~\ref{ETA33} becomes $x_0a_1 +y_1 = x_0a'_1 w_2$. Choose $w_2$ such that $a'_1= 0$. With these, $B$ is reduced to:
$
\begin{pmatrix}
a_0&&1&\\
&a_0&&1\\
&&a_0&\\
&&&a_0
\end{pmatrix}
$, and $\ZT{4}(A,B) =\left\{\begin{pmatrix}x_0&y_1&y_2&x_1\\&x_0&z_1&w_1\\&&x_0&y_1\\&&&x_0\end{pmatrix}\right\}$.Now, this is a centralizer we have not seen so far. Thus we have a new type, $tNT_4$. There are $q-1$ such branches.

When $d_2 \neq 0$, in Equation~\ref{ETA35}, choose $z_1$ so that $d'_1 = 0$. Equation~\ref{ETA33} becomes $x_0a_1+ y_2d_2 = x_0a'_1 + w_2$. Choose $w_2$ such that $a'_1 = 0$. So, $B$ is reduced to 
$
\begin{pmatrix}
a_0&&1&\\
&a_0&&\\
&&a_0&d_2\\
&&&a_0
\end{pmatrix}
$
and $\ZT{4}(A,B) = \left\{\begin{pmatrix}x_0&y_1&y_2&x_1\\&z_0&&w_1\\&&x_0&d_2y_2\\&&&x_0\end{pmatrix} \right\}$. $(A,B)$ is of type $A_8$, and there are $(q-1)^2$ such branches.

When $b_1 = b_2 = d_2 = 0$. Here $d_1\neq 0$. Choose $z_0$ so that $d'_1 =1$. Then, in Equation~\ref{ETA33}, we have $x_0a_1 + y_1 = x_0a'_1$. Choose $y_1$ so that $a'_1 = 0$.
Thus $B$ is reduced to:
$
\begin{pmatrix}
a_0&&&\\
&a_0&&1\\
&&a_0&\\
&&&a_0
\end{pmatrix}
$, and $\ZT{4}(A,B) = \left\{\begin{pmatrix}x_0&&y_2&x_1\\&x_0&z_1&w_1\\&&z_2&w_2\\&&&x_0\end{pmatrix} \right\}$. $(A,B)$ is of type $tNT_3$. There are $q-1$ such branches.

When $b_1 = b_2 = 0$, and $d_2 \neq 0$. In Equation~\ref{ETA35}, choose $z_2 $ such that $d'_2 = 1$, and in the same equation, choose $z_1$ so that $d'_1 = 0$. With these, Equation~\ref{ETA33} becomes $x_0a_1 + y_2 = x_0a'_1$. Choose $y_2$ such that $a'_1 = 0$. Thus $B$ is reduced to:
$
\begin{pmatrix}
a_0&&&\\
&a_0&&\\
&&a_0&1\\
&&&a_0
\end{pmatrix}
$, and $\ZT{4}(A,B)= \left\{\begin{pmatrix}x_0&y_1&&x_1\\&z_0&&w_1\\&&x_0&w_2\\&&&x_0\end{pmatrix} \right\}$. This too is of type $A_8$. There are $q-1$ such branches.

When $b_1 \neq 0$. In Equation~\ref{ETA34}, choose $z_0$ so that $b'_1 = 1$, and choose $z_1$ so that $b'_2 = 0$. On replacing $b_1$ with $b'_1 = 1$, and $b_2 $ with $b'_2 = 0$ in Equation~\ref{ETA34}, we get $z_0 = x_0$, and $z_1 =0$. Putting these in equation~\ref{ETA35} leaves us with $d'_1 = d_1$ and $z_2d_2 = x_0d'_2$. 

With these, Equation~\ref{ETA33} is reduced to $x_0a_1 + d_1y_1 = x_0a'_1+w_1$. Choose $w_1$ so that $a'_1 = 0$.

When $d_2 = 0$, $B$ is reduced to 
$
\begin{pmatrix}
a_0&1&&\\
&a_0&&d_1\\
&&a_0&\\
&&&a_0
\end{pmatrix}
$, and $\ZT{4}(A,B)= \left\{\begin{pmatrix}x_0&y_1&y_2&x_1\\&x_0&&d_1y_1\\&&z_2&w_2\\&&&x_0\end{pmatrix} \right\}$. This too is of type $A_8$. There are $q(q-1)$ such branches.

When $d_2 \neq 0$, in Equation~\ref{ETA35}, choose $z_2$ so that $d'_2 = 1$. With these Equation~\ref{ETA33} becomes $x_0a_1 + y_1d_1 + y_2 = x_0a'_1 + w_1$. Choose $w_1$ such that $a'_1 = 0$. Thus $B$ is reduced to 
$
\begin{pmatrix}
a_0&1&&\\
&a_0&&d_1\\
&&a_0&1\\
&&&a_0
\end{pmatrix}
$, and $\ZT{4}(A,B)= \left\{\begin{pmatrix}x_0&y_1&y_2&x_1\\&x_0&&y_2+d_1y_1\\&&x_0&w_2\\&&&x_0\end{pmatrix} \right\}$. This is of type $A_9$. There are $q(q-1)$ such branches.

So, with these, we are done with all the cases, when $C = a_0I_2$.

{\bfseries When $C= \begin{pmatrix}a_0&1\\&a_0\end{pmatrix}$:} Here $Z = \begin{pmatrix}z_0&z_1\\&z_0\end{pmatrix}$. Equation~\ref{ETA31} becomes:
\begin{equation*}
\begin{pmatrix}x_0b_1&x_0b_2\end{pmatrix} + \begin{pmatrix}0 & y_1\end{pmatrix} = \begin{pmatrix}z_0b'_1&z_1b'_1+z_0b'_2\end{pmatrix}
\end{equation*}
Choose $y_1$ such that $b'_2 =0$. On substituting $b_2$ with $b'_2 = 0$  in the above equation, we have $y_1 = z_1b'_1$.

Similarly, Equation~\ref{ETA32} becomes
\begin{equation*}
\begin{pmatrix}z_0d_1+z_1d_2\\ z_0d_2\end{pmatrix} + \begin{pmatrix}-w_2 \\ 0\end{pmatrix} = \begin{pmatrix}x_0b'_1\\ x_0b'_2\end{pmatrix}.
\end{equation*}
Choose $w_2$ such that $d'_1 = 0$. On substituting $d_1$ with $d'_1 = 0$ in the above equation, we have $w_2 = d_2z_1$.

When $b_1 \neq 0$, choose $z_0$ so that $b'_1 = 1$. Then, on substituting $b_1$ with $b'_1 = 1$ in Equation~\ref{ETA31}, we get $z_0 = x_0$, and thus $d'_2 = d_2$. With these, Equation~\ref{ETA33} becomes $x_0a_1 + y_2d_2 = x_0a'_1 + w_1$. Choose $w_1$ such that $a'_1 = 0$. Thus, $B$ is reduced to 
$
\begin{pmatrix}
a_0&1&&\\
&a_0&1&\\
&&a_0&d_2\\
&&&a_0
\end{pmatrix},
$
and $\ZT{4}(A,B) = \left\{\begin{pmatrix}x_0&z_1&y_2&x_1\\&x_0&z_1&d_2y_2\\&&x_0&d_2z_1\\&&&x_0 \end{pmatrix}\right\}$. This $(A,B)$ is of type $R_1$, and there are $q(q-1)$ such branches.

When $b_1 = 0$, and $d_2 \neq 0$ $y_1 = 0$. In Equation~\ref{ETA32}, choose $z_0$ so that $d'_2 = 1$. With these, Equation~\ref{ETA33} becomes $x_0a_1 + y_2 = x_0a'_1$. Choose $y_2$ so that $a'_1 = 0$. 
Thus, $B$ is reduced to 
$
\begin{pmatrix}
a_0&&&\\
&a_0&1&\\
&&a_0&1\\
&&&a_0
\end{pmatrix},
$
and $\ZT{4}(A,B) = \left\{\begin{pmatrix}x_0&&&x_1\\&x_0&z_1&w_1\\&&x_0&z_1\\&&&x_0 \end{pmatrix}\right\}$. By a routine check, one can see that this is commutative. Thus $(A,B)$ is of type $R_1$, and there are $q-1$ such branches.

{\bfseries When $C = \begin{pmatrix}a_0&\\&b_0\end{pmatrix}, b_0\neq a_0$:} Here $Z= \begin{pmatrix}z_0&\\&z_2\end{pmatrix}$. Equation~\ref{ETA31} becomes :
\begin{equation*}
\begin{pmatrix}x_0b_1 & x_0b_2\end{pmatrix} + \begin{pmatrix}0& (b_0-a_0)y_2\end{pmatrix} = \begin{pmatrix} z_0b'_1 & z_2b'_2\end{pmatrix}
\end{equation*}
As $b_0-a_0\neq 0$, choose $y_2$ such that $b'_2 = 0$. Hence, on replacing $b_2$ by $b'_2 = 0$ in the above equation, we get $y_2 = 0$.

Similarly Equation~\ref{ETA32} becomes:

\begin{equation*}
\begin{pmatrix}z_0d_1 & z_2d_2\end{pmatrix} + \begin{pmatrix}0 \\ (a_0-b_0)w_2\end{pmatrix} = \begin{pmatrix} x_0d'_1 & x_0d'_2\end{pmatrix}
\end{equation*}
Choose $w_2$ so that $d'_2 = 0$. So, if we replace $d_2$ by $d'_2 = 0$ in the above equation, we have $w_2 =0$.

When $b_1 =0$ and $d_1\neq 0$, choose $z_0$ so that $d'_1 =1$. With these, Equation~\ref{ETA33} becomes $x_0a_1 + y_1 = x_0a'_1$. Choose $y_1$ so that $a'_1 = 0$. $B$ is thus reduced to
$
\begin{pmatrix}
a_0&&&\\
&a_0&&1\\
&&b_0&\\
&&&a_0
\end{pmatrix}
$, and $\ZT{4}(A,B)=\left\{ \begin{pmatrix}x_0&&&x_1\\&x_0&&w_1\\&&z_2&\\&&&x_0\end{pmatrix}\right\}$. By a routine check, we can see that this centralizer is commutative. Thus $(A,B)$ is of type $R_2$, and there are $(q-1)(q-2)$ such branches.

When $b_1 \neq 0$. in Equation~\ref{ETA31} for this $C$, choose $z_0$ so that $b'_1 = 1$. Thus on substituting $b_1$ with $b'_1=1$ in the same, we get $z_0 = x_0$. Hence, from Equation~\ref{ETA32} for this case, we have $d'_1 = d_1$. With these Equation~\ref{ETA33} becomes $x_0a_1 + d_1y_1 = x_0a'_1 + w_1$. Choose $w_1$ so that $a'_1=0$. Hence $B$ is reduced to 
$
\begin{pmatrix}
a_0&1&&\\
&a_0&&d_1\\
&&b_0&\\
&&&a_0
\end{pmatrix}
$
and $\ZT{4}(A,B)=\left\{ \begin{pmatrix}x_0&y_1&&x_1\\&x_0&&d_1y_1\\&&z_2&\\&&&x_0\end{pmatrix}\right\}$. Easy to see that this centralizer too is commutative. Thus $(A,B)$ is of type $R_2$, and there are $q(q-1)(q-2)$ such branches.

{\bfseries When $C = \begin{pmatrix}b_0&\\&a_0\end{pmatrix}, b_0\neq a_0$:} Here $Z= \begin{pmatrix}z_0&\\&z_2\end{pmatrix}$. Equation~\ref{ETA31} becomes :
\begin{equation*}
\begin{pmatrix}x_0b_1 & x_0b_2\end{pmatrix} + \begin{pmatrix} (b_0-a_0)y_1 & 0\end{pmatrix} = \begin{pmatrix} z_0b'_1 & z_2b'_2\end{pmatrix}
\end{equation*}
As $b_0-a_0\neq 0$, choose $y_1$ such that $b'_1 = 0$. Hence, on replacing $b_1$ by $b'_1 = 0$ in the above equation, we get $y_1 = 0$.

Similarly Equation~\ref{ETA32} becomes $
\begin{pmatrix}z_0d_1 & z_2d_2\end{pmatrix} + \begin{pmatrix}(a_0-b_0)w_1\\ 0\end{pmatrix} = \begin{pmatrix} x_0d'_1 & x_0d'_2\end{pmatrix}.$
Choose $w_1$ so that $d'_1 = 0$. So, if we replace $d_1$ by $d'_1 = 0$ in the above equation, we have $w_1 =0$.

When $b_2 = 0$ and $d_2\neq 0$, choose $z_2$ so that $d'_2 =1$. With these, Equation~\ref{ETA33} becomes $x_0a_1 + y_2 = x_0a'_1$. Choose $y_2$ so that $a'_1 = 0$. $B$ is thus reduced to
$
\begin{pmatrix}
a_0&&&\\
&b_0&&\\
&&a_0&1\\
&&&a_0
\end{pmatrix}
$, and $\ZT{4}(A,B)=\left\{ \begin{pmatrix}x_0&&&x_1\\&z_0&&\\&&x_0&w_2\\&&&x_0\end{pmatrix}\right\}$. By a routine check, we can see that this centralizer is commutative. This $(A,B)$ is of type $R_2$, and there are $(q-1)(q-2)$ such branches.

When $b_2 \neq 0$, in Equation~\ref{ETA31} for this $C$, choose $z_2$ so that $b'_2 = 1$. Thus on substituting $b_2$ with $b'_2= 1$ in the same, we get $z_2 = x_0$. Hence, from Equation~\ref{ETA32} for this case, we have $d'_2 = d_2$. With these Equation~\ref{ETA33} becomes $x_0a_1 + d_2y_2 = x_0a'_1 + w_2$. Choose $w_2$ so that $a'_1=0$. Hence $B$ is reduced to 
$
\begin{pmatrix}
a_0&1&&\\
&a_0&&d_1\\
&&b_0&\\
&&&a_0
\end{pmatrix}
$
and $\ZT{4}(A,B)=\left\{ \begin{pmatrix}x_0&&y_2&x_1\\&z_0&&\\&&x_0&d_2y_2\\&&&x_0\end{pmatrix}\right\}$. Easy to see that this centralizer is commutative. This $(A,B)$  too is of type $R$, and there are $q(q-1)(q-2)$ such branches.

With these, we have covered all the subcases under the case of $a_0$ being an eigenvalue of $C$.

\noindent{When $a_0$ is not an eigenvalue of $C$:} In this case $C-a_0I_2$ is invertible. Hence, in Equation~\ref{ETA31}, choose $y_1, y_2$ so that $b'_1 = b'_2 = 0$. Similarly, in Equation~\ref{ETA32}, choose $w_1, w_2$ so that $d'_1 = d'_2 = 0$.

So, Equation~\ref{ETA33} becomes $x_0a_1 = x_0a'_1 $, thus $a'_1 = a_1$.

{\bfseries When $C = b_0I_2, b_0\neq a_0$:} $B$ is reduced to 
$
\begin{pmatrix}
a_0&&&a_1\\
&b_0&&\\
&&b_0\\
&&&a_0
\end{pmatrix}
$, and $\ZT{4}(A,B) = \left\{\begin{pmatrix}x_0&&&x_1\\&z_0&z_1&\\&&z_2&\\&&&x_0\end{pmatrix} \right\}$. This $(A,B)$ is of type $B_5$. There are $q(q-1)(q-2)$ such branches.

{\bfseries When $C = \begin{pmatrix}b_0&1\\&b_0\end{pmatrix}, b_0 \neq a_0$:}In this case, $B$ is reduced to 
$
\begin{pmatrix}
a_0&&&a_1\\
&b_0&1&\\
&&b_0&\\
&&&a_0
\end{pmatrix}
$, and $\ZT{4}(A,B) = \left\{\begin{pmatrix}x_0&&&x_1\\&z_0&z_1&\\&&z_0&\\&&&x_0\end{pmatrix} \right\}$. This one is a commutative centralizer. This $(A,B)$ is of type $R_3$, and there are $q(q-1)(q-2)$ such branches.

{\bfseries When $C = \begin{pmatrix}b_0&\\&c_0\end{pmatrix}, b_0,c_0 \neq a_0$, and $b_0\neq c_0$:}In this case, $B$ is reduced to 
$
\begin{pmatrix}
a_0&&&a_1\\
&b_0&&\\
&&c_0&\\
&&&a_0
\end{pmatrix}
$, and $\ZT{4}(A,B) = \left\{\begin{pmatrix}x_0&&&x_1\\&z_0&&\\&&z_2&\\&&&x_0\end{pmatrix} \right\}$. This $(A,B)$ is of type $R$, and there are $q(q-1)(q-2)(q-3)$ such branches.

So, those are all the cases available. 

Adding up all the branches of type $A_8$, we have a total of $q-1 + q(q-1) + (q-1)^2 = 2q(q-1)$ branches.
\end{proof}

\begin{prop}\label{BrT4A4}
A matrix of type $A_4$ has:
\begin{center}
 \begin{tabular}{c|c||c|c}\hline
Branch & No. of Branches & Branch & No. of Branches\\ \hline
$A_4$ & $q(q-1)$ & $tNT_1$ & $q(q-1)^2$\\
$R_1$ & $q^3-q^2$ & $tNT_5$ & $q(q-1)$\\ 
$R_3$& $q^2(q-1)(q-2)$ &&\\ \hline
 \end{tabular}
 \end{center}
A new type $tNT_5$ appears with centralizer $\left\{  \begin{psmallmatrix} a_0 & a_1 & b_0 & b_1\\&a_0& &b_0\\&&a_0 & c_1 \\&&& a_0 \end{psmallmatrix}\mid a_0 \neq 0\right\}$. 
\end{prop}
\begin{proof}
 The canonical form of a matrix of this type is $A = \begin{pmatrix}a&1&&\\&a&&\\&&a&1\\&&&a\end{pmatrix}$. Then 
 $\ZT{4}(A) = \left\{\begin{pmatrix}a_0&a_1&b_0&b_1\\&a_0&&b_0\\&&c_0&c_1\\&&&c_0\end{pmatrix}\mid a_0,c_0\neq 0 \right\}.$
 Let $B = \begin{pmatrix}a_0&a_1&b_0&b_1\\&a_0&&b_0\\&&c_0&c_1\\&&&c_0\end{pmatrix}$, and $B' = \begin{pmatrix}a_0&a_1&b'_0&b'_1\\&a_0&&b'_0\\&&c_0&c_1\\&&&c_0\end{pmatrix} = XBX^{-1}$, where $X = \begin{pmatrix}x_0&x_1&y_0&y_1\\&x_0&&y_0\\&&z_0&z_1\\&&&z_0\end{pmatrix}$. $XB = B'X$ gives us the following:
 \begin{eqnarray}
  x_0b_0 + y_0c_0 &=& z_0b'_0 + y_0a_0\label{ETA41}\\
  x_0b_1 + x_1b_0 + y_0c_1 + y_1c_0 &=& y_1a_0 + y_0a_1 + z_1b'_0 + z_0b'_1\label{ETA42}
 \end{eqnarray}

\noindent{\bfseries When $a_0 = c_0$:} From Equation~\ref{ETA41}, we have $x_0b_0 = z_0b'_0$.

{\bfseries When $b_0 = 0$}: Equation~\ref{ETA42} becomes $x_0b_1 + y_0c_1 = z_0b'_1 +y_0a_1$. Here we first look at what happens when $a_1 = c_1$, and $b_1 = 0$. Here $B$  reduces to 
$\begin{pmatrix}
a_0&a_1&&\\
&a_0&&\\
&&a_0&a_1\\
&&&a_0
\end{pmatrix}$, $\ZT{4}(A,B) = \ZT{4}(A)$. $(A,B)$ is of type $A_4$, and there are $q(q-1)$ such branches.

When $a_1 = c_1$, and $b_1\neq 0$. We can choose $x_0$ such that $b'_1 = 1$. Thus $B$ is reduced to 
$\begin{pmatrix}
a_0&a_1&&1\\
&a_0&&\\
&&a_0&a_1\\
&&&a_0
\end{pmatrix}
$, and $\ZT{4}(A,B) = \left\{ \begin{pmatrix}x_0&x_1&y_0&y_1\\&x_0&&y_0\\&&x_0&z_1\\&&&x_0\end{pmatrix} \right\}$. We see a centralizer, not isomorphic to the ones seen so far. Thus, we have a new type $tNT_5$. $(A,B)$ is of type $tNT_5$, and there are $q(q-1)$ such branches.

When $a_1\neq c_1$, in Equation~\ref{ETA42}, we can choose $y_0$, so that $b'_1 = 0$. Thus, $B$ is reduced to $\begin{pmatrix}
a_0&a_1&&\\
&a_0&&\\
&&a_0&c_1\\
&&&a_0
\end{pmatrix}
$, and $\ZT{4}(A,B) = \left\{ \begin{pmatrix}x_0&x_1&&y_1\\&x_0&&\\&&z_0&z_1\\&&&z_0\end{pmatrix} \right\}$ $(A,B)$ is of type $tNT_1$, and there are $q(q-1)^2 $ such branches.

{\bfseries When $b_0 \neq 0$}: In Equation~\ref{ETA41}, choose $x_0$ such that $b'_0 = 1$. Then, on replacing $b_0$ and $b'_0$ by 1 in the same equation, we have $x_0 = z_0$. Hence, Equation~\ref{ETA42} becomes $x_0b_1 + x_1 + y_0c_1 = x_0b'_1 + z_1 + y_0a_1$. Hence, choose $z_1$ so that $b'_1 = 0$. 
Then, $B$ is reduced to $\begin{pmatrix}
a_0&a_1&1&\\
&a_0&&1\\
&&a_0&c_1\\
&&&a_0
\end{pmatrix}
$, and $\ZT{4}(A,B) = \left\{ \begin{pmatrix}x_0&x_1&&y_1\\&x_0&&\\&&x_0&\begin{smallmatrix}x_1+\\  y_0(c_1-a_1)\end{smallmatrix}\\&&&x_0\end{pmatrix} \right\}$ $(A,B)$ is of type $R_1$, and there are $q^2(q-1)$ such branches.

\noindent{\bfseries When $a_0\neq c_0$}: In Equation~\ref{ETA41}, choose $y_0$ so that $b'_0 = 0$. With this, Equation~\ref{ETA42} becomes $x_0b_1+y_1c_0 = z_0b'_1+y_1a_0$. Choose $y_1$ such that $b'_1 = 0$. Thus, $B$ is reduced to $\begin{pmatrix}
a_0&a_1&&\\
&a_0&&\\
&&c_0&c_1\\
&&&c_0
\end{pmatrix}
$, and $\ZT{4}(A,B) = \left\{ \begin{pmatrix}x_0&x_1&&\\&x_0&&\\&&z_0&z_1\\&&&z_0\end{pmatrix} \right\}$. $(A,B)$ is of type $R_3$, and there are $q^2(q-1)(q-2) = q^4-q^3 $ such branches. Thus, there are no more cases left to deal with.

\end{proof}
\begin{prop}\label{BrT4A5}
 An upper triangular matrix of type $A_5$ has $q^2(q-1)$ branches of type $A_5$, $q^2(q-1)(q-2)$ branches of type $R_3$, and $q^2(q^2-1)$ branches of type $NR_1$.
 \end{prop}
\begin{proof}
 A matrix of type $A_5$ has the canonical form: $A = \begin{pmatrix}a&&&1\\&a&1&\\ &&a&\\&&&a\end{pmatrix}$. Thus its centralizer $\ZT{4}(A)$ is: 
 $\left\{\begin{pmatrix}a_0 & & a_2 & a_3\\ &b_0 & b_1 & b_2 \\ & & b_0 & \\ &&&a_0\end{pmatrix} \right\}.$
Let $B = \begin{pmatrix}a_0 & & a_2 & a_3\\ &b_0 & b_1 & b_2 \\ & & b_0 & \\ &&&a_0\end{pmatrix}$, and $B' = \begin{pmatrix}a_0 & & a'_2 & a'_3\\ &b'_0 & b'_1 & b_2 \\ & & b_0 & \\ &&&a_0\end{pmatrix} = XBX^{-1}$, where $X = \begin{pmatrix}x_0 & & x_2 & x_3\\ &y_0 & y_1 & y_2 \\ & & y_0 & \\ &&&x_0\end{pmatrix}.$ Thus, from $XB = B'X$, we have $a'_3 = a_3$, $b'_1 = b_1$, and the following equations:

\begin{eqnarray}
 x_0a_2 + x_2b_0 &=& y_0a'_2 + x_2a_0 \label{ETA51}\\
 y_0b_2 + a_0y_2 &=& x_0b'_2 + y_2b_0 \label{ETA52}
\end{eqnarray}

Case $a_0 = b_0$.
Equations~\ref{ETA51} and \ref{ETA52} become $x_0a_2 = y_0a'_2$, and $y_0b_2 = x_0b'_2$ respectively.

If $a_2 = b_2 = 0$, the above equations are void, and we have $B$ reduced to 
$\begin{pmatrix}
a_0&&&a_3\\
&a_0&b_1&\\
&&a_0&\\
&&&a_0
\end{pmatrix}$, and $\ZT{4}(A,B)  = \ZT{4}(A)$. Thus $(A,B)$ is a branch of type $A_5$, and there are $q^2(q-1)$ such branches.

If $a_2 \neq 0$, then choose $x_0$ so that $a'_2 = 1$. Substituting $a_2$ with $a'_2 = 1$ in the equation $x_0a_2 = y_0a'_2$, we get $x_0 = y_0$, thus leaving us with $b'_2 = b_2$. Hence $B$ is reduced to $\begin{pmatrix}a_0&&1&a_3\\&a_0&b_1&b_2\\&&a_0&\\*&&&a_0\end{pmatrix}$, and $\ZT{4}(A,B) = \left\{\begin{pmatrix}x_0&&x_2&x_3\\&x_0&y_1&y_2\\&&x_0&\\&&&x_0\end{pmatrix} \right\}$. Thus $(A,B)$ is a branch of type $NR_1$, and there are $q^3(q-1)$ such branches.

If $a_2 =0$ and $b_2 \neq 0$, then we choose $y_0$ such that $b'_2=1$. Thus $B$ is reduced to $\begin{pmatrix}a_0&&&a_3\\&a_0&b_1&1\\&&a_0&\\*&&&a_0\end{pmatrix}$, and $\ZT{4}(A,B) = \left\{\begin{pmatrix}x_0&&x_2&x_3\\&x_0&y_1&y_2\\&&x_0&\\&&&x_0\end{pmatrix} \right\}$. This branch too is of type $NR_1$, and there are $q^2(q-1)$ such branches.

If $a_0 \neq b_0$. Then, in Equation~\ref{ETA51}, choose $x_2$ such that $a'_2 = 0$. Similarly in Equation~\ref{ETA52}, choose $y_2 $ such that $b'_2 = 0$. Thus $B$ boils down to 
$\begin{pmatrix}a_0&&&a_3\\&b_0&b_1&\\&&b_0&\\*&&&a_0\end{pmatrix}$, and $\ZT{4}(A,B) = \left\{\begin{pmatrix}x_0&&&x_3\\&y_0&y_1&\\&&y_0&\\&&&x_0\end{pmatrix} \right\}$. This $(A,B)$ is of type $R_3$, and there are $q2^2(q-1)(q-2)$ such branches. 

Adding up the branches of type $NR_1$, we have a total of $q^3(q-1) + q^2(q-1) = q^2(q^2-1)$ branches of type $NR_1$.
\end{proof}

\begin{prop}\label{BrT4A6}
 For a matrix of type $A_6$, the branchings are:
 \begin{center}
 \begin{tabular}{c|c||c|c}\hline
 Branch & No. of Branches & Branch & No. of Branches\\ \hline
 $A_6$& $q(q-1)$& $R_3$ & $q^2(q-1)(q-2)$\\
 $A_5$& $q(q-1)^2$& $tNT_4$&$q^2(q-1)$\\
 $R_1$& $q^2(q-1)$& $NR_1$& $q^2(q-1)$.\\ \hline
 \end{tabular}
  \end{center}
\end{prop}
\begin{proof}
A matrix of type $A_6$ has the canonical form 
$\begin{pmatrix}
a&&1&\\
&a&&1\\
&&a&\\
&&&a
\end{pmatrix}$. The centralizer subgroup $\ZT{4}(A)$ is $\left\{\begin{pmatrix}C&D\\&C\end{pmatrix}\mid C\in \Tq{2}\right\},$ where $D = \begin{pmatrix}d_0&d_1\\ d_2&d_3\end{pmatrix}$, and $W = \begin{pmatrix}w_0&w_1\\ w_2&w_3\end{pmatrix}$.

Let $B = \begin{pmatrix}C&D\\&C\end{pmatrix}$, and $B' = \begin{pmatrix}C'&D'\\&C'\end{pmatrix}= XBX^{-1}$, where $X = \begin{pmatrix}Z&W\\&Z\end{pmatrix}$. So $XB = B'X$ leads to $ZC = C'Z$. Hence, we can take $C$ to be a representative of a conjugacy class in $\Tq{2}$, and $Z = \ZT{2}(C)$. We have the following equation
\begin{equation}\label{ETA61}
ZD + WC = CW + D'Z
\end{equation}
So the cases to deal with here are the three conjugacy class types in $\Tq{2}$.

\noindent{\bfseries Case $C = \left(\begin{smallmatrix}a_0&1\\&a_0\end{smallmatrix}\right)$:} here $Z = \begin{pmatrix}x_0&x_1\\&x_0\end{pmatrix}$, and Equation~\ref{ETA61} becomes:
$$\begin{pmatrix}x_0d_0 + x_1d_2 & x_0d_1 + x_1d_3 + w_0\\ x_0d_2 & x_0d_3  +w_2\end{pmatrix} = \begin{pmatrix} w_2 + x_0d'_0 & w_3 + x_1d'_0 + x_0d_1\\ x_0d'_2 & x_1d'_2 + x_0d_3\end{pmatrix}$$

 Choose $w_2 $ so that $d'_0 = 0$. Thus, on replacing $d_0$ by $0$, we get $w_2 = x_1d_2$, and hence $d'_3 = d_3$.

We can choose $w_0$ such that $d'_1 = 0$. Thus $B$ is reduced to $\begin{pmatrix} a_0&1&&\\&a_0&d_2&d_3\\&&a_0&1\\&&&a_0\end{pmatrix}$, and $\ZT{4}(A,B)= \left\{\begin{pmatrix}x_0&x_1&w_0&w_1\\&x_0&x_1d_2&w_0+x_1d_3\\&&x_0&x_1\\&&&x_0\end{pmatrix}\right\}$. This $(A,B)$ is of type $R_1$, and there are $q^2(q-1)$ such branches. 

\noindent{\bfseries Case $C = \left(\begin{smallmatrix}a_0&\\&b_0\end{smallmatrix}\right), a_0\neq b_0$:} here $Z = \begin{pmatrix}x_0&\\&x_3\end{pmatrix}$, and Equation~\ref{ETA61} becomes:
$$\begin{pmatrix}x_0d_0 + a_0w_0 & x_0d_1 + w_1b_0\\ x_3d_2+a_0w_2 & x_3d_3+b_0w_3  \end{pmatrix} = \begin{pmatrix} a_0w_0 + x_0d'_0 & w_1a_0 + x_3d_1\\ b_0w_2 +x_0d_2& b_0w_3 + x_3d_3\end{pmatrix}.$$
We have $d'_0 = d_0$ and $d'_3 = d_3$. As $a_0 \neq b_0$, choose $w_2$ such that $d'_2 = 0$, and $w_1$ so that $d'_1 = 0$. Thus $B$ is reduced to $\begin{pmatrix} a_0&&d_0&\\&b_0&&d_3\\&&a_0&\\&&&d_0\end{pmatrix}$, and $\ZT{4}(A,B)= \left\{\begin{pmatrix}x_0&&w_0&\\&x_3&&w_3\\&&x_0&\\&&&x_3\end{pmatrix}\right\}$. This $(A,B)$ is of type $R_3$, and there are $q^2(q-1)(q-2)$ such branches. 

\noindent{\bfseries Case $C = a_0I_2$:} Here Equation~\ref{ETA61} becomes:
$ZD = D'Z$, where $Z \in T_2(\Fq)$. With $Z = \begin{pmatrix}x_0&x_1\\&x_2\end{pmatrix}$, we see that:
\begin{equation}
\label{ETA62}
\begin{pmatrix}x_0d_0 + x_1d_2 & x_0d_1+ x_1d_3\\ x_3d_2 & x_3d_3\end{pmatrix} = \begin{pmatrix}x_0d'_0 & x_1d'_0 + x_3d'_1 \\ x_0d'_2 & x_1d'_2 + x_3d'_3\end{pmatrix}.
\end{equation}
We see that $x_0d'_2 = x_3d_2$. We have two main cases here:

{\bfseries Case $d_2 = 0$.} In this case, from Equation~\ref{ETA62} we have $d'_0 = d_0$, and $d'_3 = d_3$, and we have $x_0d_1 + (d_3-d_0)x_1 = x_3d'_1$.

When $d_0 = d_3$, we have $x_0d_1 = x'_3 d_1$. Now, if $d_1 = 0$. we have $B = \begin{pmatrix}a_0I_2 & d_0I_2\\&a_0I_2\end{pmatrix}$, and $\ZT{4}(A,B) = \ZT{4}(A)$. Thus, $(A,B)$ is of type $A_6$, and there are $q(q-1)$ such branches.

If $d_1\neq 0$, choose $x_0$ so that $d'_1 = 1$. Thus $B$ is reduced to $\begin{pmatrix} a_0&&d_0&1\\&a_0&&d_3\\&&a_0&\\&&&a_0\end{pmatrix}$, and 
$$
\ZT{4}(A,B) = \left\{\begin{pmatrix}x_0&x_1&w_1&w_2\\ &x_0 &w_2&w_3\\&&x_0&x_1\\&&&x_0\end{pmatrix}\right\}.
$$
$(A,B)$ is therefore of type $tNT_4$, and there are $q^2(q-1)$ such branches.

When $d_0 \neq d_3$, in the (1,2)th entry of Equation~\ref{ETA62}, we choose $x_1$ so that $d'_1 = 0$. Thus $B$ is reduced to $\begin{pmatrix}a_0&&d_0&\\&a_0&&d_3\\&&a_0&\\&&&a_0\end{pmatrix}$, and $\ZT{4}(A,B) = \left\{\begin{pmatrix}x_0&&w_1&w_2\\ &x_3 &w_2&w_3\\&&x_0&\\&&&x_3\end{pmatrix}\right\}.$ This is isomorphic to the centralizer of a matrix of type $A_5$. Thus $(A,B)$ is a branch of type $A_5$, and there are $q^2(q-1)$ such branches.

{\bfseries Case $d_2 \neq 0$.} First, we choose $x_0$ such that $d'_2 = 1$. On replacing $d_2$ with $d'_2 = 1$ in Equation~\ref{ETA62}, and equating, we get $x_0 = x_3$. 

In the same equation, we can choose $x_1$ such that $d'_0  =0$. On replacing $d_0$ with $d'_0 = 0$ and equating, we get $x_1 =0$. Thus, $d'_3 = d_3$. Lastly, we have $x_0d_1 = x_0d'_1$, hence $d'_1 = d_1$.

Thus $B$ is reduced to $\begin{pmatrix}a_0&&&d_1\\&a_0&1&d_3\\&&a_0&\\&&&a_0\end{pmatrix}$, and $\ZT{4}(A,B) = \left\{\begin{pmatrix}x_0I_2&W\\&x_0I_2\end{pmatrix}\mid W \in M_2(\Fq)\right\}.$  $(A,B)$ is a branch of type $NR_1$, and there are $q^2(q-1)$ such branches.

There are no other cases. 
\end{proof}

\begin{prop}\label{BrT4A789}
The branching rules of remaining $A$ types are as follows.
\begin{enumerate}
\item For a matrix of type $A_7$, there are $q^2(q-1)$ branches of type $A_7$, $q^2(q-1)$ branches of type $R_1$, and $q^2(q-1)(q-2)$ branches of type $R_2$.
\item The type $A_8$ has $q^2(q-1)$ branches of type $A_8$,  $q^3-q$ branches of type $R_1$, and $q^2(q-1)(q-2)$ branches of type $R_2$.
 \item The type $A_9$ has $q^2(q-1)$ branches of type $A_9$, $(q^2-q)(q^2-1)$ branches of type $R_1$.
\end{enumerate}\end{prop}
\begin{proof}
 \begin{enumerate}
  \item A matrix of type $A_7$ has two non-similar canonical forms, $\begin{pmatrix}a&1&&\\&a&1&\\&&a&\\&&&a\end{pmatrix}$, and $\begin{pmatrix}a&&&\\&a&1&\\&&a&1\\&&&a\end{pmatrix}$. As their centralizer subgroups in $\tq{4}$ are conjugate in $GL_4(\Fq)$, we may prove the branching for any one. Let $A = \begin{pmatrix}a&1&&\\&a&1&\\&&a&\\&&&a\end{pmatrix}$. Then $\ZT{4}(A)= \left\{\begin{pmatrix} a_0&a_1&a_2&a_3\\&a_0&a_1&\\&&a_0&\\&&&d_0\end{pmatrix}\mid a_0,d_0\neq 0\right\}$. 
  
  Let $B = \begin{pmatrix} a_0&a_1&a_2&a_3\\&a_0&a_1&\\&&a_0&\\&&&d_0\end{pmatrix}$, and $B' = \begin{pmatrix} a_0&a'_1&a'_2&a'_3\\&a_0&a'_1&\\&&a_0&\\&&&d_0\end{pmatrix} = XBX^{-1}$, where $X = \begin{pmatrix} x_0&x_1&x_2&x_3\\&x_0&x_1&\\&&x_0&\\&&&z_0\end{pmatrix}$. From $XB = B'X$ we have $a'_1 = a_1$, $a'_2 = a_2$, and this equation:
 \begin{equation}\label{ETA7}
 x_0a_3 +x_3d_0 = z_0a'_3 + x_3 a_0
 \end{equation}.
 
 If $a_0 = d_0$, then Equation~\ref{ETA7} becomes $x_0a_3 = z_0z'_3$.
 
   Here, if $a_3 = 0$, then $B$ is reduced to $\begin{pmatrix}a_0&a_1&a_2&\\ &a_0&a_1&\\&&a_0&\\&&&a_0\end{pmatrix}$, and $\ZT{4}(A,B) = \ZT{4}(A)$. Thus $(A,B)$ is of type $A_7$, and there are $q^2(q-1)$ such branches,.
   
   If $a_3\neq 0$, then choose $z_0$ so that $a'_3 = 1$. Thus, $B$ is reduced to
   $
   \begin{pmatrix}
   a_0&a_1&a_2&1\\
   &a_0&a_1&\\
   &&a_0&\\
   &&&a_0
   \end{pmatrix}
   $, and $\ZT{4}(A,B) = \left\{\begin{pmatrix}x_0&x_1&x_2&x_3\\ &x_0&x_1&\\&&x_0&\\&&&x_0\end{pmatrix}\right\}$.  This $(A,B)$ is of type $R_1$, and there are $q^2(q-1)$ such branches.
   
 When $a_0\neq d_0$, then, in Equation~\ref{ETA7}, choose $x_3$ so that $a'_3 = 0$. Thus $B$ is reduced to 
 $
 \begin{pmatrix}
 a_0&a_1&a_2&\\
 &a_0&a_1&\\
 &&a_0&\\
 &&&d_0
 \end{pmatrix}
 $, and $\ZT{4}(A,B) = \left\{\begin{pmatrix}x_0&x_1&x_2&\\ &x_0&x_1&\\&&x_0&\\&&&z_0\end{pmatrix}\right\}$. This too is commutative (by a routine check). $(A,B)$ is of type $R$, and there are $q^2(q-1)(q-2)$ such branches. There are no other cases left to analyze, so these are all the branches.

 \item Matrices of type $A_8$ have either of the two non-similar canonical forms:
 $\begin{pmatrix}a&1&&\\&a&&1\\&&a&\\&&&a\end{pmatrix}$, and $\begin{pmatrix}a&&1&\\&a&&\\&&a&1\\&&&a\end{pmatrix}$. As their centralizers are conjugate in $GL_4(\Fq)$, it is enough to prove for any one of the canonical forms. Let $A = \begin{pmatrix}a&1&&\\&a&&1\\&&a&\\&&&a\end{pmatrix}$. Then the centralizer of $A$ is $\ZT{4}(A) = \left\{\begin{pmatrix}a_0&a_1&b&a_2\\&a_0&&a_1\\&&d&c\\&&&a_0\end{pmatrix}\right\}$. let $B \in \ZT{4}(A)$ be the matrix 
 $
 \begin{pmatrix}
 a_0&a_1&b&a_2\\
 &a_0&&a_1\\
 &&d&c\\
 &&&a_0
 \end{pmatrix}
 $, and let $B' = 
  \begin{pmatrix}
  a_0&a'_1&b'&a'_2\\
  &a_0&&a'_1\\
  &&d&c'\\
  &&&a_0
  \end{pmatrix}
   = XBX^{-1}$, where $X=
    \begin{pmatrix}
    x_0&x_1&y&x_2\\
    &x_0&&x_1\\
    &&z&w\\
    &&&x_0
    \end{pmatrix}
    $. Now $XB = XB'X$ leads us to $a'_1 = a_1$, and the following equations:
    \begin{eqnarray}
    x_0b + yd &=& zb' + ya_0\label{ETA81}\\
    zc + wa_0 &=& x_0c' + wd \label{ETA82}\\
    x_0a_2 + yc &=& wb' + x_0a'_2 \label{ETA83}
    \end{eqnarray}
    
\noindent{\bfseries When $a_0 = d$:} Here, Equations~\ref{ETA81} and \ref{ETA82} become 
$x_0b = zb'$, and $zc = x_0c'$ respectively.

When $b = c = 0$, Equation~\ref{ETA83} becomes $x_0a_2 = x_0a'_2$, hence $a'_2 =a_2$. $B$ is reduced to 
$\begin{pmatrix}
a_0&a_1&&a_2\\
&a_0&&a_1\\
&&a_0&\\
&&&a_0
\end{pmatrix}
$, and $\ZT{4}(A,B) = \ZT{4}(A)$. $(A,B)$ is of type $A_8$, and there are $q^2(q-1)$ such branches.

When $b \neq 0$, choose $z$ such that $b' = 1$. Then, on substituting $b$ with $b' =1$ in Equation~\ref{ETA81}, we get $z = x_0$. Thus, we have $c = c'$. And, in Equation~\ref{ETA83}, choose $w$ so that $a'_2 = 0$. Thus $B$ is reduced to 
$
\begin{pmatrix}
a_0 & a_1 & 1&\\
&a_0&&a_1\\
&&a_0&\\
&&&a_0
\end{pmatrix}$, and $\ZT{4}(A,B) = \left\{\begin{pmatrix}x_0&x_1&y&x_2\\&x_0&&x_1\\&&x_0&cy\\&&&x_0
\end{pmatrix}\right\}$. $(A,B)$ is of type $R_1$, and there are $q^2(q-1)$ such branches.

When $b= 0$ and $c \neq 0$, in Equation~\ref{ETA82}, choose $x_0$ such that $c'=1$. Then Equation~\ref{ETA83} becomes $x_0a_2 + y = x_0a'_2$. Thus, choose $y$ so that $a'_2 = 0$. Hence $B$ is reduced to 
$
\begin{pmatrix}
a_0 & a_1 & &\\
&a_0&&a_1\\
&&a_0&1\\
&&&a_0
\end{pmatrix}$, and $\ZT{4}(A,B) = \left\{\begin{pmatrix}x_0&x_1&&x_2\\&x_0&&x_1\\&&x_0&w\\&&&x_0
\end{pmatrix}\right\}$. $(A,B)$ is of type $R_1$, and there are $q(q-1)$ such branches.

There are no further cases for us to look at here. We now look at the case of $a_0 \neq d$.

\noindent{\bfseries When $a_0\neq d$:} In Equation~\ref{ETA81}, choose $y$ such that $b'=0$, and in Equation~\ref{ETA82}, choose $w$ such that $c' = 0$. Then Equation~\ref{ETA83} becomes $x_0a_2 = a_0a'_2$, implying $a'_2 = a_2$. $B$ reduces to 
$
\begin{pmatrix}
a_0 & a_1 & &a_2\\
&a_0&&a_1\\
&&d&\\
&&&a_0
\end{pmatrix}$, and $\ZT{4}(A,B) = \left\{\begin{pmatrix}x_0&x_1&&x_2\\&x_0&&x_1\\&&z&\\&&&x_0
\end{pmatrix}\right\}$. This too is a commutative centralizer. $(A,B)$ is of type $R_2$, and there are $q^2(q-1)(q-2)$ such branches. Now, there are no more cases to look at. Adding up all the branches of type $R_1$, we have a total of $q^2(q-1) + q(q-1) = q^3-q$ branches of type $R_1$. 

 \item A matrix of type $A_9$ has the following canonical form: $A = \begin{pmatrix}a&1&1&\\&a&&\\&&a&1\\&&&a\end{pmatrix}$. Then we have $\ZT{4}(A) = \left\{\begin{pmatrix}a_0&a_1&b&a_2\\&a_0&&c\\&&a_0 &b-c \\&&&a_0\end{pmatrix}\right\}$. Let $B = \begin{pmatrix}a_0&a_1&b&a_2\\&a_0&&c\\&&a_0&b-c\\&&&a_0\end{pmatrix}$, and $B' = \begin{pmatrix}a_0&a'_1&b''&a_2\\&a_0&&c'\\&&a_0&b'-c'\\&&&a_0\end{pmatrix} = XBX^{-1}$, where $X = \begin{pmatrix}x_0&x_1&y&x_2\\&x_0&&w\\&&x_0&y-w\\&&&x_0\end{pmatrix}$, with $x_0\neq 0$. So, $XB = B'X$ leaves us with $a'_1 = a_1$, $b'=b$, and $c'= c$, and the following equation:
 \begin{equation}\label{ETA9}
 x_0a_2 + (x_1-x_2)c = x_0a'_2 + (a_1-b)w
 \end{equation}
 
 \noindent{\bfseries When $a_1=b$ and $c = 0$}
 Here Equation~\ref{ETA9} ends up as $a'_2 = a_2$. $B$ is thus reduced to
 $\begin{pmatrix}
 a_0&a_1&a_1&a_2 \\
 &a_0&&\\
 &&a_0&a_1\\
 &&&a_0
 \end{pmatrix}$, and $\ZT{4}(A,B)= \ZT{4}(A)$. Thus $(A,B)$ is of type $A_9$, and there are $q^2(q-1)$ such branches.
 
 \noindent{\bfseries When $a_1\neq b$}: Here, in Equation~\ref{ETA9}, we choose $w$ such that $a'_2 = 0$. $B$ is thus reduced to 
 $\begin{pmatrix}
 a_0&a_1&b&\\
 &a_0&&c\\
 &&a_0&b-c\\
 &&&a_0
 \end{pmatrix}
 $, with $\ZT{4}(A,B) = \left\{\begin{pmatrix}x_0&x_1&y&x_2\\&x_0&& \frac{(x_1-y)}{a_1-b}c\\ &&x_0&y- \frac{(x_1-y)}{a_1-b}c\\&&&x_0\end{pmatrix} \right\}$. $(A,B)$ is therefore of type $R_1$, and there are $q^2(q-1)^2$ such branches.
 
 \noindent{\bfseries When $a_1 = b$, and $c\neq 0$:} In Equation~\ref{ETA9}, choose $x_1 $ or $y$ such that $a'_2 =0$. Thus, $B$ is reduced to
 $\begin{pmatrix}
  a_0&a_1&a_1&\\
  &a_0&&c\\
  &&a_0&a_1-c\\
  &&&a_0
  \end{pmatrix}
  $. So $\ZT{4}(A,B) = \left\{\begin{pmatrix}x_0&x_1&x_1&x_2\\& x_0&&w\\&&x_0&x_1-w\\&&&x_0\end{pmatrix} \right\}$. This $(A,B)$ too is of type $R_1$, and there are $q(q-1)^2$ such branches.
  
  With this, we have no other cases to look at. Thus, we have $q^3$ branches of type $A_9$, and $q(q-1)^2 + q^2(q-1)^2 = (q^2-q)(q^2-1)$ branches of type $R_1$.
 \end{enumerate}
\end{proof}

\subsection{Branching rules for type $B$}
Matrices of types $B1, B2, B3, B4, B5$ are in block form of the kind $A = \begin{pmatrix}C_1 & \\ & C_2\end{pmatrix}$, where $C_1 \in GT_{m_1}(\Fq)$, and $C_2 \in GT_{m_2}(\Fq)$, where $m_1 + m_2 = 4$. Thus, $Z_{GT_4(\Fq)}(A) = \left\{\begin{pmatrix}X_1 & \\ & X_2\end{pmatrix} \right\}$ where $X_1 \in Z_{GT_{m_1}}(C_1)$ and $X_2 \in Z_{GT_{m_2}}(C_2)$. Thus, the branches of $A$ are of the form $\begin{pmatrix}D_1& \\ & D_2\end{pmatrix}$, where $D_1$ is a branch of $C_1$, and $D_2$ is a branch of $C_2$. With this argument, we can prove the following proposition.
\begin{prop}
The branching rules are as follows:
\begin{enumerate}
\item For a matrix of type $B_1$, there are:
 \begin{center}
  \begin{tabular}{c|c||c|c}\hline
 Branch & No. of Branches & Branch & No. of Branches\\ \hline
 $B_1$ & $(q-1)^2$ & $R_3$ & $(q-1)^2$\\
 $B_5$ & $2(q-1)^2$ & $R_4$ & $2(q-1)^2(q-2)$\\ 
 $B_6$& $2(q-1)^2(q-2)$ & $R_5$&$(q-1)^2(q-2)^2$\\ \hline
  \end{tabular}
  \end{center}
\item For a matrix of type $B_2$, there are:
 \begin{center}
   \begin{tabular}{c|c||c|c}\hline
  Branch & No. of Branches & Branch & No. of Branches\\ \hline
  $B_2$ & $(q-1)^2$ & $R_2$ & $(q-1)^2$\\
  $B_3$ & $(q-1)^2$ & $R_4$ & $(q-1)^2(q-2)$\\ 
  $B_4$ & $(q-1)^2$ & $R_5$ &$(q-1)^2(q-2)(q-3)$\\ 
  $B_6$ & $(q-1)^2(q-2)$ & &\\\hline
   \end{tabular}
   \end{center}
\item  For a matrix of type $B_3$, there are $q(q-1)^2$ branches of type $B_3$, $q(q-1)^2$ branches of type $R_2$, and $q(q-1)^2(q-2)$ branches of type $R_4$.
\item For a matrix of type $B_4$, there are, $q(q-1)^2$ branches of type $B_4$, $(q^2-1)(q-1)$ branches of type $R_2$, and $q(q-1)^2(q-2)$ branches of type $R_4$.
\item For a matrix of type $B_5$, there are $q(q-1)^2 $ branches of type $B_5$, $q(q-1)^2$ branches of type $R_3$, and $q(q-1)^2(q-2)$ branches of type $R_4$.
\end{enumerate}
\end{prop}
Finally, 
\begin{prop}
For a matrix of type $B_6$, there are, $(q-1)^3$ branches of type $B_6$, $(q-1)^3$ branches of type $R_4$, and $(q-1)^3(q-2)$ branches of type $R_5$.
\end{prop}
\begin{proof}
 A matrix of type $B_6$ has the canonical form: $A = \begin{psmallmatrix} a&&&\\&a&&\\&&b&\\&&&c
 \end{psmallmatrix}$. Here, $\ZT{4}(A) = \left\{
 \begin{psmallmatrix}C&&\\&c_0&\\&&d_0
 \end{psmallmatrix} \mid C \in \Tq{2}, c_0,d_0 \neq 0 \right\}.$ Enumerating the conjugacy classes of $\Tq{2}$ gives us the branches mentioned.  
\end{proof}

\subsection{Branching Rules of the New Types}
While determining the branching rules of the existing types of conjugacy classes of $GT_4(\Fq)$, we came across six new types of simultaneous conjugacy classes of pairs of commuting matrices. We called them $tNT_1$, $tNT_2$, $tNT_3$, $tNT_4$,$tNT_5$, and $NR_1$. In this subsection, we shall focus on the branching rules of these new types.
\begin{prop}\label{BrtNT1}
A commuting tuple of type $tNT_1$ has $q^2(q-1)$ branches of type $tNT_1$, $q^2(q-1)$ branches of type $R_1$, and $q^2(q-1)(q-2)$ branches of type $R_3$.
\end{prop}
\begin{proof}
 For a commuting pair $(A,B)$ of matrices of type $tNT_1$, the centralizer is $\ZT{4}(A,B) = \left\{\begin{pmatrix}a_0&a_1&&a_3\\&a_0&&\\&&c_0&c_1\\&&&c_0\end{pmatrix}\mid a_0,c_0\neq 0 \right\}$. Let $C = \begin{pmatrix}a_0&a_1&&a_3\\&a_0&&\\&&c_0&c_1\\&&&c_0\end{pmatrix}$, and $C' = \begin{pmatrix}a_0&a'_1&&a'_3\\&a_0&&\\&&c_0&c'_1\\&&&c_0\end{pmatrix}= XCX^{-1}$ by $X = \begin{pmatrix}x_0&x_1&&x_3\\&x_0&&\\&&z_0&z_1\\&&&z_0\end{pmatrix}$. $XC = C'X$ leads us to $a'_1 = a_1$, $c'_1 = c_1$, and just one equation:
 \begin{equation}\label{ETNT1}
  x_0a_3 + x_3b_0 = z_0a'_3 + x_3a_0.
 \end{equation}

\noindent{\bfseries When $a_0=c_0$:} Here Equation~\ref{ETNT1} becomes $x_0a_3 = z_0a'_3$.

So, we have two cases over here: $a_3 = 0$, and $a_3\neq 0$.

When $a_3 = 0$, $C$ is reduced to $\begin{pmatrix}a_0&a_1&&\\&a_0&&\\&&a_0&c_1\\&&&a_0\end{pmatrix}$, with $\ZT{4}(A,B,C)= \ZT{4}(A,B)$. $(A,B,C)$ is of type $tNT_1$, and there are $q^2(q-1)$ such branches. 

When $a_3\neq 0$, we choose $z_0$ such that $a'_3 = 1$. Here, $C$ is reduced to $\begin{pmatrix}
a_0&a_1&&1\\
&a_0&&\\
&&a_0&c_1\\
&&&a_0
\end{pmatrix}$, with $\ZT{4}(A,B,C) = \left\{\begin{pmatrix}x_0&x_1&&x_3\\&x_0&&\\&&x_0&z_1\\&&&x_0\end{pmatrix}\right\}$. This  $(A,B,C)$ is of type $R_1$, and there are $q^2(q-1)$ such branches.

So now, with $a_0 = c_0$, we have no other cases left to analyse. We move on to the case of $a_0\neq c_0$.

\noindent{\bfseries When $a_0 \neq c_0$}: Here, in Equation~\ref{ETNT1}, we can choose $x_3$ so that $a'_3 = 0$. So $C$ is reduced to $\begin{pmatrix}
a_0&a_1&&\\
&a_0&&\\
&&c_0&c_1\\
&&&c_0
\end{pmatrix}$, with $\ZT{4}(A,B,C) = \left\{\begin{pmatrix}x_0&x_1&&\\&x_0&&\\&&z_0&z_1\\&&&z_0\end{pmatrix}\right\}$. This $(A,B,C)$ is of type $R_3$, and there are $q^2(q-1)(q-2)$ such branches.

So, with this, we have no other cases to look at.
\end{proof}
\begin{prop}\label{BrT4N2}
 The new type $tNT_2$ has $q^2(q-1)$ branches of type $tNT_2$, $q^2(q-1)(q-2)$ branches of type $R_2$, and $q^2(q^2-1)$ branches of type $NR_1$.
 \end{prop}

\begin{proof}
 For a commuting pair $(A,B)$ of type $tNT_2$, the centralizer is 
 
 $\ZT{4}(A,B) = \left\{\begin{pmatrix}a_0&&b_0&b_1\\ &a_0&b_2&b_3\\& & a_0 &\\ && & c_0\end{pmatrix}\mid \begin{matrix}a_0,b_0,b_1\\ b_2,b_3,c_0\in \Fq\end{matrix}\right\}$. Let $C =  \begin{pmatrix}a_0&&b_0&b_1\\ &a_0&b_2&b_3\\& & a_0 &\\ && & c_0\end{pmatrix}$, and $C' = \begin{pmatrix}a_0&&b'_0&b'_1\\ &a_0&b'_2&b'_3\\& & a_0 &\\ && & c_0\end{pmatrix} = XCX^{-1}$ for some $X = \begin{pmatrix}x_0&&y_0&y_1\\ &x_0&y_2&y_3\\& & x_0 &\\ && & z_0\end{pmatrix}$. So, equating $XC = C'X$ leads us to $b'_0 = b_0$, $b'_2 = b_2$, and the following equations:
 \begin{eqnarray}
  x_0b_1 + y_1c_0 &=& z_0b'_1 + y_1a_0\label{ETN21}\\
  x_0b_3 + y_3c_0 &=& z_0b'_3 + y_3a_0\label{ETN22}
 \end{eqnarray}

We have two main cases: $a_0 = c_0$, and $a_0 \neq c_0$:

\noindent{\bfseries When $a_0 = c_0$}:
Here, Equation~\ref{ETN21} becomes $x_0b_1 = z_0b'_1$, and Equation~\ref{ETN22} becomes $x_0b_3 = z_0b'_3$. 

When $b_1 = b_3 = 0$, $C$ is reduced to 
$\begin{pmatrix}
a_0&&b_0&\\
&a_0&b_2&\\
&&a_0&\\
&&&a_0
\end{pmatrix}$, with $\ZT{4}(A,B,C) = \ZT{4}(A,B)$. Thus $(A,B,C)$ is of type $tNT_2$, and there are $q^2(q-1)$ such branches.

When $b_1 \neq 0$. In Equation~\ref{ETN21}, choose $z_0$ such that $b'_1 = 1$. Then, on replacing $b_1$ and $b'_1$ by 1 in the same equation, we get $z_0 = x_0$. Hence, Equation~\ref{ETN22} becomes $x_0b_3 = x_0b'_3$, hence $b'_3 = b_3$. $C$ is reduced to 
$\begin{pmatrix}
a_0&&b_0&1\\
&a_0&b_2&b_3\\
&&a_0&\\
&&&a_0
\end{pmatrix}$, with $\ZT{4}(A,B,C) = \left\{ \begin{pmatrix}x_0&&y_0&y_1\\&x_0&y_2&y_3\\ &&x_0&\\&&&x_0\end{pmatrix}\right\}$. $(A,B,C)$ is of type $NR_1$. There are $q^3(q-1)$ such branches.

When $b_1 = 0$, and $b_3 \neq 0$. In Equation~\ref{ETN22}, choose $z_0$ so that $b'_3 = 1$. Thus $C$ is reduced to 
$\begin{pmatrix}
a_0&&b_0&\\
&a_0&b_2&1\\
&&a_0&\\
&&&a_0
\end{pmatrix}$, with $\ZT{4}(A,B,C) = \left\{\begin{pmatrix}x_0&&y_0&y_1\\&x_0&y_2&y_3\\&&x_0&\\&&&x_0\end{pmatrix}\right\}$. $(A,B,C)$ is of type $NR_1$. There are $q^2(q-1)$ such branches. We have exhausted all the cases under $a_0 = c_0$.

\noindent{\bfseries When $a_0 \neq c_0$}: Here, in Equation~\ref{ETN21}, choose $y_1$ so that $b'_1 = 0$, and in Equation~\ref{ETN22}, choose $y_3$ so that $b'_3 = 0$. $C$ is thus reduced to
$\begin{pmatrix}
a_0&&b_0&\\
&a_0&b_2&\\
&&a_0&\\&&&b_0
\end{pmatrix}
$, with $\ZT{4}(A,B,C)= \left\{\begin{pmatrix}x_0&&y_0&\\&x_0&y_2&\\&&x_0&\\&&&z_0\end{pmatrix}\right\}$. This $(A,B,C)$ is of type $R_2$, and there are $q^2(q-1)(q-2)$ such branches.

This leaves us with no further cases to analyse. Adding up the branches of type $NR_1$, we have a total of $q^2(q-1) + q^3(q-1)= q^2(q^2-1)$ branches of type $NR_1$.
\end{proof}
\begin{prop}\label{BRT4N3}
A commuting pair of type $tNT_3$ has $q^2(q-1)$ branches of type $tNT_3$, $q^2(q-1)$ branches of type $R_1$, $q^2(q-1)(q-2)$ branches of type $R_2$, and $q(q^2-1)$ branches of type $NR_1$.
\end{prop}

\begin{proof}
 Let $(A,B)$ be a pair of commuting matrices of type $tNT_3$. Their common centralizer is $\ZT{4}(A,B) = \left\{\begin{pmatrix}D&E\\&D_{11}I_2\end{pmatrix}\mid D \in T_2(\Fq), E\in M_2(\Fq)
 \right\}$. Let $C = \begin{pmatrix}D&E\\&a_0I_2\end{pmatrix}$, where $D = \begin{pmatrix}a_0&a_1\\&b_0\end{pmatrix}$ and $E = \begin{pmatrix}b_0&b_1\\ b_2&b_3\end{pmatrix}$. Let $C' = \begin{pmatrix}D'&E'\\&a_0I_2\end{pmatrix} = XCX^{-1}$, where $X = \begin{pmatrix}Z&Y\\&x_0I_2\end{pmatrix} \in \ZT{4}(A,B)$, where $Z = \begin{pmatrix}
 x_0&x_1\\ &z_0\end{pmatrix} \in GT_2(\Fq)$, and $Y = \begin{pmatrix}y_0&y_1\\ y_2&y_3\end{pmatrix}$.
 
 So $XC = C'X$ leaves us with the following $ZD = D'Z$. Thus $D$ can be taken to be a representative of a conjugacy class in $\Tq{2}$, and $Z \in \ZT{2}(D)$. We are therefore left with the following equation:
 \begin{equation*}
 ZE + a_0Y = DY + x_0E'
 \end{equation*}
 Exapanding this, we have:
 \begin{equation}\label{ETN31}
 \begin{pmatrix}x_0b_0 + x_1b_2 & x_0b_1+x_1b_3\\ z_0b_2&z_0b_3\end{pmatrix} + \begin{pmatrix}-a_1y_2&-a_1y_3\\(a_0-b_0)y_2&(a_0-b_0)y_3\end{pmatrix} = \begin{pmatrix}x_0b'_0&x_0b'_1\\ x_0b'_2& x_0b'_3\end{pmatrix}
 \end{equation}
 
 \noindent{When $D = a_0I_2$:} Here Equation~\ref{ETN31} becomes:
 \begin{equation*}
  \begin{pmatrix}x_0b_0 + x_1b_2 & x_0b_1+x_1b_2\\ z_0b_2&z_0b_3\end{pmatrix} +  = \begin{pmatrix}x_0b'_0&x_0b'_1\\ x_0b'_2& x_0b'_3\end{pmatrix}
  \end{equation*}
  
  When $b_2 = b_3 = 0$, we have $b'_0 = b_0$, and $b'_1 = b_1$. Thus, $C$ is reduced to
  $\begin{pmatrix}
  a_0&&b_0&b_1\\
  &a_0&&\\
  &&a_0&\\
  &&&a_0
  \end{pmatrix}$, and $\ZT{4}(A,B,C) = \ZT{4}(A,B)$. $(A,B,C)$ is of type $tNT_3$, and there are $q^2(q-1)$ such branches.
  
  When $b_2 \neq 0$, choose $z_0$ such that $b'_2 = 1$. Thus, on replacing $b_0$ by $b'_0 = 1$ in Equation~\ref{ETN31}, we get $z_0 = x_0$. Hence $b'_3 = b_3$. With these, Eqaution~\ref{ETN31} becomes 
  \begin{equation*}
    \begin{pmatrix}x_0b_0 + x_1& x_0b_1+x_1b_2\\ 1&b_3\end{pmatrix} +  = \begin{pmatrix}x_0b'_0&x_0b'_1\\ 1& b'_3\end{pmatrix}
    \end{equation*}

Choose $x_1$ so that $b'_0 = 0$. On replacing $b_0$ by $b'_0 = 0$ in the above equation, we have $x_1 = 0$. Thus $b'_1 = b_1$. So $C$ is reduced to 
$
\begin{pmatrix}
a_0&&&b_1\\
&a_0&1&b_3\\
&&a_0&\\
&&&a_0
\end{pmatrix}
$
with $\ZT{4}(A,B,C) = \left\{\begin{pmatrix}x_0&&y_0&y_1\\ &x_0&y_2&y_3\\ &&x_0&\\&&&x_0\end{pmatrix}\right\}$. $(A,B,C)$ is of type $NR_1$, and there are $q^2(q-1)$ such branches.

When $b_2 = 0$ and $b_3 \neq 0$. Choose $z_0$ so that $b'_3 = 1$. Equation~\ref{ETN31} becomes 
\begin{equation*}
    \begin{pmatrix}x_0b_0 & x_0b_1+x_1\\ 0&1\end{pmatrix} +  = \begin{pmatrix}x_0b'_0&x_0b'_1\\0 & 1\end{pmatrix}
    \end{equation*}

Hence, $b'_0 = b_0$, and choose $x_1$ so that $b'_1 = 0$. $C$ is reduced to 
$\begin{pmatrix}
a_0&&b_0&\\
&a_0&&1\\
&&a_0&\\
&&&a_0
\end{pmatrix}$, with $\ZT{4}(A,B,C) = \left\{\begin{pmatrix}x_0&&y_0&y_1\\ &x_0&y_2&y_3\\ &&x_0&\\&&&x_0\end{pmatrix}\right\}$. This $(A,B,C)$ too is of type $NR_1$, and there are $q(q-1)$ such branches.

With this, we have no other cases to analyse when $D = a_0I_2$.

\noindent{When $D = \begin{pmatrix}a_0&1\\&a_0\end{pmatrix}$:} Here $Z = \begin{pmatrix}x_0&x_1\\&x_0\end{pmatrix}$. Equation~\ref{ETN31} becomes:
\begin{equation*}
 \begin{pmatrix}x_0b_0 + x_1b_2 & x_0b_1+x_1b_3\\ x_0b_2&x_0b_3\end{pmatrix} + \begin{pmatrix}-y_2&-y_3\\0&0\end{pmatrix} = \begin{pmatrix}x_0b'_0&x_0b'_1\\ x_0b'_2& x_0b'_3\end{pmatrix}
 \end{equation*}
 We have from this $b'_2 = b_2$, $b'_3 = b_3$, and we can choose $y_2$ so that $b'_0 = 0$ and $y_3$ such that $b'_1 = 0$. Hence $C$ is reduced to 
 $
 \begin{pmatrix}
 a_0&1&&\\
 &a_0&b_2&b_3\\
 &&a_0&\\
 &&&a_0
 \end{pmatrix}
 $, with $\ZT{4}(A,B, C) =\left\{\begin{pmatrix}x_0&x_1&x_2&x_3\\&x_0&b_2x_1&b_3x_1\\&&x_0&\\&&&x_0\end{pmatrix} \right\}$. This $(A,B,C)$ is of type $R_1$, and there are $q^2(q-1)$ such branches.
 
 \noindent{\bfseries When $C = \begin{pmatrix}a_0&\\&c_0\end{pmatrix}$, $c_0\neq a_0$}: Here $Z = \begin{pmatrix}x_0&\\&z_0\end{pmatrix}$. Equation~\ref{ETN31} becomes:
 \begin{equation*}
  \begin{pmatrix}x_0b_0  & x_0b_1\\ z_0b_2&z_0b_3\end{pmatrix} + \begin{pmatrix}&\\(a_0-c_0)y_2&(a_0-c_0)y_3\end{pmatrix} = \begin{pmatrix}x_0b'_0&x_0b'_1\\ x_0b'_2& x_0b'_3\end{pmatrix}
  \end{equation*}
  We have $b'_0 = b_0$ and $b'_1 = b_1$. Choose $y_2$ and $y_3$ such that $b'_2 = b'_3 = 0$. $C$ is reduced to 
  $
  \begin{pmatrix}
  a_0&&b_0&b_1\\
  &c_0&&\\
  &&a_0&\\
  &&&a_0
  \end{pmatrix}
  $, and $\left\{ \begin{pmatrix}x_0&&y_0&y_1\\&z_0&&\\&&x_0&\\&&&x_0\end{pmatrix} \right\}$. Here $(A,B,C)$ is of type $R_2$, and there are $q^2(q-1)(q-2)$ such branches. 
  
  With this, we have no other cases to deal with. 
  
  Adding up the branches of type $NR_1$, we have a total of $q(q-1)+q^2(q-1)= q(q^2-1)$ branches of this type.
 \end{proof}
\begin{prop}\label{BRT4N4}
For a pair of commuting matrices of type $tNT_4$, there are $q^2(q-1)$ branches of type $tNT_4$, $q^2(q-1)^2$ branches of type $R_1$, and $q(q^2-1)(q-1)$ branches of type $NR_1$.
\end{prop}
\begin{proof}
 The centralizer of a commuting pair $(A,B)$ of this type is $$\ZT{4}(A,B) = \left\{\begin{pmatrix} \begin{smallmatrix}a_0&a_1\\&a_0\end{smallmatrix} & B_1 \\ & \begin{smallmatrix}a_0&a_1\\&a_0\end{smallmatrix}\end{pmatrix} \mid a_0 \neq 0, B_1 \in M_2(Fq)\right\}.$$
 This was seen, and proved in \cite[, Lemma 5.14]{Sh1} as the new type $NT_1$. 
\end{proof}
\begin{prop}
For a commuting pair of type $tNT_5$, there are $q^2(q-1)$ branches of type $tNT_5$, and $q(q^2-1)(q-1)$ branches of type $R_1$.
\end{prop}
\begin{proof}
 The centralizer of a commuting pair $(A,B)$ of type $tNT_5$ is:
 $$\ZT{4}(A,B) = \left\{\begin{pmatrix} a_0&a_1 & b_0&b_1\\&a_0&&b_0\\&& a_0 & c_1\\ & & & a_0\end{pmatrix} \mid a_0 \neq 0 \right\}.$$ 
 Let $C = \begin{pmatrix} a_0&a_1 & b_0&b_1\\&a_0&&b_0\\&& a_0 & c_1\\ & & & a_0\end{pmatrix}$, and $C' = \begin{pmatrix} a_0&a_1 & b'_0&b'_1\\&a_0&&b'_0\\&& a_0 & c_1\\ & & & a_0\end{pmatrix} = XCX^{-1}$, for some $X = \begin{pmatrix} x_0&x_1 & y_0&y_1\\&x_0&&y_0\\&& x_0 & z_1\\ & & & x_0\end{pmatrix}$. So $XC = C'X$ leads us to $b'_0 = b_0$, and the equation:
 \begin{equation}\label{ETN51}
  x_0b_1 + x_1b_0 + y_0c_1 = x_0b'_1 + z_1b_0 + y_0a_1.
 \end{equation}
We have two main cases: $a_1 = c_1$ and $a_1 \neq c_1$.

\noindent{\bfseries When $a_1 = c_1$:}
Equation~\ref{ETN51} becomes $x_0b_1+x_1b_0 = x_0b'_1 + z_1b_0$. 

When $b_0 = 0$, we have $b'_1 =b_1$. $C$ is reduced to 
$
\begin{pmatrix}
a_0&a_1&&b_1\\
&a_0&&\\
&&a_0&a_1\\
&&&a_0
\end{pmatrix}
$, with $\ZT{4}(A,B,C) = \ZT{4}(A,B)$. $(A,B,C)$ is thus of type $tNT_5$, and there are $q^2(q-1)$ such branches.

When $b_0\neq 0$, choose $z_1$ such that $b'_1 = 0$. $C$ is reduced to 
$
\begin{pmatrix}
a_0&a_1&b_0&\\
&a_0&&b_0\\
&&a_0&a_1\\
&&&a_0
\end{pmatrix}
$, with $\ZT{4}(A,B,C) = \left\{\begin{pmatrix}x_0&x_1&y_0&y_1\\ &x_0&&y_0\\&&x_0&x_1\\&&&x_0 \end{pmatrix}\right\}$. $(A,B,C)$ is of type $R_1$, and there are $q(q-1)^2$ such branches.

So, we have no other cases to look at for $a_1 = c_1$.

\noindent{\bfseries $a_1\neq c_1$:} In Equation~\ref{ETN51}, choose $y_0$ so that $b'_1 = 0$.
Thus, $C$ is reduced to 
$
\begin{pmatrix}
a_0 &a_1&b_0&\\
&a_0&&b_0\\
&&a_0&c_1\\
&&&a_0
\end{pmatrix}
$, with $\ZT{4}(A,B,C) =\left\{\begin{pmatrix}x_0&x_1&\frac{b_0(z_1-x_1)}{c_1-a_1}&\\&x_0&&\frac{b_0(z_1-x_1)}{c_1-a_1}\\&&x_0&z_1\\&&&x_0\end{pmatrix}\right\}$. Here $(A,B,C)$ is of type $R_1$, and there are $q^2(q-1)^2$ such branches. 

With this, we have no other case to look at. So, adding up the branches of type $R$, we have a total of $q(q-1)^2+q^2(q-1)^2 = q(q^2-1)(q-1)$ branches of type $R_1$.
 
\end{proof}

\begin{prop}
 For a commuting pair of type $NR_1$, there are $q^4(q-1)$ branches of type $NR_1$.
\end{prop}
\begin{proof}
The centralizer of a commuting pair $(A,B)$ of type $NR_1$ is 
$$Z_{GT_4(\Fq)}(A,B) = \left\{\begin{pmatrix}a_0I_2 & D\\ & a_0I_2\end{pmatrix}\mid a_0 \neq 0, D \in M_2(\Fq) \right\}.$$
The result follows, as this is a commutative subgroup.
\end{proof}

\section{Branching in $UT_3(q)$}
For the unitriangular group $UT_3(\Fq)$, the conjugacy classes are as follows:
\begin{center}
\begin{tabular}{|c|c|c|c|} \hline 
Canonical Form & No. of Classes & Centralizer & Name of Type\\ \hline
 $\begin{matrix}\begin{psmallmatrix} 1 & 0 & a \\ 0 & 1 & 0\\ 0 & 0& 1 \end{psmallmatrix},\\~a\in \Fq\end{matrix}$ & $q$ &$UT_3(\Fq)$& $C$ \\ \hline
 $\begin{matrix}\begin{psmallmatrix} 1 & a & 0 \\ 0 & 1 & 0\\ 0 & 0& 1 \end{psmallmatrix},\\~a\in \Fq^*.\end{matrix}$ & $(q-1)$ &$\left\{\begin{psmallmatrix}1 & x_0&x_1\\&1&\\&&1\end{psmallmatrix} \mid x_0,x_1 \in \Fq\right\}$ & $R_1$\\ \hline
$\begin{matrix}\begin{psmallmatrix} 1 & 0 & 0 \\ 0 & 1 & a\\ 0 & 0& 1 \end{psmallmatrix},\\~a\in \Fq^*.\end{matrix}$ & $(q-1)$ &$\left\{\begin{psmallmatrix}1 & &x_1\\&1&x_0\\&&1\end{psmallmatrix} \mid x_1, x_0 \in \Fq\right\}$ & $R_1$\\ \hline
$\begin{matrix}\begin{psmallmatrix} 1 & a & 0 \\ 0 & 1 & b\\ 0 & 0& 1 \end{psmallmatrix},\\~a,b\in \Fq^*.\end{matrix}$ & $(q-1)^2$ &$\left\{\begin{psmallmatrix} 1 & x_0&x_1\\&1&\frac{b}{a} x_0\\&&1\end{psmallmatrix} \mid x_0,x_1 \in \Fq \right\}$ & $R_2$\\ \hline
\end{tabular}
\end{center}
We see that there are two types here: central $C$ and regular $R$. Note that the centralizers of both regulars $R_1$ and $R_2$ are isomorphic (not conjugate). For the type $C$, the centralizer is full group $UT_3(\Fq)$, thus all types appear in the first column. For the regular type, it has $q^2$ branches of the same $R$ type, as the centralizer is commutative, of size $q^2$, hence the number of branches is $q^2$.

\begin{theorem}\label{TheoremUT3}
The branching matrix (with the order of type $C, R1$):
$$B_{UT_3(\Fq)} = \begin{pmatrix}
q & 0 \\
q^2-1 & q^2 
\end{pmatrix}.$$
\end{theorem}

 We prove the branching rules below.
\begin{prop}
An upper unitriangular matrix of type $C$ has $q$ branches of type $C$, and $q^2-1$ branches of the type $R$.
\end{prop}
\begin{proof}
 The result follows as matrices of this type are central.
\end{proof}

\begin{prop}
A matrix of of any of the $R$ types has $q^2$ branches of the same $R$ type.
\end{prop}
\begin{proof}
 A matrix of any of the $R$ types is a $\Reg$ type, hence its centralizer  in $UT_3(\Fq)$ is commutative, of size $q^2$, hence the number of branches is $q^2$.
\end{proof}

\section{Branching in $UT_4(q)$}\label{Unip4}

We shift our focus to commuting tuples of matrices in $UT_4(\Fq)$. The conjugacy classes according to the types of this group are listed in Appendix~\ref{CCUT45}.
\begin{theorem}\label{TheoremUT4}
The branching rules for the upper unitriangular group is given by the following matrix (with order $C, A_1,A_2, A_3, R_1, R_2$):
 \begin{displaymath}
  B_{UT_4(\Fq)} = \begin{pmatrix}
            q & 0 & 0 & 0 & 0 & 0  \\
          2(q-1) & q^2 & 0 & 0 & 0 & 0\\
          (q-1)^2 & 0 & q^2 & 0 & 0 & 0\\
           q(q^2-1)&0 &0 &q^2 & 0 & 0 \\
           q(q-1) & q(q^2-1) & q^2(q-1) & q(q^2-1) & q^4& 0\\
           (q^2-1)(q-1) & q^2(q-1) & q(q^2-1) & 0 & 0 & q^3
         \end{pmatrix}.
 \end{displaymath}
\end{theorem}
The first column corresponds to type $C$, thus all types of $UT_4(\Fq)$ appears. The last two columns are the regular types. There are no new types here. The proof for other columns is listed below in propositions. 

\begin{prop}\label{TU4A1}
An upper unitriangular matrix of type $A_1$ has $q^2$ branches of type $A_1$, and $q(q^2-1)$ branches of type $R_1$, and $q^2(q-1)$ branches of type $R_2$.
\end{prop}
\begin{proof}
 Let $A = \begin{pmatrix} 1&&& \\&1&&a\\&&1&\\&&&1\end{pmatrix}$, a matrix of type $A_1$. The centralizer $Z_{UT_{4}}(A)$ of $A$ is: $\left\{\begin{pmatrix}1&&x_1&x_2\\&1&y_0&y_1\\&&1&z_0\\ &&&1\end{pmatrix}\mid x_i,y_i,z_0\in \Fq\right\}$. Let $X = \begin{pmatrix}1&&x_1&x_2\\&1&y_0&y_1\\&&1&z_0\\ &&&1\end{pmatrix},$ be an element of $Z_{UT_{4}}(A)$. Let $B = \begin{pmatrix}1&&b_1&b_2\\&1&c_0&c_1\\&&1&d_0\\ &&&1\end{pmatrix}$, and $B^\prime = \begin{pmatrix}1&&b_1^\prime&b_2^\prime\\&1&c_0^\prime&c_1^\prime\\&&1&d_0^\prime\\ &&&1\end{pmatrix}$ be the conjugate of $B$ by $X$, i.e., $B' = XBX^{-1}$. Thus equating $XB = B'X$ leads us to $b'_0=b_0$, $c'_0 = c_0$, $c'_1 = c_1$, and  the following equations:
 \begin{center}
 \begin{tabular}{ccc}
 $x_0c_0+b_1=y_0b_0^\prime+b_1^\prime$\\
 $x_0c_1+b_2=y_1b_0^\prime+b_2^\prime$\\
 \end{tabular}
 \end{center}
 We use these to simplify $B$ to the branches mentioned in the statement of the proposition.
\end{proof}
\begin{prop}\label{TU4A2}
 An upper unitriangular matrix of type $A_2$ has $q^2$ branches of type $A_2$, and $q^2(q-1)$ branches of type $R_1$, and $q(q^2-1)$ branches of $R_2$.
\end{prop}
\begin{proof}
 Given $A = \begin{pmatrix} 1&&a& \\&1&&b\\&&1&\\&&&1\end{pmatrix}$ , where $a,b\in \Fq^*$. the canonical form of a matrix of type $A_2$. The centralizer of $A$, $Z_{UT_{4}}(A)$ is $\left\{\begin{pmatrix}1&x_0&x_1&x_2\\&1&y_0&y_1\\&&1&\lambda x_0\\ &&&1\end{pmatrix}\mid \lambda=\frac{b}{a}, x_i,y_i,z_0\in \Fq\right\}$. Let $X =  \begin{pmatrix}1&x_0&x_1&x_2\\&1&y_0&y_1\\&&1&\lambda x_0\\ &&&1\end{pmatrix}$ be an element of $Z_{UT_{4}}(A)$. Let $B = \begin{pmatrix}1&b_0&b_1&b_2\\&1&c_0&c_1\\&&1&\lambda b_0\\ &&&1\end{pmatrix}$, and $B^\prime = \begin{pmatrix}1&b_0^\prime&b_1^\prime&b_2^\prime\\&1&c_0^\prime&c_1^\prime\\&&1&\lambda b_0^\prime\\ &&&1\end{pmatrix}$ be the conjugate of $B$ by $X$. Thus equating $XB = B'X$ gives us the following equations:
 \begin{center}
  \begin{tabular}{ccc}
  $b_0=b_0^\prime$\\
 $c_0=c_0^\prime$\\
 $x_0c_0+b_1=y_0b_0^\prime+b_1^\prime$\\
 $\lambda b_0y_0+c_1=\lambda x_0c_0^\prime+c_1^\prime$\\
 $x_0c_1+\lambda b_0x_1+b_2=y_1b_0^\prime+\lambda b_1^\prime x_0b_2^\prime$\\
  \end{tabular}
 \end{center}
Using these we reduce $B$ to the mentioned branches. 
\end{proof}
\begin{prop}\label{TU4A3}
 An upper triangular matrix of type $A_3$ has $q^2$ branches of type $A_3$, and $q(q^2-1)$ branches of type $R_1$.
\end{prop}
\begin{proof}
 One of the canonical forms of an upper triangular matrix of type $A_3$ is $A = \begin{pmatrix} 1&a&& \\&1&&\\&&1&\\&&&1\end{pmatrix}$, where $a\in \Fq^*$. Here $\ZU{4}(A) = \left\{\begin{pmatrix}1&x_0&x_1&x_2\\&1&&\\&&1&z_0\\ &&&1\end{pmatrix}\mid x_i,z_0\in \Fq\right\}$. Let $X = \begin{pmatrix}1&x_0&x_1&x_2\\&1&&\\&&1&z_0\\ &&&1\end{pmatrix},$ be an element of $Z_{UT_{4}}(A)$. Let $B = \begin{pmatrix}1&b_0&b_1&b_2\\&1&&\\&&1&d_0\\ &&&1\end{pmatrix}$, and $B^\prime = \begin{pmatrix}1&b_0^\prime&b_1^\prime&b_2^\prime\\&1&&\\&&1&d_0^\prime\\ &&&1\end{pmatrix}$ be the conjugate of $B$ by $X$, i.e., $B' = XBX^{-1}$. Thus equating $XB = B'X$ leads us to the $b'_0 = b_0$, $b'_1 = b_1$, $d'_0 = d_0$, and the following equation:
 \begin{center}
 \begin{tabular}{ccc}
 $x_1d_0+b_2=z_0b_1^\prime+b_2^\prime$\\
 \end{tabular}
 \end{center}
 We use these to simplify $B$ to the branches mentioned in the statement of the proposition.
\end{proof}
\begin{prop}\label{TU4R}
 A matrix of the $R_1$ type has $q^4$ branches of type $R_1$ and A matrix of the $R_2$ type has $q^3$ branches of type $R_2.$
\end{prop}
\begin{proof}
 The type $R_1$ and $R_2$ are $\Reg$ types, hence the centralizer of matrices of such a type is a commutative.
\end{proof}
\begin{proof}[Proof of Theorem~\ref{TheoremUT4}]
 From the data in Propositions~\ref{TU4A1} to~\ref{TU4R}, we summarize the branching rules for $UT_4$, as in the table described in the theorem.
\end{proof}

Here are some isomorphisms between centralizers of matrices of the same $z$-class for some $z$-classes in $UT_4(\Fq)$.

\begin{prop}\label{Isomorphism_in_type_A1}
 The centralizer of conjugacy classes with representative $\begin{pmatrix} 1&&a& \\&1&&\\&&1&\\&&&1\end{pmatrix}$ and  $\begin{pmatrix} 1&&& \\&1&&a\\&&1&\\&&&1\end{pmatrix},$ for $a\in \Fq^*$ are isomorphic.
\end{prop}

\begin{proof}
 The centralizer of conjugacy class with representative $\begin{pmatrix} 1&&a& \\&1&&\\&&1&\\&&&1\end{pmatrix}$ is

 $\left\{\begin{pmatrix} 1&x_0&x_1&x_2 \\&1&y_0&y_1\\&&1&\\&&&1\end{pmatrix} \mid x_i, y_i\in \Fq \right\}.$
 
 The centralizer of conjugacy class with representative  $\begin{pmatrix} 1&&& \\&1&&a\\&&1&\\&&&1\end{pmatrix}$ is 
 
 $\left\{  \begin{pmatrix} 1&&x_1&x_2 \\&1&y_0&y_1\\&&1&z_0\\&&&1\end{pmatrix} \mid x_i, y_i, z_0\in \Fq \right\}.$

 The following map gives isomorphism between these two centralizers.

 $ \begin{pmatrix} 1&&x_1&x_2 \\&1&y_0&y_1\\&&1&z_0\\&&&1\end{pmatrix} \mapsto \begin{pmatrix} 1&-z_0&y_1-z_0y_0&x_2-x_1z_0 \\&1&y_0&x_1\\&&1&\\&&&1\end{pmatrix}$
\end{proof}

\begin{prop}\label{Isomorphism_in_type_A3}
 The centralizers of all conjugacy classes of type $A_3$ are isomorphic.
\end{prop}

\begin{proof}
 There are six conjugacy classes of type $A_3.$ In the following table, we give the centralizer of these conjugacy classes. We also set a notation for these conjugcay classes which will be used later in this proof.

 \begin{tabular}{|c|c|c|}\hline
 Class Representative & Centralizer in $UT_4(\Fq)$ & Name of Conjugacy class\\ \hline
 $\begin{matrix}\begin{psmallmatrix} 1&a&&\\&1&&\\&&1&\\&&&1
 \end{psmallmatrix}, a \in \Fq^* \end{matrix}$ & $\left\{\left(\begin{smallmatrix}1 & x_0&x_1&x_2\\&1&&\\&&1&z_0\\&&&1\end{smallmatrix}\right)\mid x_i,z_0 \in \Fq\right\}$  & $A_{3_1}$ \\ \hline
 $\begin{matrix}\begin{psmallmatrix} 1&&&\\&1&&\\&&1&a\\&&&1
 \end{psmallmatrix}, a \in \Fq^* \end{matrix}$ & $\left\{\left(\begin{smallmatrix}1 & x_0&&x_2\\&1&&y_1\\&&1&z_0\\&&&1\end{smallmatrix}\right)\mid x_i,y_1,z_0 \in \Fq\right\}$  & $A_{3_2}$ \\ \hline
 $\begin{matrix}\begin{psmallmatrix} 1&a&&\\&1&&\\&&1&b\\&&&1
 \end{psmallmatrix}, a,b \in \Fq^* \end{matrix}$ & $\left\{\left(\begin{smallmatrix}1 & x_0&\frac{a}{b}y_1&x_2\\&1&&y_1\\&&1&z_0\\&&&1\end{smallmatrix}\right)\mid x_i,y_1,z_0 \in \Fq\right\}$  & $A_{3_3}$ \\ \hline
 \end{tabular}
 
 \begin{tabular}{|c|c|c|}\hline
 Class Representative & Centralizer in $UT_4(\Fq)$ & Name of Conjugacy class\\ \hline
 $\begin{matrix}\begin{psmallmatrix} 1&a&&\\&1&&b\\&&1&\\&&&1
 \end{psmallmatrix}, a,b\in \Fq^* \end{matrix}$ & $\left\{\left(\begin{smallmatrix}1 & x_0&x_1&x_2\\&1&&\frac{b}{a}x_0\\&&1&z_0\\&&&1\end{smallmatrix}\right)\mid x_i,z_0 \in \Fq\right\}$  & $A_{3_4}$ \\ \hline
 $\begin{matrix}\begin{psmallmatrix} 1&&a&\\&1&&\\&&1&b\\&&&1
 \end{psmallmatrix}, a,b \in \Fq^* \end{matrix}$ & $\left\{\left(\begin{smallmatrix}1 & x_0&x_1&x_2\\&1&&y_1\\&&1&\frac{b}{a}x_1\\&&&1\end{smallmatrix}\right)\mid x_i,y_1\in \Fq\right\}$  & $A_{3_5}$ \\ \hline
 $\begin{matrix}\begin{psmallmatrix} 1&a&b&\\ &1&&\\ &&1&c\\ &&&1
 \end{psmallmatrix}, a,b,c \in \Fq^* \end{matrix}$ & $\left\{\left(\begin{smallmatrix}1 & x_0&x_1&x_2\\&1&&y_1\\&&1&\frac{c}{b}x_1-\frac{a}{b}y_1\\&&&1\end{smallmatrix}\right)\mid x_i,y_1 \in \Fq\right\}$  & $A_{3_6}$ \\ \hline
 \end{tabular}
 \begin{enumerate}
  \item The following map gives isomorphism between centralizers of representative of conjugacy classes $A_{3_1}$ and $A_{3_2}.$ 
  
 $ \begin{pmatrix} 1&x_0&&x_2 \\&1&&y_1\\&&1&z_0\\&&&1\end{pmatrix} \mapsto \begin{pmatrix} 1&z_0&y_1&x_2-y_1x_0 \\&1&&\\&&1&-x_0\\&&&1\end{pmatrix}$
 \item The following map gives isomorphism between centralizers of representative of conjugacy classes $A_{3_1}$ and $A_{3_4}.$ 
 
 $ \begin{pmatrix} 1&x_0&x_1&x_2 \\&1&&\\&&1&z_0\\&&&1\end{pmatrix} \mapsto \begin{pmatrix} 1&x_0&x_1&x_2-\left(\frac{x_0(x_0-1)}{2}\right)\lambda \\&1&&\lambda x_0\\&&1&z_0\\&&&1\end{pmatrix}$
 
 \item The following map gives isomorphism between centralizers of representative of conjugacy classes $A_{3_2}$ and $A_{3_5}.$
 
 $ \begin{pmatrix} 1&x_0&&x_2 \\&1&&y_1\\&&1&z_0\\&&&1\end{pmatrix} \mapsto \begin{pmatrix} 1&x_0&\lambda z_0&x_2+\left( \frac{z_0(z_0-1)}{2}\right)\lambda \\&1&&y_1\\&&1&z_0\\&&&1\end{pmatrix}$
 
  \item The following map gives isomorphism between centralizers of representative of conjugacy classes $A_{3_2}$ and $A_{3_3}.$ 
  
 $ \begin{pmatrix} 1&x_0&&x_2 \\&1&&y_1\\&&1&z_0\\&&&1\end{pmatrix} \mapsto \begin{pmatrix} 1&x_0+\lambda z_0&\lambda y_1&x_2+\lambda y_1z_0 \\&1&&y_1\\&&1&z_0\\&&&1\end{pmatrix}$
 
  \item The following map gives isomorphism between centralizers of representative of conjugacy classes $A_{3_2}$ and $A_{3_6}.$

 $ \begin{pmatrix} 1&x_0&&x_2 \\&1&&y_1\\&&1&z_0\\&&&1\end{pmatrix} \mapsto \begin{pmatrix} 1&x_0+\lambda_2z_0&\lambda_1 z_0+\lambda_2y_1&x_2+\lambda_2y_1z_0+\left( \frac{z_0(z_0-1)}{2}\right)\lambda_1 \\&1&&y_1\\&&1&z_0\\&&&1\end{pmatrix}$
 
 \end{enumerate}
\end{proof}

\section{Branching rules for $UT_5(\Fq)$}
In this section, we will discuss the simultaneous conjugacy classes of tuples of commuting matrices of $UT_{5}(\Fq)$. The types are listed in Section~\ref{CCUT45}. 
The branching matrix is as follows:   
\begin{theorem}\label{TheoremUT5}
The branching rule of $UT_5(\Fq)$ has $3$ new types. The branching matrix $B_{UT_5(\Fq)}$ is in table~\ref{TUT5} which is a $20\times 20$ matrix. 
\end{theorem}
Once again it's easy to see the branches for central and regular types.

\begin{landscape}
\begin{table}
\caption{Branching matrix of $UT_5(\Fq)$}
\label{TUT5}
\begin{center}
 $$    \left(\begin{smallmatrix}C & A_1& A_2 &A_3&A_4&A_5&B_1&B_2&B_3&B_4&B_5&B_6&D_1&D_2&R_1&R_2&R_3& UNT_1& UNT_2& UNT_3\\ 
      &&&&&&&&&&&&&&&&&&& \\                  
\hline
     &&&&&&&&&&&&&&&&&&& \\          
 q & 0& 0 & 0 & 0 & 0 & 0 & 0 & 0 & 0 & 0 & 0 & 0 & 0 & 0 & 0 & 0 & 0 & 0 & 0\\
 2(q-1) & q^2 & 0 & 0 & 0 & 0 & 0 & 0 & 0 & 0 & 0 & 0 & 0 & 0 & 0 & 0 & 0 & 0 & 0 & 0\\
 q^2-q & q(q^2-1)& q^4 &0 &0 &0 &3q^2-3q &0 &0 &0 &0 &0 &0 &0 &0 &0 &0 & 0& 0& 0\\
 2q^2-2q & 0& 0 &q^2& 0& 0& 0 &0 &0 &0 &0 &0 &0 &0 &0 &0 &0 & 0& 0& 0\\
 2q^2-2q & 2q^2(q-1) & 0 &q(q^2-1)& q^4& 0& 0& q^3-q& 0& 0& 0& 0& 0& 0& 0& 0& 0& 0& 0& 0\\
 \begin{smallmatrix}(q^2-1).\\(2q-1)\end{smallmatrix} & 0& 0 &0 &0 &q^2& 0& 0& 0& 0& 0& 0& 0& 0& 0& 0& 0& 0& 0& 0\\
 (q-1)^2 & 0& 0 & 0 & 0 & 0 & q^2 & 0 & 0 & 0 & 0 & 0 & 0 & 0 & 0 & 0 & 0 & 0 & 0 & 0\\
 2q^2-2q & 0& 0 &0 & 0& 0 &0 & q^2 &0 & 0& 0 & 0 & 0 & 0 & 0 & 0 & 0 & 0 &0 & 0\\
 2(q-1)^2 & q^2(q-1)& 0 &q^2(q-1)& 0 & 0 & 0 & 0 & q^3 &0 & 0 &  0 & 0 & 0 & 0 & 0 & 0 & 0& 0& 0\\
 \begin{smallmatrix}(2q^2+4).\\(q-1)^2\end{smallmatrix} & \begin{smallmatrix} q(q-1).\\(q^3+q^2-1)\end{smallmatrix}& 0 & 0 & 0 & 2q(q-1) & 0 & q^2-q & 0 & q^3 & 0 & 0 & 2q(q-1) & 0 & 0 & 0 & 0 & 0& 0& 0\\
 q(q-1)^2 & 0& 0 &q^3(q-1) &0 &  0 & 0 & 0 & 0 & 0 & q^2& 0 & 0 & 0 &0  & 0 & 0 & 0& 0& 0\\
 2q(q-1)^2 & 0& 0 & 0 & 0 &q^2(q-1) &q^2(q-1) &q^4-q^3 &0 &0 & \begin{smallmatrix} (q^3+q).\\(q^2-1)\end{smallmatrix} &q^3 &0 & 0 & 0 & 0 &0 & 0& 0& 0\\
 (q-1)^3 & 0& 0 & 0&  0&  0& 0  & 0&  0&  0 & 0 & 0& q^2& 0&  0& 0  &0  &  0 & 0& 0\\
 \begin{smallmatrix}(2q+1).\\(q-1)^3\end{smallmatrix} & 0& 0 &0 & 0 & 0 & 0 & 0 & 0 &0  &0 &  0&  0& q^3&  0&  0 & 0 & 0& 0& 0\\
 {\color{blue}2(q-1)^2} & {\color{blue}2q^3-2q^2}& {\color{blue}2q^4-2q^2} &{\color{blue}\begin{smallmatrix} q(q-1).\\(q^2+q-1)\end{smallmatrix}}&0&0& {\color{blue}2q^3-4q+2}& {\color{blue}\begin{smallmatrix}(q^2-q).\\(q^2+q-1)\end{smallmatrix}}& {\color{blue}\begin{smallmatrix} q^2(q-1).\\(q^2+q-1)\end{smallmatrix}}&0 &0 & 0 & 0 & 0& {\color{blue}q^6}& 0&  0& 0& {\color{blue}q^5-q^2}& 0\\
 {\color{blue}q(q-1)^2} & {\color{blue}\begin{smallmatrix} q(q-1)^2.\\(q+1)\end{smallmatrix}}& {\color{blue}q(q^2-1)^2}&{\color{blue}\begin{smallmatrix}q(q^2-1).\\(q-1)\end{smallmatrix}}&{\color{blue}q^4(q-1)} &{\color{blue}q^2(q-1)} &{\color{blue}\begin{smallmatrix} q(q-1)^2.\\(q+2)\end{smallmatrix}} &{\color{blue}\begin{smallmatrix}(q-1).\\(q^3-q)\end{smallmatrix}}&0 &{\color{blue}q^4-q^2} &0 &0 &0 &0 &0 &{\color{blue}q^5} &0 & {\color{blue}q^4-q^2}& {\color{blue}q^4-q^3}& {\color{blue}q^4-q^2}\\
 {\color{blue}\begin{smallmatrix}(q^2-1).\\(q-1)^2\end{smallmatrix}} & {\color{blue}q^2(q-1)^2} & 0 &{\color{blue}q^2(q-1)^2} &{\color{blue}q^3(q^2-1)} &{\color{blue}\begin{smallmatrix} q(q-1).\\(q^2-1)\end{smallmatrix}}& {\color{blue}\begin{smallmatrix} (q-1)^2.\\(q^2+q+1)\end{smallmatrix}}&{\color{blue}q^2(q-1)^2} &{\color{blue}q^3(q-1)} &0 &{\color{blue}q^4-q^3} &{\color{blue}q^2(q^2-1)} &{\color{blue}q^2(q^2-1)} &{\color{blue}q^2(q^2-1)} &0 &0 &{\color{blue}q^4} & {\color{blue}q^4-q^3} & {\color{blue}q^4-q^3}& {\color{blue}q^4-q^3}\\
0& 0& 0 &{\color{red}q^2(q-1)}& 0 &{\color{red}q(q-1)^2} &{\color{red}(q-1)^2} &{\color{red}q(q-1)^2} &0 &0 &0 &0 &0 &0 &0 &0 &0 & {\color{red}q^3}& 0 & 0\\
0 & 0& 0 &0 &0 &0 & {\color{red}2q(q-1)} &{\color{red}q^3-q^2} &0 &0 &0 &0 &0 &0 &0 & 0& 0& 0& {\color{red}q^3}& 0\\
0 & 0& 0 &0 &0 &0 & {\color{red}(q-1)^3} &0 &0 &0 &0 &0 & {\color{red}q(q-1)^2} &0 &0 &0 & 0& 0& 0& {\color{red}q^3}\\
\end{smallmatrix}\right) $$
\end{center}
\end{table}
\end{landscape}

\subsection{Branching of type $A$}
\begin{prop}
 An upper unitriangular matrix of type $A_1$ has the following branches:
 \begin{center}
\begin{tabular}{cc|cc}\hline
  Branch & No. of Branches & Branch & No. of Branches\\ \hline
  $A_1$ & $q^2$ & $B_4$ & $q(q-1)(q^3+q^2-1)$\\
  $A_2$ & $q(q^2-1)$ & $R_1$ & $2q^2(q-1)$ \\ 
  $A_4$ & $2q^2(q-1)$ & $R_2$ & $q(q-1)^2(q+1)$ \\
  $B_3$ & $q^2(q-1)$ & $R_3$ & $q^2(q-1)^2$\\\hline
  \end{tabular}
 \end{center}
\end{prop}
\begin{proof}
 For a matrix of type $A_1$, there are two canonical forms: $I_5 + aE_{14}$, and $I_5 + aE_{25}$, where $a \neq 0$. We will take our matrix $A$ of type $A_1$, to be the canonical form $I_5 + aE_{14}$, $a \neq 0$. So the centralizer of $A$ is $\ZU{5}(A) = \left\{ %
$, be such that $XB = B'X$. From $XB = B'X$, we get that $a'_1 = a_1$. Let $C$ denote the middle $3\times 3$ unitriangular block $%
$ in the matrix $B$, and let $Z$ denote the middle block, $%
$, from $X$. Likewise, we have $C'$. We see that from $XB = B'X$, we have $ZC = C'Z$. Thus we take $C$ to be a conjugacy class representative from $UT_3(\Fq)$, and $Z$ to be its centralizer element in $UT_3(\Fq)$. Now, with this, we have the following set of equations:
 \begin{eqnarray}
  %
 \label{EA12}\\ 
  x_1d_1 + y_1d_2 + a_2 &=& a_1w_1 + b'_1w_2 + a'_2\label{EA13}
 \end{eqnarray}
We look at two main cases, $a_1 =0$, and $a_1 \neq 0$.

\noindent {\bfseries Case $a_1 = 0$:} Here Equation~\ref{EA13} is reduced to $x_1d_1 +y_1d_2 + a_2 = b'_1w_2 + a'_2$. Here we look at subcases: 

\noindent{\bfseries When $(b_1, b_2) = (d_1, d_2) = (0,0)$:} Thus Equations~\ref{EA11} and~\ref{EA12} become:
\begin{eqnarray*}
  %
  
  \end{eqnarray*}
  and $a'_2 = a_2$. 
  
  {\bfseries When $C = I_3$:} Equations~\ref{EA11} and \ref{EA12} are void, and $B$ is reduced to $%
$. Thus $\ZU{5}(A,B) = \ZU{5}(A)$. So $(A,B)$ is a branch of type $A_1$, and there are $q$ branches.
  
  {\bfseries When $C = \left(%
\right)$, $c \neq 0$:} Equation~\ref{EA12} remains void, but from Equation~\ref{EA11}, we get $cx_1 + y_2 = y_2$, which leaves us with $x_1 = 0$, as $c\neq 0$. Thus the branch is $B =%
 \right\} $, which is the centralizer of one of the canonical forms of type $A_2$. So $(A,B)$ is a branch of type $A_2$, and there are $q(q-1)$ such branches.
  
  {\bfseries When $C = \left(%
\right),~c\neq 0$:} Here we have $Z = \left(%
\right)$. From Equations~\ref{EA11} and~\ref{EA12}, we have $cx_1 +y_1 = y_1$ and $w_1 + cw_2 = w_1$, thus we have $x_1 = w_2 = 0$. So we have $B = %
\right\}$, and by a routine check, we see that $\ZU{5}(A,B)$ is commutative, of size $q^6$. $(A,B)$ is of the regular type $R_1$, and there are $(q-1)q$ branches of this type.
  
  {\bfseries When $C = \left(%
\right)$. In this case Equation~\ref{EA12} becomes void, and from Equation~\ref{EA11}, we have $cy_1 + y_2 = y_2$, thus leading to $y_1 = 0$. Hence, $B = %
\right\}$, which is the centralizer of a unitriangular matrix of type $A_4$. So $(A,B)$ is a branch of type $A_4$, and there are $q(q-1)$ branches.
  
  {\bfseries When $C = \left(%
\right)$, where $\lbd = c_2/c_1$. From Equation~\ref{EA11}, we have $c_1x_1 + y_1 = y_1$, which leaves us with $x_1 = 0$, and then we have $c_2y_1 + y_2 = y_2$, which leaves us with $y_1 = 0$. Then, from Equation~\ref{EA12}, we have $w_1 + c_1w_2 = w_1$, leaving us with with $w_2 = 0$. So, we have $B = %
\right\}$. This is of size $q^5$, and by a routine check, it can be seen that $\ZU{5}(A,B)$ is commutative. Thus $(A,B)$ is of the regular type $R_2$, and there are $q(q-1)^2$ branches.
  
\noindent {\bfseries When $((b_1,b_2), (d_1,d_2)) \neq ((0,0), (0,0))$:} We shall start with $C = I_3$.

{\bfseries When $C = I_3$:} Here $Z$ is any aribtrary matrix in $UT_3(\Fq)$, and Equations~\ref{EA11} and \ref{EA12} become:
\begin{eqnarray*}
 %
  
\end{eqnarray*}

From the above equation, we have $b'_1 = b_1$, and $b_2 = b'_2 + b_1 z_3$, and we have $d_1 + z_1d_2 = d'_1$, and $d'_2 = d_2$.

Firstly, if both $b_1\neq0$ and $d_2 \neq 0$. Then we can choose a $z_3$ such that $b'_2 = 0$, and similarly we can choose $z_1$ such that $d'_1 = 0$. Hence, with this Equation~\ref{EA13} is reduced to $y_1d_2 + a_2 = b_1w_2 + a'_2$. We may choose a $w_2$ such that $d'_2 = 0$. Thus, we have reduced $B$ to $%
\mid \lbd = d_2/b_1\right\}.$ This is 
the centralizer of a matrix of type $A_4$. Thus, we have $(q-1)^2$ branches of this type. 

When $b_1 \neq 0$ and $d_2 = 0$, we again pick a $z_3$ such that $b'_2 = 0$, and Equation~\ref{EA13} is reduced to $x_1d_1 + a_2 = b_1w_2 + a'_2$. Again, choose $w_2$ so that $a'_2 = 0$. Thus $B$ is reduced to $%
\mid \lbd = d_1/b_1\right\}$, which is isomorphic (conjugation by the matrix that swaps the 4th and 5th rows and columns) to centralizer of a matrix of type $B_3$. Thus there are $q(q-1)$ branches of this type. 

When $b_1 = 0$, we have $b'2_2 = b_2$. We consider $d_2 \neq 0$, and choose a suitable $z_1$ so that $d'_1 = 0$. Equation~\ref{EA13} is reduced to $y_1d_2 +a_2 = a'_2$. Thus, we choose an apporpriate $y_1$ so that $a'_2 = 0$. $B$ is thus reduced to $%
\mid \lbd = d_1/b_1\right\}$, which is the centralizer of one of the canonical forms of type $A_4$. There are $q(q-1)$ such branches.

When $b_2 \neq 0$, $d_2 = 0$, we have $d'_1 = d_1$. We first take $d_1 = 0$. Then Equation~\ref{EA13} is reduced to $a'_2 = a_2$. We thus have $B$ reduced to $%
$, and thus $\ZU{5}(A,B) = \ZU{5}(A)$. Hence, $(A,B)$ is a branch of type $A_1$, and there are $q(q-1)$ branches. 

When $b_1 = 0$, with $d_2 =0$, and $d_1 \neq 0$. Equation~\ref{EA13} is reduced to $x_1d_1 + a_2 = a'_2$. With a suitable $x_1$, we can get rid of $d_1$. Hence $B$ is reduced to $%
 \right\}$. Thus $(A,B)$ is of type $A_2$, and there are $q(q-1)$ such branches.

{\bfseries When $C = \left(%
\right)$:} Equation~\ref{EA11} is reduced to $%
$. Thus, we have $b'_1 = b_1$, and we can choose $x_1$ such that $b'_2 = 0$. Now, here, on replacing $b'_2$ and $b_2$ by 0 in the above equation, we get that $x_1 = \frac{b_1}{c}z_3$. From Equation~\ref{EA12}, we have $d'_2 = d_2$, and $d'_1 = d_1 + z_1d_2$. Equation~\ref{EA13} becomes $\frac{b_1}{c}z_3d_1 + y_1d_2 + a_2 = w_2b_1 + a'_2$.

We now look at the case when $b_1 \neq 0$, and $d'_2 \neq 0$. We choose $z_1$ so that $d'_1 = 0$, and $w_2$ such that $a'_2 = 0$. Hence, we reduce $B$ to $%
$, and we have $\ZU{5}(A,B) = \left\{ %
 \mid \lbd = \frac{b_1}{c}, \mu = \frac{d_2}{b_1} \right\}$, which is isomorphic to the centralizer of some canonical matrix of type $B_4$. There are $(q-1)^3$ such branches.
 
 When $b_1 \neq 0$, and $d_2 = 0$, then $d'_1 = d_1$. Equation~\ref{EA13} becomes $\frac{b_1}{c}z_3d_1 + a_2 = w_2b_1 + a'_2$. Choose a suitable $w_2$, to make $a'_2 =0$. Then $B$ is reduced to $%
 \mid \lbd = \frac{b_1}{c}, \mu = \frac{d_1}{c}\right\}$. If we write $z_3$ in terms of $z_1$, then $\ZU{5}$ will be this: $\left\{%
 \mid \lbd = \frac{c}{b_1}, \mu = \frac{d_1}{b_1}\right\}$. If we conjugate this centralizer by the matrix $I + \frac{\mu}{\lbd}E_{45}$, we get the centralizer of a canonical unitriangular matrix of type $B_3$. Thus $(A,B)$ is a branch of type $B_3$, and there are $q(q-1)^2$ such branches.
 
 Now, when $b_1 = 0$, and $(d_1,d_2) \neq (0,0)$. We have $x_1 = \frac{b_1}{c}z_3 = 0$,a nd Equation~\ref{EA13} becomes $y_1d_2 + a_2 = a'_2$. First, when $d_2 \neq 0$, then we choose $z_1$ so that $d'_1 = 0$, and choose $y_1$ so that $a'_2 = 0$. So, $B$ is reduced to $%
 \right\}$, which is commutative of size $q^6$, $(A,B)$ is of regular type $R_1$, and there are $q(q-1)$ such branches.

 When $b_1 = d_2 =0$, $d_1 \neq 0$. We have $d'_1 = d_1$, and Equation~\ref{EA13} reduces to $a'_2 = a_2$. Thus, $B$ is reduced to $%
 \right\} $, which is the centralizer of a matrix of type $A_2$. Thus $(A,B)$ is of type $A_2$, and there are $q(q-1)^2$ such branches.
 
 {\bfseries When $C = \left(%
.$$ 
 Using a nice $x_1$, we can make $b'_1 = 0$, and $b'_2 = b_2$. So, if we replace $b_1$ by $b'_1 = 0$ in the above equation, we have $x_1 = 0$. Next, Equation~\ref{EA12} becomes:
 $$%
. $$ As $c\neq 0$, we choose a $w_2$ so that $d'_1 = 0$. We have $d'_2 = d_2$. With these, Equation~\ref{EA13} becomes
 \begin{equation}
  y_1d_2 +a_2 = a'_2
 \end{equation}
When $d_2 \neq 0$, choose $y_1$ such that $a'_2 = 0$. $B$ is reduced to $%
 \mid \lbd = \frac{d_2}{c} \right\}$. Thus $(A,B)$ is of regular type $R_2$, and there are $q(q-1)^2$ branches of this type.

When $d_2 = 0$, then we are left with $b_2 \neq 0$. Hence Equation~\ref{EA13} is reduced to $a'_2=a_2$. Hence $B$ is reduced to $%
  \right\}$, which is the centralizer of a matrix of type $R_1$. $(A,B)$ is a branch of type $R_1$, and there are $(q-1)^2q$ such branches.

{\bfseries When $C =  \left(%
$. With these, Equation~\ref{EA11} becomes:
\begin{eqnarray*}
 %
\end{eqnarray*}
So, we have $b'_1 = b_1$, and we can choose $y_1$ so that $b'_2 = 0$. Thus, on equating the above equation, with $b_2$ replaced by $0$, we get that $y_1 = \frac{b_1}{c}z_3$; and from Equation~\ref{EA12}, we have $d'_1 = d_1$, and $d'_2 = d_2$, and thus Equation~\ref{EA13} boils down to $x_1d_1 + \frac{b_1}{c}z_3.d_1 = b_1w_2 + a'_2$. We first look at the case, when $b_1 \neq 0$. Then choose $w_2$ so that $a'_2 =0$. So $B$ reduces to $%
 \right\}.$ This is isomorphic to a centralizer of canonical form of type $A_4$. So $(A,B)$ is a branch of type $A_4$, and there are $q^2(q-1)^2$ such branches. 

When $b_1 = 0$. Then we have $y_1 = 0$. Hence Equation~\ref{EA13} becomes $x_1d_1 + a_2 = a'_2$. When $d_1\neq 0$, choose $x_1$ so that $a'_2 = 0$. $B$ is reduced to $%
  \right\}$, which is the centralizer of a matrix of type $R_1$. Thus $(A,B)$ is of type $R_1$, and thus there are $q(q-1)^2$ branches of this type.

When $d_1 = 0$, and $d_2 \neq 0$. Equation~\ref{EA13} ends up becoming $a'_2 = a_2$, and $B$ is reduced to $%
  \right\}$. Thus $(A,B)$ is a branch of type $A_4$, and there are $q(q-1)^2$ such branches.

{\bfseries When $C =  \left(%
$, where $\lbd_0 = \frac{c_2}{c_1}$. Thus, from Equation~\ref{EA11}, we have: $%
$. So, we choose $x_1$ so that $b'_1 = 0$. Similarly, we choose $y_1$ such that $b'_2 = 0$. Thus, on replacing $b_1$, and $b_2$ by $0$ in the above equation, we get that $x_1 = 0$, and $y_1 = 0$. Equation~\ref{EA12} becomes $%
$. Thus $d'_2 = d_2$, and we can choose $w_2$ so that $d'_1 = 0$. So we are left with $d_2 \neq 0$. With $x_1 = y_1 = b_1 = 0$, Equation~\ref{EA13} becomes $a'_2 = a_2$. Hence $B$ is reduced to $%
 \mid \lbd_0 = \frac{c_2}{c_1}, \lbd_1 = \frac{d_2}{c_1} \right\}$, which is a centralizer of type $R_2$. $(A,B)$ is a branch of type $R_2$, and there are $q(q-1)^3$ branches of this type.

\noindent{\bfseries Case $a_1 \neq 0$}: We look at the various types of $C$ as our subcases.

{\bfseries When $C = I_3$:} Here Equation~\ref{EA11} becomes: 
$$%
.$$ Using a suitable $z_1$, we can make $b'_1 = 0$, and using a suitable $z_2$, we can make $b'_2 = 0$. Thus, on replacing $b_1$ and $b'_2$ by 0 in the above equation, we have $z_1 = z_2 = 0$. Hence with this, Equation~\ref{EA12} becomes $%
$. Equation~\ref{EA13} is reduced to $a_2 + x_1d_1 + y_1d_2 = a'_2 + a_1w_1$. So we choose $w_1$ such that $a'_2 =0$. Thus $B$ is reduced to $%
\mid \lbd = \frac{d_1}{a_1}, \mu = \frac{d_2}{a_1}\right\}$, which is the centralizer of type $B_4$. $(A,B)$ is thus a branch of type $B_4$, and there are $q^2(q-1)$ such branches. 

{\bfseries When $C = \left(%
.$$
Choose $z_1$ and $z_2$ such that $b'_1 = b'_2 = 0$. Again, like in the previous case on replacing $b_1$ and $b_2$ by $0$ in the above equation, we have $z_1 = 04$ and $z_2 = \frac{c}{a_1}x_1$. From Equation~\ref{EA12}, we get $d'_1 = d_1$ amd $d'_2 = d_2$. Equation~\ref{EA13} is reduced to $x_1d_1 + y_1d_2 + a_2 = w_1a_1 + a'_2$. We choose $w_1$ such that $a'_2 = 0$. Thus $B$ is reduced to $%
\mid\lbd_1 = \frac{c}{a_1},  \lbd_2 = \frac{d_1}{a_1}, \mu = \frac{d_2}{a_1}\right\}$. This is of type $B_4$. Hence $(A,B)$ is a branch of type $B_4$, and there are $(q-1)^2q^2$ such branches.

{\bfseries When $C = \left(%
 $$
Choose $z_1$ such that $b'_1 - 0$, and choose $z_2$ such that $b'_2 = 0$. So, on substituting $b_1$ and $b_2$ with $0$ in the above, we have $z_1 = \frac{c}{a_1}y_1$, and $z_2 = 0$. Thus Equation~\ref{EA12} is reduced to $%
$. We have $d'_2 = d_2$. Choose $w_2$ so that $d'_1 = 0$. Equation~\ref{EA13} is reduced to $y_1d_2 + a_2 = a_1w_1 + a'_2$. Choose $w_1$ such that $a'_2 = 0$. Thus $B$ is reduced to $%
\right\}$, which is a centralizer of a matrix of type $R_3$. So $(A,B)$ is a branch of type $R_3$, and there are $(q-1)^2q$ branches of this type.

{\bfseries When $C = \left(%
.$$ We have $b'_1 = b_1$. we can choose $y_2$ such that $b'_2 =0$. Thus, on replacing $b_2$ by 0 in the above equation, we have $y_2 = \frac{a_1}{c}z_2 + \frac{b_1}{c}z_3$. And Equation~\ref{EA12} ends up giving us $d'_1 = d_1$, and $d'_2 = d_2$. Thus Equation~\ref{EA13} stays as it is. Since $a_1\neq 0$, we choose $w_1$ so that $a'_2 = 0$. $B$ is therefore reduced to $%
\right\}$, which is that of type $B_4$. $(A,B)$ is of type $B_4$, and the number of branches is $q^3(q-1)^2$.

{\bfseries When $C = \left(%
.$$ Choose $z_1$ such that $b'_1 = 0$, and choose $z_2$ such that $b_2 =0$. On replacing $b_1$ and $b_2$ by 0 in the above equation, we see that $z_1 = \frac{c_1}{a_1}x_1$, and $z_2 = \frac{c_2}{a_1}y_1$. From Equation~\ref{EA12}, we have $d_1 + \frac{c_1}{a_1}x_1d_2 = c_1w_2 + d'_1$, and $d'_2 = d_2$. So we choose $w_2$ such that $d'_1 = 0$. Equation~\ref{EA13} becomes: $y_1d_2 + a_2 = w_1a_1 + a'_2$. Choose $w_1$ such that $a'_2 = 0$. Thus $B$ is reduced to $%
\right\}$, which is the centralizer of a matrix of type $R_3$. Thus $(A,B)$ is branch of type $R_3$, and there are $q(q-1)^3$ such branches. Hence, adding up the branches of each type, we get the numbers as mentioned in the statement of this proposition. 
\end{proof}
\begin{prop}\label{T5A2}
 An upper unitriangular matrix of type $A_2$ has $q^4$ branches of type $A_2$, $2q^2(q^2-1)$ branches of regular type $R_1,$ and $q(q^2-1)^2$ branches of regular type $R_2.$
 \end{prop}
\begin{proof}
 Let $A=%
, a\neq0$ a matrix of type $A_2.$ The  centralizer $Z_{UT_5}(A)$ of $A$ is $\left\{ %
$ be an element of $Z_{UT_5(\Fq)}(A).$ Let $B=%
$ be a conjugate of $B$ by $X.$ Thus equating 
  $XB=B^\prime X$ gives us $a'_1 = a_1$, $b'_0 = b_0$, $c'_0 = c_0$, $c'_1 = c_1$, and the following equations:
  \begin{center}
 \begin{tabular}{ccc}
 $x_1c_0+a_2=a_1z_0+a_2^\prime$\\
 $x_1c_1+a_3=a_1z_1+a_3^\prime$\\
 $y_0c_0+b_1=b_1^\prime+b_0^\prime z_0$\\
 $y_0c_1+b_2=b_2^\prime+b_0^\prime z_1$\\
 \end{tabular}
 \end{center}
 We consider two cases when $(a_1,b_0,c_0,c_1)={\bf 0}$ and when $(a_1,b_0,c_0,c_1)\neq {\bf 0}.$
 
 {\bf Case: $(a_1,b_0,c_0,c_1)={\bf 0}.$ } In this case, we get $a_2=a_2^\prime,~ a_3=a_3^\prime,~b_1=b_1^\prime$ and $b_2=b_2^\prime.$ Therefore $Z_{UT_5(\Fq)}(A,B)=Z_{UT_5(\Fq)}(A).$ So $(A,B)$ is a branch of type $A_2,$ and there are 
 $q^4$ branches.
 
{\bf Case: $(a_1,b_0,c_0,c_1)\neq{\bf 0}.$ } First we consider that $c_1\neq 0.$ We choose $x_1$ and $y_0$ in such a way that we get $a_3=b_2=0.$ Now if $(a_1,b_0)=(0,0),$ then by simple calculations, we get $Z_{UT_5(\Fq)}(A,B)$ is a commutative group of size $q^6.$  Thus $(A,B)$ is of regular type $R_1,$ and there are $q^3(q-1)$ branches of this type. If we consider that case when at least one of $a_1$ and $b_0$ is non-zero, then we can choose $z_0$ suitably so that we get one of $a_2$ or $b_1$ equal to zero. By routine check, we get that $Z_{UT_5(\Fq)}(A,B)$ is a commutative group of size $q^5.$  Thus $(A,B)$ is of regular type $R_2,$ and there are $(q^3-q^2)(q^2-1)$ branches of this type.

Now we consider that $c_1= 0$ and $c_0\neq 0.$ We choose $x_1$ and $y_0$ in such a way that we get $a_2=b_1=0.$ Now if $(a_1,b_0)=0,$ then by simple calculations, we get $Z_{UT_5(\Fq)}(A,B)$ is a commutative group of size $q^6.$  Thus $(A,B)$ is of regular type $R_1,$ and there are $q^2(q-1)$ branches of this type. If we consider that case when at least one of $a_1$ and $b_0$ is non-zero, then we can choose $z_1$ suitably so that we get one of $a_3$ or $b_2$ equal to zero. By routine check, we get that $Z_{UT_5(\Fq)}(A,B)$ is a commutative group of size $q^5.$  Thus $(A,B)$ is of regular type $R_2,$ and there are $(q^2-q)(q^2-1)$ branches of this type.

Next we consider when $c_1=c_0= 0$ and $b_0\neq 0.$ We choose $z_0$ and $z_1$ in such a way that we get $b_1=b_2=0.$ Now by simple calculations, we get $Z_{UT_5(\Fq)}(A,B)$ is a commutative group of size $q^6.$ Thus $(A,B)$ is of regular type $R_1,$ and there are $q^3(q-1)$ branches of this type.

Finaly we consider when $c_1=c_0=b_0 0$ and $a_1\neq 0.$ We choose $z_0$ and $z_1$ in such a way that we get $a_2=a_3=0.$ Now by simple calculations, we get $Z_{UT_5(\Fq)}(A,B)$ is a commutative group of size $q^6.$ Thus $(A,B)$ is of regular type $R_1,$ and there are $q^2(q-1)$ branches of this type.

Therefore a matrix of type $A_2$ has $q^4$ branches of type $A_2$, $2q^2(q^2-1)$ branches of regular type $R_1,$ and $q(q^2-1)^2$ braches of regular type $R_2.$
 
\end{proof}
\begin{prop}\label{T5A3}
An upper unitriangular matrix of type $A_3$ has 
\begin{center}
\begin{tabular}{c|c||c|c}\hline
 Branch & No. of Branches & Branch & No. of Branches\\ \hline
 $A_3$ & $q^2$ & $R_1$ &$q(q^2+q-1)(q-1)$ \\
 $A_4$ & $q(q^2-1)$ & $R_2$ &$q(q^2-1)(q-1)$\\
 $B_3$ & $q^2(q-1)$ & $R_3$ &$q^2(q-1)^2$\\
 $B_5$ & $q^3(q-1)$ & $UNT_1$ &$q^2(q-1)$. \\ \hline
\end{tabular}
\end{center}
It has a new type branch, named $UNT_1$, with common centralizer $\left\{\begin{psmallmatrix}1&x_0&x_1&\lambda z_0 &x_3\\&1&y_0&&y_2\\&&1&& \\&&&1&z_0\\&&&&1 \end{psmallmatrix}\right\}$. 
\end{prop}
\begin{proof}
 Let $A=\begin{pmatrix} 1& &a&&\\ &1&&&\\ &&1&&\\&&&1&&\\&&&&1\end{pmatrix}, a\neq0$ a matrix of type $A_3.$ The  centralizer $Z_{UT_5}(A)$ of $A$ is $\left\{ \begin{pmatrix} 1&x_0  & x_1& x_2&x_3\\ & 1&y_0 &y_1  &y_2 \\ &&1&& \\ &&&1&z_0\\&&&&1\end{pmatrix} ~|~ ,x_i,y_i,w_0\in \Fq\right\}.$ Let $X=\begin{pmatrix} 1&x_0 & x_1& x_2&x_3\\ & 1&y_0 &y_1  &y_2 \\ &&1&& \\ &&&1&w_0\\&&&&1\end{pmatrix}$ be an element of $Z_{UT_5(\Fq)}(A).$ Let $B=\begin{pmatrix} 1&a_0  & a_1& a_2&a_3\\ & 1&b_0&b_1& b_2 \\ &&1&& \\ &&&1&d_0\\&&&&1\end{pmatrix},$ and 
  $B^\prime= \begin{pmatrix} 1&a_0^\prime & a_1^\prime& a_2^\prime&a_3^\prime\\ & 1&b_0^\prime &b_1^\prime&b_2^\prime \\ &&1&& \\ &&&1&d_0^\prime\\&&&&1\end{pmatrix}$ be a conjugate of $B$ by $X.$ Thus equating 
  $XB=B^\prime X$ gives us $a'_0 = a_0$, $b'_0 = b_0$, $b'_1 = b_1$, $d'_0 = d_0$, and the following equations:
  \begin{center}
 \begin{tabular}{c}
 $a_1+x_0b_0=a_0y_0+a_1^\prime$\\
 $a_2+x_0b_1=a_0y_1+a_2^\prime$\\
 $a_3+x_0b_2+x_2d_0=a_0y_2+a_2^\prime z_0+a_3^\prime$\\
 $b_2+y_1d_0=z_0b_1+b_2^\prime$
 \end{tabular}
 \end{center}
 We consider two cases when $(a_0,b_0,b_1,d_0)={\bf 0}$ and when $(a_0,b_0,b_1,d_0)\neq {\bf 0}.$
 
 {\bf Case: $(a_0,b_0,b_1,d_0)={\bf 0}.$ } In this case, we get $a'_1 = a_1$, $a'_2 = a_2$, $b'_2 = b_2$, and  $a_3+x_0b_2=a_2z_0+a_3^\prime$.

 If $(a_2,b_2)=\bf{0},$ then we get $b_3=b_3^\prime.$ Therefore $Z_{UT_5(\Fq)}(A,B)=Z_{UT_5(\Fq)}(A).$ So $(A,B)$ is a branch of type $A_3,$ and there are 
 $q^2$ branches.
 Now we consider that $a_2\neq0.$ In this case, we can choose $w_0$ in such a way that we get $a_3=0.$ By routine check, we get 
 $Z_{UT_5(\Fq)}(A,B)$ is a group of order $q^7$ and $(A,B)$ is the type $B_3$, and there are 
 $q^2(q-1)$ branches.
 
 If we consider $a_2=0$ and $b_2\neq 0,$ choose $x_0$ in such a way that we get $a_3=0.$  By routine check, we get 
 $Z_{UT_5(\Fq)}(A,B)$ is a group of order $q^7$ and $(A,B)$ is a branch of type $A_4,$ and there are 
 $q(q-1)$ branches.
 
 {\bf Case: $(a_0,b_0,b_1,d_0)\neq{\bf 0}.$} First we consider that $a_0\neq 0.$ In this case, we can choose $y_0,y_1$ and $y_2$ in such a way that we get $a_1=a_2=a_3=0$ and $b_2+\frac{d_0b_1}{b_0}x_0=z_0b_1+b_2^\prime.$ Now if $b_1=0,$ then we get $b_2=b_2^\prime.$ By routine check, we get 
 $Z_{UT_5(\Fq)}(A,B)$ is a group of order $q^5$ and $(A,B)$ is a branch of type $B_5,$ and there are 
 $q^3(q-1)$ branches. On the other hand if $b_1\neq0,$ then we choose $z_0$ in such a way that we get $b_2=0$ By routine check, we get 
 $Z_{UT_5(\Fq)}(A,B)$ is a commutative group of order $q^4$ and $(A,B)$ is a branch of regular type $R_3,$ and there are 
 $q^2(q-1)^2$ branches.

 Now we consider that $a_0=0$ and $b_0\neq 0.$ In this case, we can choose $x_0$ in such a way that we get $a_1=0$ and this implies $x_0=0.$ Thus we get $b_2=b_2^\prime$ and the following equalities:
 \begin{tabular}{ccc}
 $a_3+x_2d_0=a_2z_0+a_3^\prime$\\
 $b_2+d_0y_1=z_0b_1+b_2^\prime$
 \end{tabular}
 Now if $(d_0,a_2,b_1)=\bf{0},$ then we get $a_3=a_3^\prime$ and $b_2=b_2^\prime.$ By routine check, we get 
 $Z_{UT_5(\Fq)}(A,B)$ is a group of order $q^7$ and $(A,B)$ is a branch of type $A_4,$ and there are 
 $q^2(q-1)$ branches. If $d_0\neq 0,$ then we choose $x_2$ and $y_1$ in such a way that we get $a_3=b_2=0.$ By routine check, we get 
 $Z_{UT_5(\Fq)}(A,B)$ is a commutative group of order $q^5$ and $(A,B)$ is a branch of regular type $R_2,$ and there are 
 $q^2(q-1)^2$ branches.
 
 If $d_0=0$ and $a_2\neq 0,$ then we choose $w_0$ in such a way that we get $a_3=0$ and this implies $w_0=0.$ Thus we get $b_2=b_2^\prime.$ By routine check, we get 
 $Z_{UT_5(\Fq)}(A,B)$ is a commutative group of order $q^6$ and $(A,B)$ is a branch of regular type $R_1,$ and there are 
 $q^2(q-1)^2$ branches.
 
 If $d_0=a_2=0$ and $b_1\neq 0,$ then  we get $a_3=a_3^\prime$ and we choose $w_0$ in such a way that we get $b_2=0.$ By routine check, we get 
 $Z_{UT_5(\Fq)}(A,B)$ is a commutative group of order $q^6$ and $(A,B)$ is a branch of regular type $R_1,$ and there are 
 $q(q-1)^2$ branches.
 
 Now we consider that $a_0=b_0=0$ and $b_1\neq 0.$ In this case, we can choose $x_0$ and $z_0$ in such a way that we get $a_2=b_2=0.$ In addition to this, if $d_0=0,$ then we get $a_3=a_3^\prime.$ By routine check, we get 
 $Z_{UT_5(\Fq)}(A,B)$ is a commutative group of order $q^6$ and $(A,B)$ is a branch of regular type $R_3,$ and there are 
 $q^2(q-1)$ branches. Now if we consider $d_0\neq 0,$ then we can choose $x_2$ in such a way that we get $a_3=0.$ By routine check, we get 
 $Z_{UT_5(\Fq)}(A,B)$ is a commutative group of order $q^5$ and $(A,B)$ is a branch of regular type $R_2,$ and there are 
 $q(q-1)^2$ branches. 
 
 Finally we consider the case when $a_0=b_0=b_1=0$ and $d_0\neq 0,$ then we get $a_2=a_2^\prime, a_1=a_1^\prime$ and  we can choose $y_1$ and $x_2$ in such a way that we get $a_3=b_2=0.$ By routine check, we get 
 $Z_{UT_5(\Fq)}(A,B)$ is a group of order $q^6$, and $\ZU{5}(A,B) =\left\{\begin{psmallmatrix}1&x_0&x_1&\lambda z_0 &x_3\\&1&y_0&&y_2\\&&1&& \\&&&1&z_0\\&&&&1 \end{psmallmatrix}\right\}$. As we have not seen this centralizer before, and  This $(A,B)$ is a branch of new type, which we call $UNT_1$ and there are $q^2(q-1)$ branches. 
\end{proof}
\begin{prop}\label{T5A4}
 An upper unitriangular matrix of type $A_4$ has $q^4$ branches of type $A_4$, $q^3(q^2-1)$ branches of regular type $R_1$, and $q^4(q-1)$ breaches of regular type $R_2$.
\end{prop}
\begin{proof}
 Let $A=\begin{pmatrix} 1& &&&\\ &1&a&&\\ &&1&&\\&&&1&&\\&&&&1\end{pmatrix}, a\neq0$ a matrix of type $A_4.$ The  centralizer  $Z_{UT_5}(A)$ of $A$ is $\left\{ \begin{pmatrix} 1&  & x_1& x_2&x_3\\ & 1&y_0 &y_1  &y_2 \\ &&1&& \\ &&&1&w_0\\&&&&1\end{pmatrix} ~|~ ,x_i,y_i,w_0\in \Fq\right\}.$ Let $X=\begin{pmatrix} 1& & x_1& x_2&x_3\\ & 1&y_0 &y_1  &y_2 \\ &&1&& \\ &&&1&w_0\\&&&&1\end{pmatrix}$ be an element of $Z_{UT_5(\Fq)}(A).$ Let $B=\begin{pmatrix} 1&  & a_1& a_2&a_3\\ & 1&b_0&b_1& b_2 \\ &&1&& \\ &&&1&d_0\\&&&&1\end{pmatrix},$ and 
  $B^\prime= \begin{pmatrix} 1& & a_1^\prime& a_2^\prime&a_3^\prime\\ & 1&b_0^\prime &b_1^\prime&b_2^\prime \\ &&1&& \\ &&&1&d_0^\prime\\&&&&1\end{pmatrix}$ be a conjugate of $B$ by $X.$ Thus equating 
  $XB=B^\prime X$ gives us  $a_1=a_1^\prime$ $a,_2=a_2^\prime$, $b_0=b_0^\prime$, $b_1=b_1^\prime$,  $d_0=d_0^\prime$ and the following equations:
  \begin{center}
 \begin{tabular}{ccc}
 $a_3+x_2d_0=a_2^\prime w_0+a_3^\prime$\\
 $b_2+ d_0y_1=w_0b_1^\prime+b_2^\prime$
 \end{tabular}
 \end{center}
 We consider two cases when $(a_2,b_1,d_0)={\bf 0}$ and when $(a_2,b_1,d_0)\neq {\bf 0}.$
 
 {\bf Case: $(a_2,b_1,d_0)={\bf 0}.$ } In this case, we get $a_3=a_3^\prime$ and $b_2=b_2^\prime.$ Therefore $Z_{UT_5(\Fq)}(A,B)=Z_{UT_5(\Fq)}(A).$ So $(A,B)$ is a branch of type $A_4,$ and there are 
 $q^4$ branches.

 {\bf Case: $(a_2,b_1,d_0)\neq{\bf 0}.$} First we consider that $d_0\neq 0.$ Now we can choose $x_2$ and $y_1$ in such a way that we get $a_3=b_2=0.$ By routine check, we get 
 $Z_{UT_5(\Fq)}(A,B)$ is a commutative group of size $q^5.$  Thus $(A,B)$ is of regular type $R_2,$ and there are $q^4(q-1)$ branches of this type. 
 
 Now we consider that $d_0=0$ and $a_2\neq 0.$ In this case, we can choose $w_0$ in such a way that we get $a_3=0.$ By routine check, we get 
 $Z_{UT_5(\Fq)}(A,B)$ is a commutative group of size $q^6.$  Thus $(A,B)$ is of regular type $R_1,$ and there are $q^4(q-1)$ branches of this type. 
 
 Finaly we consider when $d_0=a_2=0$ and $b_1\neq 0.$, now we can choose $w_0$ in such a way that we get $b_2=0.$  Again, we get 
 $Z_{UT_5(\Fq)}(A,B)$ is commutative group of size $q^6.$  Thus $(A,B)$ is of regular type $R_1,$ and there are $q^3(q-1)$ branches of this type. 
 
 Therefore we get that a matrix of type $A_4$ has $q^4$ branches of type $A_4$,   $q^3(q^2-1)$ braches of regular type $R_1,$ and $q^4(q-1)$ braches of regular type $R_2.$
 \end{proof}
\begin{prop}\label{T5A5}
 An upper unitriangular matrix of type $A_5$ has:
 \begin{center}
 \begin{tabular}{c|c||c|c}\hline
 Branch Type & No. of Branches & Branch Type & No. of Branches\\ \hline
  $A_5$& $q^2$ & $R_2$ & $q^2(q-1)$\\
  $B_4$ & $2q(q-1)$ & $R_3$ & $q(q-1)(q^2-1)$\\
  $B_6$ & $q^2(q^2-1)$ & $UNT_1$ & $q(q-1)^2$. \\ \hline
 \end{tabular}
 \end{center}
 It has the new branch $UNT_1$ already seen in previous case.
\end{prop}
\begin{proof}
There are several canonical forms for a matrix in $UT_5(\Fq)$, of type $A_5$. We prove this proposition for the canonical form $A = \begin{pmatrix}1&a&&&\\&1&&&\\&&1&&\\&&&1&\\&&&&1\end{pmatrix}$, where $a \neq 0$. We have:
$\Zu{5}(A) = \left\{\begin{pmatrix}1&a_0&a_1&a_2&a_3\\&1&&&\\&&1&b_0&b_1\\&&&1&c_0\\&&&&1\end{pmatrix} \right\}$. We can rewrite this centralizer subgroup as $\left\{\begin{pmatrix}1&a_0&{}^t\overrightarrow{b}\\&1&\\&&C \end{pmatrix}\mid \begin{matrix}C \in UT_3(\Fq)\\{}^t\overrightarrow{b} = (b_1~b_2~b_3)\end{matrix} \right\}$. Let $B = \begin{pmatrix} 1&a_0&{}^t\overrightarrow{b}\\&1&\\
&&C\end{pmatrix}$, and $B' = \begin{pmatrix} 1&a'_0&{}^t\overrightarrow{b}\\&1&\\
&&C'\end{pmatrix}$ be a conjugate in $UT_5$ of $B$. $B' = XBX^{-1}$, where $X = \begin{pmatrix} 1&x_0&{}^t\overrightarrow{y}\\&1&\\
&&Z\end{pmatrix}$. So, equating $XB = B'X$ gives us $a'_0 = a_0$, $ZC = C'Z$. So, we may take $C$ to be the representative of a conjugacy class in $UT_3(\Fq)$, and we have the equation:
\begin{equation*}
{}^t\overrightarrow{y}.C + {}^t\overrightarrow{b} = {}^t\overrightarrow{b'}Z + {}^t\overrightarrow{y}
\end{equation*}

We rewrite this equation slightly to get:
\begin{equation}\label{EUA51}
\begin{pmatrix}y_1&y_2&y_3\end{pmatrix}(C-I_3) + \begin{pmatrix}b_1&b_2&b_3\end{pmatrix} = \begin{pmatrix}b'_1&b'_2&b'_3\end{pmatrix}Z
\end{equation}

The cases:

\noindent{\bfseries When $C = I_3$}. Here Equation~\ref{EUA51} becomes:
$\begin{pmatrix}b_1&b_2&b_3\end{pmatrix} = \begin{pmatrix}b'_1&b_2&b'_3\end{pmatrix}\begin{pmatrix}1&z_0&z_1\\&1&z_2\\&&1\end{pmatrix}$, which gives us $b'_1 = b_1$, and the following equation:
\begin{eqnarray}
b_2 &=& b'_2 + z_0b_1\label{EUA52}\\
b_3 &=& b'_3 + z_1b_1 + z_2b'_2\label{EUA53}
\end{eqnarray}

We have two subcases here: When $b_1 = 0$ and when $b_1 \neq 0$.

{\bfseries When $b_1 = 0$} Equation~\ref{EUA52} becomes $b'_2 = b_2$, and Equation~\ref{EUA53}
 becomes $b_3 = b'_3 + z_2b_2$.
 
 When $b_2 = 0$, we have $b'_3 = b_3$. So $B$ is reduced to $\begin{pmatrix} 1&a_0&&&b_3\\&1&&&\\&&1&&\\&&&1&\\&&&&1\end{pmatrix}$, and $\Zu{5}(A,B) = \Zu{5}(A)$. Thus $(A,B)$ is of type $A_5$, and there are $q^2$ such branches.
 
 When $b_2 \neq 0$, in Equation~\ref{EUA52}, choose $z_2$ so that $b'_3 = 0$. So, We have $B$ reduced to $\begin{pmatrix}1&a_0&&b_2&\\&1&&&\\&&1&&\\&&&1&\\&&&&1\end{pmatrix}$, and $\Zu{5}(A,B) =\left\{ \begin{pmatrix}1&x_0&y_1&y_2&y_3\\&1&&&\\&&1&z_0&z_1\\&&&1&\\&&&&1\end{pmatrix}\right\}$. $(A,B)$ is of type $B_4$ and there are $q(q-1)$ such branches.
 
 {\bfseries When $b_1 \neq 0$}: In Equation~\ref{EUA52}, choose $z_0$ such that $b'_2 = 0$, and in Equation~\ref{EUA52}, choose  $z_1$ such that $b'_3 = 0$. So $B$ is reduced to $\begin{pmatrix}1&a_0&b_1&&\\
 &1&&&\\
 &&1&&\\
 &&&1&\\
 &&&&1\end{pmatrix}$, and $\Zu{5}(A,B) =\left\{ \begin{pmatrix}1&x_0&y_1&y_2&y_3\\&1&&&\\&&1&&\\&&&1&z_2\\&&&&1\end{pmatrix}\right\}$. $(A,B)$ is a branch of type $B_6$, and there are $q(q-1)$ such branches.
 
 \noindent{\bfseries When $C = \left(\begin{smallmatrix}1&&c\\&1&\\&&1\end{smallmatrix}\right), c\neq 0$:} From Equation~\ref{EUA51}, we have $b'_1 = b_1$, and the following following equations:
 \begin{eqnarray}
 b_2 &=& b'_2 + z_0b_1\\
 b'_3 + cy_1 &=& b'_3 + z_1b'_1 + z_2b'_2 . 
 \end{eqnarray}
 As $c \neq 0$, choose $y_1$ so that $b'_3 = 0$. 
 
 {\bfseries Case: $b_1 = 0$}
 
 We have $b'_2 = b_2$. When $b_2 = 0$, $B$ is reduced to 
 $
 \begin{pmatrix}
 1&a_0&&&\\
 &1&&&\\
 &&1&&c\\
 &&&1&\\
 &&&&1
 \end{pmatrix}$, and $\Zu{5}(A,B) =\left\{ \begin{pmatrix}1&x_0&&y_2&y_3\\&1&&&\\&&1&z_0&z_1\\&&&1&z_2\\&&&&1\end{pmatrix}\right\}$. $(A,B)$ is of type $B_4$, and there are $q(q-1)$ such branches. 
 
 When $b_2 \neq 0$, we have $B$ reduced to 
 $\begin{pmatrix}
 1&a_0&&a_2&\\
 &1&&&\\
 &&1&&c\\
 &&&1&\\
 &&&&1
 \end{pmatrix}$, and $\Zu{5}(A,B) =\left\{ \begin{pmatrix}1&x_0&\frac{b_2}{c}z_2&y_2&y_3\\&1&&&\\&&1&z_0&z_1\\&&&1&z_2\\&&&&1\end{pmatrix}\right\}$. This centralizer is isomorphic to that of a new type, $UNT_1$, which we had come across earlier. There are $q(q-1)^2$ such branches.
 
 {\bfseries When $b_1 \neq 0$}. We can choose $z_0$ so that $b'_2 = 0$. Here $B$ is reduced to 
 $
 \begin{pmatrix}
 1&a_0&a_1&&\\
 &1&&&\\
 &&1&&c\\
 &&&1&\\
 &&&&1
 \end{pmatrix}
 $, and $\Zu{5}(A,B) =\left\{ \begin{pmatrix}1&x_0&\frac{b_1}{c_1}z_1&y_2&y_3\\&1&&&\\&&1&&z_1\\&&&1&z_2\\&&&&1\end{pmatrix}\right\}$. Hence $(A,B)$ is of type $B_6$, and there are $q(q-1)^2$ such branches.
 
 \noindent{\bfseries When $C = \left(\begin{smallmatrix}1&c&\\&1&\\&&1\end{smallmatrix}\right), c\neq 0$}: Here $Z = \begin{pmatrix}1&z_0&z_1\\&1&\\&&1 \end{pmatrix}$. With this, Equation~\ref{EUA51} becomes:
 $\begin{pmatrix}b_1 &b_2 + cy_1 & b_3 \end{pmatrix} = \begin{pmatrix}b'_1 & z_0b'_1 + b'_2 & z_1b'_1 + b'_3\end{pmatrix}$. Now, as $c \neq 0$, choose $y_1$ so that $b'_2 = 0$. 
 
 When $b_1  = 0$, we have $b'_3 = b_3$. Thus, $B$ is reduced to 
 $
 \begin{pmatrix}
 1&a_0&&&b_3\\
 &1&&&\\
 &&1&c&\\
 &&&1&\\
 &&&&1
 \end{pmatrix}$, and $\Zu{5}(A,B) =\left\{ \begin{pmatrix}1&x_0&&y_2&y_3\\&1&&&\\&&1&z_0&z_1\\&&&1&\\&&&&1\end{pmatrix}\right\}$. By a routine check, we can see that this centralizer is commutative, and of size $q^5$. $(A,B)$ is of type $R_2$, and there are $q^2(q-1)$ such branches.
 
 When $b_1 \neq 0$, choose $z_1$ such that $b'_3 = 0$. Thus, $B$ is reduced to 
 $\begin{pmatrix}
 1&a_0&a_1&&\\
 &1&&&\\
 &&1&c&\\
 &&&1&\\
 &&&&1
 \end{pmatrix}$, and $\Zu{5}(A,B) =\left\{ \begin{pmatrix}1&x_0&\frac{b_1}{c}z_0&y_2&y_3\\&1&&&\\&&1&z_0&\\&&&1&\\&&&&1\end{pmatrix}\right\}$. This centralizer is of size $q^4$, and is commutative. Thus $(A,B)$ is of type $R_3$, and there are $(q-1)^2q$ such branches.
 
 \noindent{\bfseries When $C = \left(\begin{smallmatrix}1&&\\&1&c\\&&1\end{smallmatrix}\right), c \neq 0$}: Here $Z = \begin{pmatrix}1&&z_1\\&1&z_2\\&&1\end{pmatrix}$. Here Equation~\ref{EUA51} becomes:
 $\begin{pmatrix}b_1 & b_2 & b_3 + cy_2\end{pmatrix} = \begin{pmatrix} b'_1 & b'_2 & b'_3 + z_1b'_1 + z_2b'_2\end{pmatrix}$.
 
 We have $b'_1 = b_1$ and $b'_2 = b_2$, and choose $y_2$ so that $b'_3 = 0$. Thus $B$ is reduced to 
 $
 \begin{pmatrix}
 1&a_0&b_1&b_2&\\
 &1&&&\\
 &&1&&\\
 &&&1&c\\
 &&&&1
 \end{pmatrix}$, and $\Zu{5}(A,B) =\left\{ \begin{pmatrix}1&x_0&y_1&\frac{b_1}{c}z_1+\frac{b_2}{c}z_2&y_3\\&1&&&\\&&1&&z_1\\&&&1&z_2\\&&&&1\end{pmatrix}\right\}$. This too is of type $B_6$, and there are $q^3(q-1)$ such branches.
 
 And now we have the last case:
 \noindent{\bfseries When $C = \left(\begin{smallmatrix}1&c&\\&1&d\\&&1\end{smallmatrix}\right), c,d \neq 0$}: Here $Z = \begin{pmatrix}1&z_0&z_1\\&1&\lbd z_0\\&&1\end{pmatrix}$, where $\lbd = \frac{d}{c}$. Equation~\ref{EUA51} becomes:
 $\begin{pmatrix}b_1 & b_2 + cy_1 & b_3 + dy_2\end{pmatrix} = \begin{pmatrix}b'_1&b'_2 + z_0b'_1 & b'_3 + z_1b'_1 + \lbd z_0 b'_2\end{pmatrix}$. We have $b'_1 = b_1$, and choose $y_1$ so that $b'_2 = 0$, and $y_2$ so that $b'_3 = 0$. Hence $B$ is reduce to 
 $
 \begin{pmatrix}
 1&a_0&b_1&&\\
 &1&&&\\
 &&1&c&\\
 &&&1&d\\
 &&&&1
 \end{pmatrix}$, and $\Zu{5}(A,B) =\left\{ \begin{pmatrix}1&x_0&\frac{b_1}{c}z_0&\frac{b_1}{d}z_1&y_3\\&1&&&\\&&1&z_0&z_1\\&&&1&\lbd z_0\\&&&&1\end{pmatrix}\right\}$. This centralizer is 4 dimensional, and commutative. Thus $(A,B)$ is of type $R_3$, and there are $(q-1)^2q^2$ such branches.
 
With this, we have no other cases to analyse. So from the calculations, we have:
\begin{itemize}
\item $q^2$ branches of type $A_5$.
\item $q(q-1)+q(q-1) = 2q(q-1)$ branches of type $B_4$.
\item $q(q-1) + q(q-1)^2 + q^3(q-1) = q^4-q^2$ branches of type $B_6$.
\item $q^2(q-1)$ branches of type $R_2$.
\item $q(q-1)^2 + q^2(q-1)^2 = q(q-1)(q^2-1)$ branches of type $R_3$, and
\item $q(q-1)^2$ branches of the new type $UNT_1$.
\end{itemize}
\end{proof}
\subsection{Branching of type $B$}
Now we look at the $B$ types and decide its branching.
\begin{prop}\label{T5B1}
An upper unitriangular matrix of type $B_1$ has the following branches:
\begin{center}
\begin{tabular}{c|c||c|c}\hline
Branch & No. of Branches & Branch & No. of Branches\\ \hline
$B_1$ & $q^2$ &$R_3$ & $(q-1)^2(q^2+q+1)$\\
$A_2$ & $3q^2-3q$ &  $UNT_1$ & $(q-1)^2$\\
$R_1$ & $2q^3-4q+2$ &  $UNT_2$&$2q^2-2q$\\
$R_2$ & $q(q-1)^2(q+2)$&  $UNT_3$&$(q-1)^3$. \\
$B_6$ & $q^2(q-1)$ & & \\ \hline
\end{tabular}
\end{center}
We have seen $UNT_1$ earlier. There are two more new types here $UNT_2$ with centralizer $\left\{\begin{psmallmatrix}1&x_1&y_1&y_2&x_2\\&1&z_1&z_2&w_1\\&&1&& \lambda x_1\\&&&1&x_1\\&&&&1 \end{psmallmatrix} \right\}$ and $UNT_3$ with centralizer $\left\{\begin{psmallmatrix}1&x_1&y_1&y_2&x_2\\&1&\lbd_1 x_1 & z_2&w_1\\&&1&\lbd_2x_1 &\frac{\lbd_2}{\lbd_1} y_1\\&&&1&x_1\\&&&&1 \end{psmallmatrix} \right\}$.
 \end{prop}
 
\begin{proof}
 A matrix of type $B_1$ has the canonical form: $%
$. We may take $a = b = 1$. Then $\ZU{5}(A) = \left\{%
 = XBX^{-1}$. Then $XB = B'X$ leads to firstly $ZC = C'Z$, so we might as well take $C$ to be a conjugacy class representative in $UT_3(\Fq)$, and $Z$, a centralizer matrix of $C$. We also get in ${}^t\overrightarrow{b}$, and $\overrightarrow{d}$, $a'_1 = a_1$, and the following equations:
 \begin{eqnarray}
  %
\label{EB12}\\
  a_2 + x_1d_1 + y_1d_2 + y_2d_1&=& a'_2 + a_1w_1 + b'_1w_2 + b'_2x_1\label{EB13}
 \end{eqnarray}
We look at two main cases: $a_1 = 0$, and $a_1 \neq 0$.

\noindent{\bfseries Case $a_1 = 0$:} First we look at $b_1 = b_2 = d_1 = d_2 = 0$. Here Equation~\ref{EB11} reduces to $%
$, and from Equation~\ref{EB13}, we have $a'_2 = a_2$.

When $C = I_3$, Equations~\ref{EB11} and \ref{EB12} are void, and we have $B = %
$, with $\ZU{5}(A,B) = \ZU{5}(A)$. Thus $(A,B)$ is of type $B_1$, and there are $q$ such branches.

When $C = \left(%
\right), c\neq 0$, we have from Equation~\ref{EB11}: $cx_1 = 0$. Hence $x_1 = 0$. With this Equation~\ref{EB12} becomes void. So, we have $B = %
\right\}$, which is the centralizer of a canonical form of type $A_2$. $(A,B)$ is a branch of type $A_2$, and there are $q(q-1)$ such branches.

When $C = \left(%
$, which leaves us with $x_1 = 0$. Equation~\ref{EB12} becomes: $%
$, thus we have $w_2 = 0$. So we have $B = I_5 + cE_{24} + a_2E_{15}$, with $\ZU{5}(A,B) =\left\{%
\right\}$, which is the centralizer of a matrix of type $R_1$. Thus $(A,B)$ is of type $R_1$, and there are $q(q-1)$ such branches.

When $C = \left(%
$, which leaves us with $y_1 = 0$. Equation~\ref{EB12} becomes $%
$, which leads to $x_1 = 0$. So $B = I_5 + a_2E_{15} + cE_{34}$, with $\ZU{5}(A,B) = \left\{%
\right\}$, which is the centralizer of a matrix of type $R_1$. $(A,B)$ is a branch of type $R_1$, and there are $q(q-1)$ such branches.

When $C = %
$, which leaves us with $x_1 = y_1 = 0$. Equation~\ref{EB12} becomes $%
$, which leaves us with $w_2 = 0$. Hence $B = I_5 + a_2E_{15}+cE_{23} + dE_{34}$, and $\ZU{5}(A,B) = \left\{%
\right\}$, which is the centralizer of a matrix of type $R_2$. $(A,B)$ is a branch of type $R_2$, and there are $q(q-1)^2$ branches.

{\bfseries When $((b_1,b_2), (d_1,d_2))\neq (\overrightarrow{0},\overrightarrow{0})$:} 

{\bfseries We start with $C = I_3$:} Thus Equation~\ref{EB11} becomes $%
$. We have $b'_1 = b_1$, and thus $b_2 = b'_2 + b_1z_3$. First, when $b_1 \neq 0$, we choose $z_3$ so that $b'_2 = 0$. Thus, on replacing $b_2$ by $b'_2 = 0$ in the equation above, we have $z_3 = 0$. So Equation~\ref{EB12} boils down to $%
$. So we have $d'_2 = d_2$. So, again, over here, when $d_2 \neq 0$, choose $z_1$ such that $d'_1 = 0$. With these, Equation~\ref{EB13} becomes $y_1d_2 + a_2 = a'_2 +w_2b_1$. So, choose $w_2$ such that $a'_2 = 0 $. So, $B$ is reduced to $I_5+b_1E_{13} + d_2E_{35}$, and $\ZU{5}(A,B) = \left\{%
\right\}$. This centralizer is isomorphic to that of the new type $UNT_1$ (as seen in Proposition~\ref{T5A3}), via the isomorphism that maps generators to generators, and extended product-wise.

When $b_1\neq0$, and $d_2 =0$, we have $d'_1 = d_1$, and Equation~\ref{EB13} becomes $x_1d_1 +  = a'_2 + b_1w_2$. So, we choose $w_2$ such that $a'_2 = 0$. So $B$ is reduced to $%
\right\}$. Thus $(A,B)$ is of a new type, which we call $UNT_2$, and there are $q(q-1)$ such branches.

When $b_1 = 0$, $b'_2 = b_2$. First, when $d_2 \neq 0$, we choose $z_1$ so that $d'_1 = 0$, and hence Equation~\ref{EB13} becomes $a_2 + y_1d_2 = a'_2z + b_2x_1$. As $d_2\neq 0$, choose $y_1$ so that $a'_2 = 0$. Hence $B$ is reduced to $%
 \right\}$. This too is of type $UNT_2$, and there are $q(q-1)$ such branches.

When $d_2 = 0$, we have $d'_1 = d_1$. So Equation~\ref{EB13} looks like: $a_2 + x_1d_1 = a'_2 + x_1b_2$. When $b_2 = d_1\neq 0$, we have $a'_2 =a_2$, and $B$ is reduced to $%
$, and $\ZU{5}(A,B) = \ZU{5}(A)$. Thus $(A,B)$ is of type $B_1$, and there are $q(q-1)$ such branches. 

When $b_2\neq d_1$, choose an $x_1$ such that $a'_2 =0$. So, $B$ is boiled down to $%
\right\}$, which is the centralizer of a matrix of type $A_2$, and there are $q(q-1)$ such branches.

{\bfseries When $C = \left(%
$. We have $b'_1 = b_1$, and as $c \neq 0$, we can choose $x_1$ so that $b'_2 = 0$. So, on replacing $b_2$ by $0$ in the above equation, we have $x_1 = \frac{b_1}{c}z_3$. 

Then from Equation~\ref{EB12}, we have $%
$. First, when $b_1\neq 0$, we choose a $z_3$ so that $d'_1 = 0$. With these, Equation~\ref{EB13} becomes: $a_2 + y_1d_2 = a'_2 + b_1w_2$. As $b_1\neq 0$, choose $w_2$ so that $a'_2 = 0$. So $B$ is reduced to $%
\right\}$. This isn't isomorphic to the centralizer of any matrix in $UT_5(\Fq)$, hence $(A,B)$ is of a new type $UNT_3$, and there are $(q-1)^3$ such branches. When $d_2 =0$, $\ZU{5}(A,B) = \left\{%
\right\}$. Hence $(A,B)$ is of type $R_1$, and there are $(q-1)^2$ such branches.

When $b_1 =0$, we have $x_1 = \frac{b_1}{c}z_3 = 0$, we have $d'_1 = d_1+z_1d_2$. When $d_2\neq 0$, choose $z_1$ such that $d'_1 = 0$. Equation~\ref{EB13} becomes $a_2 +y_1d_2 = a'_2$. Choose $y_1$ so that $a'_2 = 0$. So, $B$ is reduced to $%
\right\}$. $(A,B)$ is of type $R_1$, and there are $(q-1)^2$ such branches.

When $b_1 =0$, and $d_2 = 0$. Then $d'_1 = d_1$, which is $\neq 0$ Then from Equation~\ref{EB13}, we simply have $a'_2 = a_2$, and $B$ is reduced to $%
\right\}$. $(A,B)$ is of type $A_2$, and there are $q(q-1)$ such branches.

{\bfseries When $C\left(%
$. We get that $b'_2 = b_2$. We choose $x_1$ so that $b'_1 = 0$. Thus, on replacing $b_1$ by 0, and equating the above equation, we have $x_1 = 0$. Equation~\ref{EB12} becomes $%
$. Again, over here, we choose $w_2$ such that $d'_1 = 0$. Now, on substituting $b_1$ with 0, we have $w_2 = \frac{d_1}{c}z_1$. So, Equation~\ref{EB13} becomes $a_2 + y_1d_2 = a'_2$.

When $d_2 \neq 0$, choose $y_1$ such that $a'_2 = $. So $B$ is reduced to $%
\right\}$. $(A,B)$ is therefore of type $R_2$, and there are $q(q-1)^2$ such branches.

When $d_2 =0$, and $b_2\neq 0$. Then Equation~\ref{EB13} becomes $a'_2 = a_2$. So, $B$ is reduced to $%
\right\}$. So $(A,B)$ is of type $R_1$, and there are $q(q-1)^2$ branches of this type.

{\bfseries When $C\left(%
$. $b'_1 = b_1$. Choose $y_1$ such that $b'_2 = 0$. So, on substituting $b_2$ with 0, we have $y_1 = \frac{b_1}{c}z_3$. Equation~\ref{EB12} becomes $%
$. Choose $x_1$ such that $d'_2 = 0$. So $d'_1 = d_1$. Hence, on replacing $d_2$ by 0, we get $x_1 = 0$. Hence, Equation~\ref{EB13} becomes $a_2 = a'_2 + b_1w_2$. When $b_1 \neq 0$, choose $w_2$ such that $a'_2 = 0$. Thus, $B$ is reduced to $%
\right\}$. So $(A,B)$ is of type $R_2$, and there are $q(q-1)^2$ such branches. 

When $b_1 = 0$, and $d'_1 = d_1$, we get from Equation~\ref{EB13}, we get $a'_2 = a_2$. Hence $B$ is reduced to $%
\right\}$. $(A,B)$ is of type $R_1$, and there are $q(q-1)^2$ such branches.

{\bfseries When $C\left(%
$. Choose $x_1$ such that $b'_1 = 0$, and then choose $y_1$ such that $b'_2 = 0$. So, on substituting $b_1$ with 0, we get $x_1 = 0$. Then, on substituting $b_2$ with 0, we get $y_1 = 0$. Equation~\ref{EB12} becomes $%
$. We have $d'_2 = d_2\neq0$, choose $w_2$ such that $d'_1 = 0$. Thus, with these Equation~\ref{EB13} becomes $a'_2 = a_2$. So $B$ becomes $%
\right\}$. $(A,B)$ is of type $R_2$. There are $q(q-1)^3$ such branches.

\noindent{\bfseries When $a_1\neq 0$:}
We now look at the branches, where the entry $a_1 \neq 0$. 

{\bfseries When $C = I_3$}: Equation~\ref{EB11} becomes $%
$. As $a_1\neq 0$, choose $z_1$ such that $b'_1 = 0$. Then choose $z_2$ such that $b'_2 = 0$. Now, when we replace $b_1$ and $b_2$ by 0 in the above equation, we get $z_1 = z_2 =0$. Then Equation~\ref{EB12} becomes $%
$. Choose $z_3$ such that $d'_2 = 0$. Equation~\ref{EB13} becomes $a_2 + x_1d_1+a_1y_2 = a'_2 + a_1w_1$. Choose $w_1$ such that $a'_2 = 0$. So, $B$ is boiled down to $%
\right\}$. Now, we see that this centralizer is of size $q^5$, hence we expect it to be a commutative one. But it isn't. We also know that no matrix in $UT_5(\Fq)$ has a non-commutative centralizer of size $q^5$, and it is isomorphic to that of the type $B_6$. Thus, $(A,B)$ is of type $B_6$, and there are $q(q-1)$ such branches.

{\bfseries When $C = \left(%
\right), c\neq 0$:} Equation~\ref{EB11} in this case is $$%
.$$
Choose $z_1$ so that $b'_1 = 0$. Then, we choose $z_2$ so that $b'_2 = 0$. Thus, on substituting $b_1$ and $b_2$ with 0, we get $z_1 = 0$, and $z_2 = \frac{c}{a_1}x_1$. Then Equation~\ref{EB12} becomes $%
.$ Choose $z_3$ such that $d'_2 = 0$. Equation~\ref{EB13} becomes $a_2 + x_1d_1 + a_1y_2 = a'_2 + a_1w_1$. Choose $w_1$ such that $a'_2 = 0$. So $B$ boils down to $%
\right\}$. This too is a branch of type $B_6$, and there are $q(q-1)^2$ such branches.

{\bfseries When $C = \left(%
$. Equation~\ref{EB11} in this case is $$%
.$$
Choose $z_1$ so that $b'_1 = 0$, and choose $z_2$ so that $b'_2 = 0$. Thus, on substituting $b_1$ and $b_2$ with 0, we get $z_1 = \frac{c}{a_1}x_1$, and $z_2 = 0$. Then Equation~\ref{EB12} becomes $%
.$ So $d'_2 = d_2$, and we choose $w_2$ such that $d'_1 = 0$. Equation~\ref{EB13} becomes $a_2 + y_1d_2 + a_1y_2 = a'_2 + a_1w_1$. Choose $w_1$ such that $a'_2 = 0$. So $B$ boils down to $%
\right\}$. This is a branch of type $R_3$, and there are $q(q-1)^2$ such branches.

{\bfseries When $C = \left(%
$. Equation~\ref{EB11} in this case is $$%
.$$ We get $b'_1 = b_1$. Choose $y_1$ so that $b'_2 = 0$. Thus, on substituting $b_2$ with 0, we get $y_1 = \frac{a_1}{c}z_2+\frac{b_1}{c}z_3$. Then Equation~\ref{EB12} becomes $%
.$ Choose $z_2$ such that $d'_1 = 0$, and $z_3$ such that $d'_2 = 0$. Equation~\ref{EB13} becomes $a_2+  a_1y_2 = a'_2 + b_1w_2+ a_1w_1$. Choose $w_1$ such that $a'_2 = 0$. So $B$ boils down to $%
\right\}$. This is a branch of type $R_3$, and there are $q(q-1)^2$ such branches.

And, lastly:

{\bfseries When $C = \left(%
$. Equation~\ref{EB11} in this case is $$%
.$$
Choose $x_1$ so that $b'_1 = 0$, and choose $y_1$ so that $b'_2 = 0$. Thus, on substituting $b_1$ and $b_2$ with 0, we get $x_1 = \frac{a_1}{c}z_1$, and $y_1 = \frac{a_1}{d}z_2$. Then Equation~\ref{EB12} becomes $%
.$ So $d'_2 = d_2$, and we choose $w_2$ such that $d'_1 = 0$. Equation~\ref{EB13} becomes $a_2 + \frac{a_1}{d}z_2d_2 + a_1y_2 = a'_2 + a_1w_1$. Choose $w_1$ such that $a'_2 = 0$. So $B$ boils down to $%
\right\}$. This is a branch of type $R_3$, and there are $q(q-1)^3$ such branches.

\end{proof}

\begin{prop}\label{T5B2}
 An upper unitriangular matrix of type $B_2$ has 
 \begin{center}
 \begin{tabular}{c|c||c|c}\hline
 Branch & No. of Branches & Branch & No. of Branches\\ \hline
   $B_2$ & $q^2$ &$R_3$ & $(q-1)^2q^2$\\
   $A_4$ & $q^3-q$ & $UNT_1$ & $q(q-1)^2$\\
   $B_4$ & $q^2-q$ &  $UNT_2$&$q^3-q^2$\\
   $R_1$ & $(q^2-q)(q^2+q-1)$&  $B_6$&$q^4-q^3$. \\
   $R_2$ & $(q-1)(q^3-q)$ & & \\ \hline
 \end{tabular}
 \end{center}
 \end{prop}

\begin{proof}
We may take $A = %
$. The first of the two canonical forms mentioned for a matrix of type $B_2$. For this $A$, we have $\ZU{5}(A) = \left\{%
\right\}$. We rewrite such a matrix as 
			 $%
$
Then from $X B = B'X$, we have $ZC = C'Z$. Thus, we can $C$ to be a conjugacy class representative in $UT_3(\Fq)$, and $Z \in \ZU{3}(C)$, and we also have $a' = a$. With this, we have the following equations
\begin{eqnarray}
Z%
\label{EB22}
\end{eqnarray}

We first look at the case $\overrightarrow{b} = \overrightarrow{0}$:

\noindent{\bfseries When $a = 0$:} Here, Equation~\ref{EB21} becomes $(C-I_3)%
$.

{\bfseries When $C=I_3$:} Here Equation~\ref{EB21} becomes void, and Equation~\ref{EB22} becomes $%
$, and $\ZU{5}(A,B) = \ZU{5}(A)$. So $(A,B)$ is a branch of type $B_2$, and there are $q$ such branches.

When $d_2 \neq 0$, choose $z_1$ so that $d'_1 = 0$. Then $B$ is reduced to $%
\right\}$, which is the centralizer of a matrix of type $A_4$. So $(A,B)$ is of type $A_4$, and there are $(q-1)$ such branches.

{\bfseries When $C = \left(%
\right), c\neq 0$:} Here too, Equation~\ref{EB21} stays void. So, we directly look at Equation~\ref{EB22}, which boils down to:
$%
$. We have $d'_2 = d_2$. We look at two cases here:

When $d_2 = c$, we get $d'_1=d_1$, and thus $B$ is reduced to 
$%
$, and $\ZU{5}(A,B) = \ZU{5}(A)$. $(A,B)$ is of type $B_2$, and there  are $q(q-1)$ such branches. 

When $d_2\neq c$, choose $z_1$ such that $d'_1 = 0$. Thus $B$ boils down to $%
\right\}$. $(A,B)$ is of type $A_4$, and there are $(q-1)^2$ such branches.

{\bfseries When $C = \left(%
$. Choose $w_2$ such that $d'_1 = 0$. Thus $B$ is reduced to $%
 \right\}$. $(A,B) $ is of type $B_6$, and there are $(q-1)q$ such branches.

{\bfseries When $C = \left(%
$. In this case, Equation~\ref{EB21} stays void. So we directly jump to Equation~\ref{EB22}. We have $%
$. So we have $B = %
\right\}$. Hence $(A,B)$ is of type $A_4$, and there are $q^2(q-1)$ such branches. 

{\bfseries When $C = \left(%
$. Hence $d'_2=d_2$, and choose $w_2 $ such that $d'_1 = 0$. So $B$ boils down to $%
 \right\}$. $(A,B)$ is of type $B_6$ and there are $(q-1)^2q$ such branches.

\noindent{bfseries $a\neq 0$:} We are still dealing with $\overrightarrow{b} = \overrightarrow{0}$ here. So Equation~\ref{EB21} becomes $(C-I_3)%
 $

{\bfseries When $C =I_3$:} Equation~\ref{EB21} becomes void, and from Equation~\ref{EB22}, we have $%
 \right\}$. $(A,B)$ is of type $B_4$, and there are $q-1$ such branches.

{\bfseries When $C = \left(%
\right), c\neq 0$:} Here also, Equation~\ref{EB21} remains void. Equation~\ref{EB22} becomes:
$%
$. Choose $y_2$ and $y_1$ so that $d'_2 = d'_1 = 0$. Thus, $B $ is reduced to $%
 \right\}$. $(A,B)$ is of type $B_4$, and there are $(q-1)^2$ such branches.

{\bfseries When $C = \left(%
\right), c\neq 0$:} From Equation~\ref{EB21}, like we saw before, $y_2 = 0$. Thus Equation~\ref{EB22} boils down to $%
$. We see that $d'_2 = d_2$. Choose $w_2$ such that $d'_1 = 0$. So $B$ is reduced to $%
\right\}$. $(A,B)$ is of type $B_6$, and there are $(q-1)^2q$ such branches.

{\bfseries When $C = \left(%
\right), c\neq 0$:} Here, Equation~\ref{EB21} stays void, and Equation~\ref{EB22} becomes $%
$. Choose $y_1, y_2$ such that $d'_1 = d'_2 = 0$. SO $B$ is reduced to $%
\right\}$. $(A,B)$ is therefore of type $R_2$, and there are $(q-1)^2$ such branches.

{\bfseries When $C = \left(%
\right), c,d\neq 0$:} Here, like earlier, from Equation~\ref{EB21}, we get $y_2 = 0$. Hence Equation~\ref{EB22} boils down to $%
$. This leaves us with $d'_2 = d_2$, and choose $y_1$ such that $d'_1 = 0$. So we have $B$ boiling down to $%
 \right\}$. $(A,B)$ is thus, of type $B_6$, and there are $q(q-1)^3$ such branches.

\noindent Now we look at what happens, when $\overrightarrow{b}\neq \overrightarrow{0}.$

{\bfseries When $C = I_3$:} Subcase $a= 0$: From Equation~\ref{EB21} we have $%
$. When $b_2 \neq 0$, we choose $z_1$ such that $b'_1 = 0$. Thus replacing $b_1$ by 0 in the above equation, we obtain $z_1 = 0$. Hence, Equation~\ref{EB22} boils down to $%
$. We have $d'_1 = d_1$. Choose $x $ such that $d'_2 = 0$. So $B$ boils down to $%
\right\}$. $(A,B)$ is of type $R_1$, and there are $q(q-1)$ such branches. 

 When $b_2 =0$, we have to look at $b_1\neq 0$, and we have $b'_1 = b_1$. Equation~\ref{EB22} becomes $%
 \right\}$. Thus $(A,B)$ is of the new type $UNT_2$, and there are $(q-1)q$ such branches. 
 
 Subcase $a\neq 0$, and $b_2\neq 0$: In Equation~\ref{EB21}, we choose $z_1$ to get rid of $b'_1$, and like before $z_1 =0$. Equation~\ref{EB22} becomes:
 $%
\right\}$. Thus $(A,B)$ is of type $R_2$, and there are $(q-1)^2$ such branches.
 
 Subcase $a\neq 0$ and $b_2 = 0$. Here we have $b'_1 = b_1 \neq 0$. From Equation~\ref{EB22} we have $%
$. Choose $y_2$ such that $d'_2 = 0$, and $x$ such that $d'_1 = 0$. Thus $B$ is reduced to $%
\right\}$. $(A,B)$ is of new type $UNT_1$, and there are $(q-1)^2$ such branches.
 
 {\bfseries When $C = \left(%
\right), c\neq 0$:} Here Equation~\ref{EB21} stays as it was in the previous case, i.e., $%
$. 
 
 When $b_2\neq 0$, choose $z_1$ so that $b'_1 = 0$, and on replacing $b_1$ with $b'_1 = 0$ in the above equation, we get $z_1 = 0$. Hence, Equation~\ref{EB22} becomes: $%
$. We can choose $x$ such that $d'_2 = 0$.
 
 Subcase $a = 0$. We have in this $d'_1 = d_1$. $B$ reduces to $%
\right\}$. So, $(A,B)$ is of type $R_1$, and there are $q(q-1)^2$ such branches.
 
 Subcase $a \neq 0$. Here, in addition to getting rid of $d'_2$, we choose $y_1$ such that $d'_1 = 0$. So, $B$ reduces to $%
\right\}$. So $(A,B)$ is of type $R_2$, and there are $(q-1)^3$ such branches.
 
 When $b_2 = 0$, here $b'_1 = b_1 \neq 0$. Equation~\ref{EB22} becomes $%
$. Choose $x$ so that $d'_1 =0$. 
 
 Subcase $a = 0$. Here $d'_2 = d_2$, and $B$ thus reduces to $%
\right\}$. Hence $(A,B)$ is of the new type $UNT_2$, and there are $(q-1)^2q$ such branches.
 
 Subcase $a \neq 0$. Here, choose $y_2$ such that $d'_2 = 0$. Hence $B$ is reduced to $%
\right\}$. Hence  $(A,B)$ is of type $UNT_1$, and there are $(q-1)^3$ such branches.
 
 {\bfseries When $C = \left(%
$. Choose $y_2$ such that $b'_1 = 0$. We have $b'_2 = b_2 \neq 0$. On replacing $b_1$ with 0 in the above equation, we get $y_2 = \frac{b_2}{c}z_1$. Equation~\ref{EB22} thus becomes $%
$. Choose $w_2$ such that $d'_1 = 0$, and $x$ such that $d'_2 = 0$. Hence $B$ is reduced to $%
\right\}$. $(A,B)$ is of type $R_3$, and there are $(q-1)^2q$ such branches.

 {\bfseries When $C = \left(%
$. 
 
 Subcase $a = 0$: When $b_1 \neq 0$, choose $x$ such that $d'_1 = 0$. Thus, on replacing $d_1$ with $d'_1 =0$, we get $x = 0$, and thus $d'_2 = d_2$.
  Hence $B$ is reduced to $%
\right\}$. $(A,B)$ is of type $R_1$, and there are $(q-1)^2q^2$ such branches.
  
  When $b_1 = 0$, we work with $b_2 \neq 0$. Choose $x$ such that $d'_2 =0$, and with this on replacing $d_2$ with $d'_2 = 0$, we have $x =0$, which leaves us with $d'_1 = d_1$. $B$ is reduced to $%
\right\}$. Hence we have another branch of type $R_1$, and these are $(q-1)^2q$ in number.
  
 Subcase $a \neq 0$. We just choose $y_1,y_2$ such that $d'_1 = d'_2 = 0$. Here $(b_1, b_2)\neq (0,0)$. So, $B = %
\right\}$. Thus $(A,B)$ is of type $R_2$, and there are $(q-1)^2(q^2-1)$ such branches (as $(b_1,b_2) \neq (0,0)$).
 
  {\bfseries When $C = \left(%
$. Choose $y_2$ such that $b'_1 = 0$. We have $b'_2 = b_2 \neq 0$. On replacing $b_1$ with 0 in the above equation, we get $y_2 = \frac{b_2}{c}z_1$. Equation~\ref{EB22} thus becomes $%
$. Choose $w_2$ such that $d'_1 = 0$, and $x$ such that $d'_2 = 0$. Hence $B$ is reduced to $%
\right\}$. $(A,B)$ is of type $R_3$, and there are $(q-1)^3q$ such branches.
\end{proof} 

\begin{prop}\label{T5B3}
 An upper unitriangular matrix of type $B_3$ has $q^3$ branches of type $B_3$, $q^2(q^2+q+1)(q-1)$ branches of regular type $R_1$, and $q^3(q-1)$ branches of regular type $R_3$.
\end{prop}
\begin{proof}
 Let $A=%
, a,b\neq0$ a matrix of type $B_3.$ The  centralizer  $Z_{UT_5}(A)$ of $A$ is $\left\{ %
$ be an element of $Z_{UT_5(\Fq)}(A).$ Let $B=%
$ be a conjugate of $B$ by $X.$ Thus equating 
  $XB=B^\prime X$ gives us $a_0=a_0^\prime$, $b_0=b_0^\prime$, $b_2=b_2^\prime$, and  the following equations:
  \begin{center}
 \begin{tabular}{ccc}
 $a_1+b_0x_0=a_0^\prime y_0+a_1^\prime$\\
 $a_2+x_0b_1+\lambda x_1a_0=a_0^\prime y_1+\lambda x_0a_1^\prime+a_2^\prime$\\
 $a_3+x_0b_2=a_3^\prime+a_0^\prime y_2$\\
 $b_1+\lambda a_0y_0=\lambda x_0b_0^\prime+b_1^\prime$
 \end{tabular}
 \end{center}
 We look at three cases, the first case is when $\lambda a_1=b_1$ and $(a_0,b_0,b_2)={\bf 0}.$ The second case is when $\lambda a_1\neq b_1$ and $(a_0,b_0,b_2)={\bf 0}.$ The third case is when $(a_0,b_0,b_2)\neq {\bf 0}.$
 
 {\bf Case: $\lambda a_1=b_1$ and $(a_0,b_0,b_2)={\bf 0}.$ } In this case, we get $a_2=a_2^\prime.$ Therefore $Z_{UT_5(\Fq)}(A,B)=Z_{UT_5(\Fq)}(A).$ So $(A,B)$ is a branch of type $B_3,$ and there are 
 $q^3$ branches.

 {\bf Case: $\lambda a_1\neq b_1$ and $(a_0,b_0,b_2)={\bf 0}.$} In this case, we can choose $x_0$ in such a way that we get $a_2=0.$ By routine check, we get 
 $Z_{UT_5(\Fq)}(A,B)$ is a commutative group of size $q^6.$  Thus $(A,B)$ is of regular type $R_1,$ and there are $q^2(q-1)$ branches of this type. 
 
 {\bf Case: $(a_0,b_0,b_2)\neq {\bf 0}.$} We first consider that $a_0\neq 0,$ then we can choose $y_0,y_1$ and $y_2$ in such a way that we get $a_1=a_2=a_3=0$ and $b_1=b_1^\prime.$ 
 By simple calculations, we get that $Z_{UT_5(\Fq)}(A,B)$ is a commutative group of size $q^4.$ Thus $(A,B)$ is of regular type $R_3,$ and there are $q^3(q-1)$ branches of this type.   
 
 Next we consider the case when $a_0=0$ and $b_0\neq 0$. Here we can choose $x_0$ in such a way that we get $a_1=0.$  By routine check, we get 
 $Z_{UT_5(\Fq)}(A,B)$ is commutative group of size $q^6.$  Thus $(A,B)$ is of regular type $R_1,$ and there are $q^4(q-1)$ branches of this type.
 
 Finaly we consider the case when $a_0=b_0=0$ and $b_2\neq 0.$, now we can choose $x_0$ in such a way that we get $a_3=0.$  Again, we get 
 $Z_{UT_5(\Fq)}(A,B)$ is commutative group of size $q^6.$  Thus $(A,B)$ is of regular type $R_1,$ and there are $q^3(q-1)$ branches of this type. 
 
 Therefore we get that a matrix of type $B_3$ has $q^3$ branches of type $B_3$,   $q^2(q^2+q+1)(q-1)$ braches of regular type $R_1,$ and $q^3(q-1)$ braches of regular type $R_3.$
 
\end{proof}

\begin{prop}\label{T5B4}
An upper unitriangular matrix of type $B_4$ has $q^3$ branches of type $B_4$, $q^2(q^2-1)$ branches of regular type $R_2,$ and $q^3(q-1)$ branches of regular type $R_3.$
\end{prop}
\begin{proof}
  Let $A=\begin{pmatrix} 1& a &&&\\ &1&&b&\\ &&1&&\\&&&1&&\\&&&&1\end{pmatrix}, a,b\neq0$ a matrix of type $B_4.$ The  centralizer  $Z_{UT_5}(A)$ of $A$ is $\left\{ \begin{pmatrix} 1& x_0 & x_1& x_2&x_3\\ & 1&& \lambda x_0 & \\ &&1&z_0& z_1\\ &&&1&\\&&&&1\end{pmatrix} ~|~ \lambda=\frac{b}{a} ,x_i,z_i\in \Fq\right\}.$ 
  
  Let $X=\begin{pmatrix} 1& x_0 & x_1& x_2&x_3\\ & 1&& \lambda x_0 & \\ &&1&z_0& z_1\\ &&&1&\\&&&&1\end{pmatrix}$ be an element of $Z_{UT_5(\Fq)}(A).$ Let $B=\begin{pmatrix} 1& a_0 & a_1& a_2&a_3\\ & 1& &\lambda a_0 & \\ &&1&c_0& c_1\\ &&&1&\\&&&&1\end{pmatrix},$ and 
  $B^\prime= \begin{pmatrix} 1& a_0^\prime & a_1^\prime& a_2^\prime&a_3^\prime\\ & 1& & \lambda a_0^\prime & \\ &&1&c_0^\prime&c_1^\prime\\ &&&1&\\&&&&1\end{pmatrix}$ be a conjugate of $B$ by $X.$ Thus equating 
  $XB=B^\prime X$ gives us $a_0=a_0^\prime$, $a_1=a_1^\prime$, $c_0=c_0^\prime$, $c_1=c_1^\prime$, and the following equations:
  \begin{center}
 \begin{tabular}{ccc}
 
 $x_1c_1+a_3=z_1a_1+a_3^\prime$\\
 $x_1c_0+a_2=z_0a_1+a_2^\prime$\\
 \end{tabular}
 \end{center}
 We look at two cases, when $(a_1,c_0,c_1)={\bf 0}$ and $(a_1,c_0,c_1)\neq{bf 0}.$
 
 {\bf Case: $(a_1,c_0,c_1)={\bf 0}$:} In this case, we get $a_2=a_2^\prime$ and $a_3=a_3^\prime.$ Therefore $Z_{UT_5(\Fq)}(A,B)=Z_{UT_5(\Fq)}(A).$ So $(A,B)$ is a branch of type $B_4,$ and there are 
 $q^3$ branches.

 {\bf Case: $(a_1,c_0,c_1)\neq{\bf 0}$:} When $a_1\neq 0,$ then we choose $z_0$ and $z_1$ in such a way that we get $a_2=a_3=0.$ By routine check, we get that
 $Z_{UT_5(\Fq)}(A,B)$ is commutative group of size $q^4.$ Thus $(A,B)$ is of the regular type $R_3,$ and there are $q^3(q-1)$ branches of this type. 
 
 When $a_1=0$ and one of $c_0$ and $c_1$ is non-zero. We can choose $x_1$ in such a way that we get either $a_2=0$ or $a_3=0.$ Again by simple calculations, we get 
 $Z_{UT_5(\Fq)}(A,B)$ is commutative group of size $q^5.$ Thus $(A,B)$ is of the regular type $R_2,$ and there are $q^2(q^2-1)$ branches of this type.   
 \end{proof}

\begin{prop}\label{T5B5}
 An upper unitriangular matrix of type $B_5$ has $q^2$ branches of type $B_5$, $(q^5-q)$ branches of regular type $B_6.$ 
\end{prop}
\begin{proof}
Let $A=\begin{pmatrix} 1&a&&&\\ &1&&&\\ &&1&&\\&&&1&b\\&&&&1\end{pmatrix}, a,b\neq0$ a matrix of type $B_5.$ The  centralizer  $Z_{UT_5}(A)$ of $A$ is $\left\{ \begin{pmatrix} 1& x_0 & x_1& x_2&x_3\\ & 1& &  & \lambda x_2\\ &&1&&z_1 \\ &&&1&w_0\\&&&&1\end{pmatrix} ~|~ \lambda=\frac{b}{a} ,x_i,z_1,w_0\in \Fq\right\}.$ Let $X=\begin{pmatrix} 1& x_0 & x_1& x_2&x_3\\ & 1&&& \lambda x_2 \\ &&1&&z_1 \\ &&&1&w_0\\&&&&1\end{pmatrix}$ be an element of $Z_{UT_5(\Fq)}(A).$ Let $B=\begin{pmatrix} 1& a_0 & a_1& a_2&a_3\\ & 1&&& \lambda a_2 \\ &&1&&c_1 \\ &&&1&d_0\\&&&&1\end{pmatrix},$ and 
  $B^\prime= \begin{pmatrix} 1& a_0^\prime & a_1^\prime& a_2^\prime&a_3^\prime\\ & 1& &&\lambda a_2^\prime \\ &&1&&c_1^\prime \\ &&&1&d_0^\prime\\&&&&1\end{pmatrix}$ be a conjugate of $B$ by $X.$ Thus equating 
  $XB=B^\prime X$ gives us  $a_0=a_0^\prime$\\
   $a_1=a_1^\prime$, $a_2=a_2^\prime$, $c_1=c_1^\prime$, $d_0=d_0^\prime$, and the following equation:
  \begin{equation*}
  x_2d_0+c_1x_1+\lambda a_2x_0+a_3=\lambda x_2a_0^\prime+z_1a_1^\prime+w_0a_2^\prime+a_3^\prime
 \end{equation*}
 We look at three cases, the first case is when $\lambda a_0=d_0$ and $(a_1,a_2,c_1)={\bf 0}.$ The second case is when $\lambda a_0\neq d_0$ and the third case is when $\lambda a_0=d_0$ but $(a_1,a_2,c_1)\neq {\bf 0}.$
 
 {\bf Case: $\lambda a_0=d_0$ and $(a_1,a_2,c_1)={\bf 0}.$ } In this case, we get $a_3=a_3^\prime.$ Therefore $Z_{UT_5(\Fq)}(A,B)=Z_{UT_5(\Fq)}(A).$ So $(A,B)$ is a branch of type $B_5,$ and there are 
 $q^2$ branches.

 {\bf Case: $\lambda a_0\neq d_0$} In this case, we can choose $x_2$ in such a way that we get $a_3=0.$ By routine check, we get 
 $Z_{UT_5(\Fq)}(A,B)$ is group of size $q^5$ isomorphic to centralizer of one of the type $B_6.$  Thus $(A,B)$ is of type $B_6,$ and there are $q^4(q-1)$ branches of this type. 
 
 {\bf Case: $\lambda a_0=d_0$ and $(a_1,a_2,c_1)\neq {\bf 0}.$} In this case, one of $a_1, a_2$ and $c_1$ is non-zero and depending on this, we can choose one of $z_1, w_0$ or $x_1$ suitably in such a way that we get $a_3=0.$ By routine check, we get 
 $Z_{UT_5(\Fq)}(A,B)$ is group of size $q^5$ isomorphic to centralizer of one of the type $B_6.$  Thus $(A,B)$ is of type $B_6,$ and there are $q(q-1)(q^2+q+1)$ branches of this type. 
 
 Therefore a  matrix of type $B_5$ has $q^2$ branches of type $B_5$ and total $q(q^4-1)$ braches of type $B_6.$
 \end{proof}
 \begin{prop}\label{T5B6}
  An upper unitriangular matrix of type $B_6$ has $q^3$ branches of type $B_6$, and $q^2(q^2-1)$ branches of regular type $R_3$.
  \end{prop}
  \begin{proof}
    Let $A=\begin{pmatrix} 1& a &&&\\ &1&b&&\\ &&1&&\\&&&1&&\\&&&&1\end{pmatrix}, a,b\neq0$ a matrix of type $B_6.$ The  centralizer  $Z_{UT_5}(A)$ of $A$ is $\left\{ \begin{pmatrix} 1& x_0 & x_1& x_2&x_3\\ & 1& \lambda x_0&  & \\ &&1&& \\ &&&1&w_0\\&&&&1\end{pmatrix} ~|~ \lambda=\frac{b}{a} ,x_i,w_0\in \Fq\right\}.$ Let $X=\begin{pmatrix} 1& x_0 & x_1& x_2&x_3\\ & 1& \lambda x_0&  & \\ &&1&& \\ &&&1&w_0\\&&&&1\end{pmatrix}$ be an element of $Z_{UT_5(\Fq)}(A).$ Let $B=\begin{pmatrix} 1& a_0 & a_1& a_2&a_3\\ & 1& \lambda a_0&  & \\ &&1&& \\ &&&1&d_0\\&&&&1\end{pmatrix},$ and 
    $B^\prime= \begin{pmatrix} 1& a_0^\prime & a_1^\prime& a_2^\prime&a_3^\prime\\ & 1& \lambda a_0^\prime&  & \\ &&1&& \\ &&&1&d_0^\prime\\&&&&1\end{pmatrix}$ be a conjugate of $B$ by $X.$ Thus equating 
    $XB=B^\prime X$ gives us $a_0=a_0^\prime$, $a_1=a_1^\prime$, $a_2=a_2^\prime$, $d_0=d_0^\prime$, and the following equation:
    \begin{center}
   \begin{tabular}{ccc}
   $x_2d_0+a_3=w_0a_2^\prime+a_3^\prime$\\
   \end{tabular}
   \end{center}
   We look at two cases, when $(a_2,d_0)=(0,0)$ and $(a_2,d_0)\neq(0,0).$
   
   {\bf Case: $(a_2,d_0)=(0,0)$} In this case, we get $a_3=a_3^\prime.$ Therefore $Z_{UT_5(\Fq)}(A,B)=Z_{UT_5(\Fq)}(A).$ So $(A,B)$ is a branch of type $B_6,$ and there are 
   $q$ branches.

   {\bf Case: $(a_2,d_0)\neq(0,0)$} In this case, one of $d_0$ and $a_2$ is non-zero. We can choose $x_2$ or $w_0$ in such a way that we get $a_3=0.$ By routine check, we get 
   $Z_{UT_5(\Fq)}(A,B)$ is commutative group of size $q^4.$ Thus $(A,B)$ is of the regular type $R_3,$ and there are $q^2(q^2-1)$ branches of this type.  
   \end{proof}
 \subsection{Branching of type $D$}
 Now we look at the branching for type $D$.
 \begin{prop}\label{T5D1}
 An upper unitriangular matrix of type $D_1$ has the following branches:
 \begin{center}
 \begin{tabular}{c|c||c|c}\hline
 Branch & No. of Branches &  Branch & No. of Branches\\ \hline
  $D_1$ & $q^2$& $R_2$ & $q^2(q-1)$\\
  $B_4$ & $2q(q-1)$ & $R_3$ & $q^2(q^2-1)$.\\ 
  $UNT_3$ & $q(q-1)^2$ &&\\\hline
 \end{tabular}
 \end{center}
 \end{prop}
\begin{proof}
 An upper unitriangular matrix of type $D_1$ has the canonical form $A = \begin{pmatrix}
1&&c_1&&\\&1&&c_2&\\&&1&&c_3\\&&&1&\\&&&&1
 \end{pmatrix}$, where $a,b,c \neq 0$. $\ZU{5}(A) = \left\{\begin{pmatrix}1&a_1&a_2&b_1&d_1 \\&1&a_3&b_2&d_2\\&&1&\frac{c_2}{c_1}a_1&\frac{c_3}{c_1}a_2\\&&&1&\frac{c_3}{c_2}a_3\\&&&&1\end{pmatrix}\right\}$, which we rewrite as:
 $$\ZU{5}(A) = \left\{\begin{pmatrix}C & \overrightarrow{b} &\overrightarrow{d}\\& 1& \frac{c_3}{c_2}C_{23}\\&&1 \end{pmatrix}\mid C\in UT_3(\Fq), \overrightarrow{b} = \left(\begin{smallmatrix}b_1\\b_2\\\frac{c_2}{c_1}C_{12}\end{smallmatrix}\right), \overrightarrow{d} =\left(\begin{smallmatrix}d_1\\d_2\\\frac{c_3}{c_1}C_{13}\end{smallmatrix}\right) \right\}.$$
Let $B = \begin{pmatrix}C & \overrightarrow{b} &\overrightarrow{d}\\& 1& \frac{c_3}{c_2}C_{23}\\&&1 \end{pmatrix}$, and $B' =\begin{pmatrix}C' & \overrightarrow{b'} &\overrightarrow{d'}\\& 1& \frac{c_3}{c_2}C'_{23}\\&&1 \end{pmatrix}$ be a conjugate of $B$ by a member $X = \begin{pmatrix}Z & \overrightarrow{y} &\overrightarrow{w}\\& 1& \frac{c_3}{c_2}Z_{23}\\&&1 \end{pmatrix} \in \ZU{5}(A)$, with $\overrightarrow{y} =\left(\begin{smallmatrix}y_1\\y_2\\\frac{c_2}{c_1}Z_{12}\end{smallmatrix}\right) $, and $\overrightarrow{w} = \left(\begin{smallmatrix}w_1\\w_2\\\frac{c_3}{c_1}Z_{13}\end{smallmatrix}\right)$. We thus have $XB = B'X$. First thing we see is that $ZC = C'Z$. So we can take $C$ to be a conjugacy class representative in $UT_3(\Fq)$, and we thus have the following equations:
\begin{eqnarray}
 Z\overrightarrow{b} + \overrightarrow{y} &=& C\overrightarrow{y} + \overrightarrow{b'} \label{ED11}\\
 Z\overrightarrow{d} + \frac{c_3}{c_2}C_{23}\overrightarrow{y} + \overrightarrow{w} &=& C'\overrightarrow{w} + \frac{c_3}{c_2}Z_{23}\overrightarrow{b'} + \overrightarrow{d'}\label{ED12}
\end{eqnarray}

\noindent{\bfseries When $C = I_3$:} In this case $C_{12} = C_{13} = C_{23} = 0$. We have $Z = \begin{pmatrix}1&z_1&z_2\\&1&z_3\\&&1\end{pmatrix}$. 
Equation~\ref{ED11} becomes:
$\begin{pmatrix}b_1 + z_1b_2\\ b_2\\ 0 \end{pmatrix} = \begin{pmatrix}b'_1\\ b'_2\\ 0 \end{pmatrix}$. We look at two cases here: When $b_2 \neq 0$, and when $b_2 = 0$.

When $b_2 = 0$, We have $b'_1 = b_1$, and Equation~\ref{ED12} becomes:
$$\begin{pmatrix}d_1 + z_1d_2\\ d_2\\ 0\end{pmatrix} =  \begin{pmatrix}d'_1 + \frac{c_3}{c_2}z_3b_1\\d'_2 +\frac{c_3z_3}{c_2}b'_2\\ 0\end{pmatrix}$$ We have $d'_2 = d_2$. 

When $b_2 = b_1 = d_2 = 0$: We have $d'_1 = d_1$. Thus $B$ is reduced to $\begin{pmatrix}1&&&&d_1\\&1&&&\\&&1&&\\&&&1&\\&&&&1\end{pmatrix}$. So $\ZU{5}(A,B) = \ZU{5}(A)$. Hence $(A,B)$ is a branch of type $D_1$, and there are $q$ such branches.

When $b_2 = b_1 = 0$, and $d_2 \neq 0$, we can choose $z_1$ such that $d_1 = 0$. Thus, $B$ is reduced to $\begin{pmatrix}1&&&&\\&1&&&d_2\\&&1&&\\&&&1&\\&&&&1\end{pmatrix}$, and $\ZU{5}(A,B) = \left\{\begin{pmatrix}1&&z_2&y_1&w_1\\&1&z_3&y_2&w_2\\&&1&&\frac{c_3}{c_1}z_2\\&&&1&\frac{c_3}{c_2}z_3\\&&&&1\end{pmatrix}\right\}$, which is of type ~~~. So $(A,B)$ is a branch of type $B_4$, as $\ZU{5}(A,B)$ can be conjugated by the elementary matrix that swaps rows and columns 1 and 2 to get the centralizer subgroup of one of the canonical matrices of type $B_4$, and there are $(q-1)$ branches of this type.

When $b_1 \neq 0$, in Equation~\ref{ED12}, we choose $z_3$ so that $d_1 = 0$. Thus $B$ is reduced to $\begin{pmatrix}1& & &b_1&\\&1&&&d_2\\&&1&&\\&&&1&\\&&&&1\end{pmatrix}$, and $\ZU{5}(A,B) = \left\{\begin{pmatrix}1 & z_1 & z_2 & y_1 & w_1\\ &1 &\frac{c_2d_2}{c_3b_1}z_1&y_2&w_2\\&&1&\frac{c_2}{c_1}z_1&\frac{c_3}{c_1}z_2\\&&&1&\frac{d_2}{b_1}z_1\\&&&&1 \end{pmatrix}\right\}$. Again, we have 2 cases here:

When $d_2 = 0$, $B = \begin{pmatrix}1& & &b_1&\\&1&&&\\&&1&&\\&&&1&\\&&&&1\end{pmatrix}$. Here $\ZU{5}(A,B) = \left\{\begin{pmatrix}1 & z_1 & z_2 & y_1 & w_1\\ &1 &&y_2&w_2\\&&1&\frac{c_2}{c_1}z_1&\frac{c_3}{c_1}z_2\\&&&1&\\&&&&1 \end{pmatrix}\right\}$. On conjugating by an elementary matrix, which swaps rows and columns 2 and 3 of each element of $\ZU{5}(A,B)$, we get the centralizer of one of the canonical matrices of the type $B_4$. Thus there are $q-1$ branches of type $B_4$.

When $d_2 \neq 0$, we have $\ZU{5}(A,B) = \left\{\begin{pmatrix}1 & z_1 & z_2 & y_1 & w_1\\ &1 &\frac{c_2d_2}{c_3b_1}z_1&y_2&w_2\\&&1&\frac{c_2}{c_1}z_1&\frac{c_3}{c_1}z_2\\&&&1&\frac{d_2}{b_1}z_1\\&&&&1 \end{pmatrix}\right\}$. Thus this branch is of the  new type $UNT_3$, and there are $(q-1)^2$ such branches.

When $b_2 \neq 0$, choose $z_1$ such that $b'_1 = 0$. Thus equating Equation~\ref{ED11} with $b_1$ replaced by $0$, we get that $z_1 = 0$. Thus with $b_1 = 0$ and $z_1 = 0$, we get from Equation~\ref{ED12}, $d'_1 = d_1$, and with a nice choice of $z_3$, we can reduce $d'_2$ to 0. Hence, $B$ is reduced to 
$\begin{pmatrix}
1&&&&d_1\\&1&&b_2&\\&&1&&\\&&&1&\\&&&&1                                              \end{pmatrix}$, and $\ZU{5} = \left\{\begin{pmatrix}1 &  & z_2 & y_1 & w_1\\ &1 &&y_2&w_2\\&&1&&\frac{c_3}{c_1}z_2\\&&&1&\\&&&&1 \end{pmatrix}\right\}$, which is a centralizer of type $R_2$. Thus $(A,B)$ is a branch of type $R_2$, and there are $q(q-1)$ such branches.

{\bfseries When $C= \left(\begin{smallmatrix}1&&c\\&1&\\&&1\end{smallmatrix}\right), c\neq 0$ :} Here Equation~\ref{ED11} becomes: $\begin{pmatrix}b_1 + z_1b_2\\b_2\\0\end{pmatrix} = \begin{pmatrix}b'_1 + \frac{c_3c}{c_2}z_1\\b'_2\\0\end{pmatrix}$. So we have $b'_2 =  b_2$. We see 2 cases here: $b_2 = \frac{c_2}{c_1}c$, and $b_2 \neq\frac{c_2}{c_1}c$. 

When $b_2 \neq \frac{c_2}{c_1}c$. In the above equation, we choose $z_1$ such that $b'_1 = 0$. Thus, with substituting $b_1$ with $b'_1 = 0$ in the above equation, we get $z_1 = 0$. Thus, with this, Equation~\ref{ED12} becomes
$\begin{pmatrix}d_1 \\ d_2 + \frac{c_3}{c_1}cz_3\\\frac{c_3}{c_1}c \end{pmatrix} = \begin{pmatrix}d'_1 \\d'_2+ \frac{c_3}{c_2}b_2z_3\\\frac{c_3}{c_1}c\end{pmatrix}.$
As $b_2\neq \frac{c_2}{c_1}c$, we can choose a $z_3$ so that $d'_2  = 0$, and we have $d'_1 = d_1$. So $B$ boils down to $\begin{pmatrix}1&&c&&d_1\\&1&&b_2&\\&&1&&\frac{c_3}{c_1}c\\&&&1&\\&&&&1\end{pmatrix}$, with $\ZU{5}(A,B) = \left\{ \begin{pmatrix}1 &  & z_2 & y_1 & w_1\\ &1 &&y_2&w_2\\&&1&&\frac{c_3}{c_1}z_2\\&&&1&\\&&&&1 \end{pmatrix} \right\}$. Thus $(A,B)$ too is a branch of type $R_2$, and there are $q(q-1)^2$ such branches.

When $b_2 = \frac{c_2}{c_1}c$, we get from Equation~\ref{ED11}, $b'_1 = b_1$. Equation~\ref{ED12} boils down to: $\begin{pmatrix}d_1 + z_1d_2\\ d_2\\0\end{pmatrix} = \begin{pmatrix}d'_1 + \frac{c_3}{c_2}z_3b_1\\ d'_2\\0\end{pmatrix}$. So we have $d'_2 = d_2$. We look first at $b_1 = d_2 = 0$. $B$ is reduced to $\begin{pmatrix}1&&c&&d_1\\&1&&\frac{c_2}{c_1}c&\\&&1&&\frac{c_3}{c_1}c\\&&&1&\\&&&&1\end{pmatrix}$, and $\ZU{5}(A,B) = \ZU{5}(A)$. Thus $(A,B)$ is a branch of type $D_1$, and there are $q(q-1)$ such branches.

When $b_1 \neq 0$ choose $z_3$ such that $d'_1 = 0$. So, $B$ becomes:
$\begin{pmatrix}1&&c&b_1&\\&1&&\frac{c_2}{c_1}c&d_2\\&&1&&\frac{c_3}{c_1}c\\&&&1&\\&&&&1\end{pmatrix}$. We have two cases here:

When $d_2 = 0$, we have $B = \begin{pmatrix}1&&c&b_1&\\&1&&\frac{c_2}{c_1}c&\\&&1&&\frac{c_3}{c_1}c\\&&&1&\\&&&&1\end{pmatrix}$ and $$\ZU{5}(A,B) = \left\{ \begin{pmatrix}1 & z_1 & z_2 & y_1 & w_1\\ &1 &&y_2&w_2\\&&1&\frac{c_2}{c_1}z_1&\frac{c_3}{c_1}z_2\\&&&1&\\&&&&1 \end{pmatrix} \right\},$$ thus $(A,B)$ is of a type $B_4$, and there are $(q-1)^2$ such branches.

When $d_2 \neq 0$, $B =\begin{pmatrix}1&&c&b_1&\\&1&&\frac{c_2}{c_1}c&d_2\\&&1&&\frac{c_3}{c_1}c\\&&&1&\\&&&&1\end{pmatrix} $, and $$\ZU{5}(A,B) = \left\{ \begin{pmatrix}1 & z_1 & z_2 & y_1 & w_1\\ &1 &\frac{c_2d_2}{c_3b_1}z_1&y_2&w_2\\&&1&\frac{c_2}{c_1}z_1&\frac{c_3}{c_1}z_2\\&&&1&\frac{d_2}{b_1}z_1\\&&&&1 \end{pmatrix} \right\},$$ so, this branch too is of the type $UNT_3$. Thus there are $(q-1)^3$ branches of this new type.

When $b_1 = 0$, and $d_2 \neq 0$. We choose $z_1$ so that $d'_1 = 0$. Thus $B$ is reduced to $\begin{pmatrix}1&&c&&\\&1&&\frac{c_2}{c_1}c&d_2\\&&1&&\frac{c_3}{c_1}c\\ &&&1&\\&&&&1\end{pmatrix}$, and $\ZU{5} = \left\{\begin{pmatrix}1&&z_2&y_1&w_1\\&1&z_3&y_2&w_2\\&&1&&\frac{c_3}{c_1}z_2\\&&&1&\frac{c_3}{c_2}z_3\\&&&&1\end{pmatrix}\right\}$. This is of type $B_4$. $(A,B)$ is a branch of type $B_4$, and there are $(q-1)^2$ such branches.

\noindent{\bfseries When $C= \left(\begin{smallmatrix}1&c&\\&1&\\&&1\end{smallmatrix}\right), c\neq 0$:} Here, $Z = \begin{pmatrix}1&z_1&z_2\\&1&\\&&1\end{pmatrix}$. Equation~\ref{ED11} boils down to $\begin{pmatrix}b_1 +b_2z_1 + \frac{c_2}{c_1}cz_2 \\ b_2\\\frac{c_2}{c_1}c\end{pmatrix} = \begin{pmatrix}cy_2 + b'_1\\b'_2\\\frac{c_2}{c_1}c\end{pmatrix}$. So $b'_2 = b_2$. As $c \neq 0$, we choose $y_2$ such that $b'_1 = 0$. Equation~\ref{ED12} becomes:
$\begin{pmatrix}d_1 +d_2z_1  \\ d_2\\0\end{pmatrix} = \begin{pmatrix}cw_2 + d'_1\\d'_2\\0\end{pmatrix}$. We have $d'_2 = d_2$. Take $w_2$ such that $d'_1 = 0$. So $B$ is reduced to $\begin{pmatrix}1&c&&&\\&1&&b_2&d_2\\&&1&\frac{c_2}{c_1}c&\\&&&1&\\&&&&1\end{pmatrix}$, and therefore $\ZU{5}(A,B) = \left\{\begin{pmatrix}1&z_1&z_2&y_1&y_2\\&1&&\frac{b_2}{c}z_1 + \frac{c_2}{c_1}z_2 &\frac{d_2}{c}z_1 \\&&1&\frac{c_2}{c_1}z_1&\frac{c_3}{c_1}z_2\\&&&1&\\&&&&1\end{pmatrix}\right\}$, which is of size $q^4$. It is routine to check that this centralizer is commutative. Thus this is a centralzier of type $R_3$. Thus $(A,B)$ is a branch of type $R_3$, and there are $q^2(q-1)$ such branches. 

\noindent{\bfseries When $C= \left(\begin{smallmatrix}1&&\\&1&c\\&&1\end{smallmatrix}\right), c\neq 0$:} In this case $Z = \begin{pmatrix}1&&z_2\\&1&z_3\\&&1\end{pmatrix}$. With this, Equation~\ref{ED11} becomes $\begin{pmatrix}b_1\\b_2\\0\end{pmatrix}=\begin{pmatrix}b'_1\\b'_2\\0\end{pmatrix}$. So, our focus thus is solely on Equation~\ref{ED12}. The equation is reduced to $\begin{pmatrix}d_1+\frac{c_3}{c_2}cy_1\\ d_2 + \frac{c_3}{c_2}cy_2\\0\end{pmatrix} = \begin{pmatrix}d'_1 + \frac{c_3}{c_2}b_1z_3\\ d'_2 + \frac{c_3}{c_1}cz_2 + \frac{c_3}{c_2}b_2z_3\\0\end{pmatrix}$

As $\frac{c_3}{c_2}c\neq 0$, choose $y_1$, $y_2 $ so that $d'_1 = d'_2 = 0$. Thus $B$ is reduced to $\begin{pmatrix}1&&&b_1&\\&1&c&b_2&\\&&1&&\\&&&1&\frac{c_3}{c_2}c\\&&&&1\end{pmatrix}$, and $\ZU{5}(A,B)= \left\{ \begin{pmatrix}1&&z_2&\frac{b_1}{c}z_3&w_1\\&1&z_3&\frac{c_2}{c_1}z_2+\frac{b_2}{c}z_3&w_2\\&&1&&\frac{c_3}{c_1}z_2\\&&&1&\frac{c_3}{c_2}z_3\\&&&&1\end{pmatrix} \right\}$. This is of size $q^4$, and with a routine check we see that it is commutative. This is a centralizer of type $R_3$, hence $(A,B)$ is a branch of type $R_3$, and there are $q^2(q-1)$ such branches.

\noindent{\bfseries When $C= \left(\begin{smallmatrix}1&c_0&\\&1&d_0\\&&1\end{smallmatrix}\right), c_0,d_0\neq 0$:} Here $Z = \begin{pmatrix}1&z_1&z_2\\&1&\lbd_0z_1\\&&1\end{pmatrix}$, where $\lbd_0 = \frac{d_0}{c_0}$. Equation~\ref{ED11} becomes:
$\begin{pmatrix}b_1+z_1b_2 + z_2\frac{c_2}{c_1}c_0\\b_2 + \frac{c_2}{c_1}d_0z_1\\ 0\end{pmatrix} = \begin{pmatrix}c_0y_2 +b'_1\\\frac{c_2}{c_1}d_0z_1 + b'_2\\ 0\end{pmatrix}$. As $c_0$ and $d_0$ are non-zero, we have $b'_2 = b_2$. We choose $y_2$ such that $b'_1 = 0$. Hence, on replacing $b_1$ with $0$ in the above equation we get $y_2 = \frac{b_2}{c_0}z_1+ \frac{c_2}{c_1}z_2$. With these, Equation~\ref{ED12} boils down to $\begin{pmatrix}d_1+z_1d_2+\frac{c_3}{c_2}d_0y_1\\d_2\\0\end{pmatrix}=\begin{pmatrix}c_0w_2 + d'_1\\d'_2\\0\end{pmatrix}$. So $d'_2 = d_2$, and choose $w_2$ such that $d'_1 = 0$. Hence, $B$ is reduced to $\begin{pmatrix}1&c_0&&&\\&1&d_0&b_2&d_2\\&&1&\frac{c_2}{c_1}c_0&\\&&&1&\frac{c_3}{c_2}d_0\\&&&&1\end{pmatrix}$, with $\ZU{5}(A,B) = \left\{\begin{pmatrix}1&z_1&z_2&y_1&w_1\\&1&\frac{d_0}{c_0}z_1 &\frac{b_2}{c_0}z_1+\frac{c_2}{c_1}z_2 & \frac{c_3}{c_2c_0}y_1 + \frac{d_2c_0}{d_0}z_1 \\&&1&\frac{c_2}{c_1}z_1&\frac{c_3}{c_1}z_2\\&&&1&\frac{c_3d_0}{c_2c_0}z_1\\&&&&1\end{pmatrix}\right\}$. This too is of type $R_3$. So $(A,B)$ is a brach of type $R_3$, and there are $q^2(q-1)^2$ such branches.

So, on adding up the branches of each of the types, we have
\begin{itemize}
\item $q^2$ branches of type $D_1$,
\item $2q(q-1)$ branches of type $B_4$,
\item $q^2(q-1)$ branches of type $R_2$,
\item $q^2(q^2-1)$ branches of type $R_3$, and 
\item $q(q-1)^2$ branches of type $UNT_3$. 
\end{itemize}
These match with the estimations done for $q = 3$ in GAP.
\end{proof}

 
 \begin{prop}\label{T5D2}
 An upper unitriangular matrix of type $D_2$ has $q^3$ branches of type $D_2$, and $q^2(q^2-1)$ branches of regular type $R_3$.
 \end{prop}
 \begin{proof}
   Let $A=\begin{pmatrix} 1& a &&&\\ &1&b&&\\ &&1&&c\\&&&1&&\\&&&&1\end{pmatrix}, a,b,c\neq0$ a matrix of type $D_2.$ The  centralizer $Z_{UT_5}(A)$ of $A$ is $\left\{ \begin{pmatrix} 1& x_0 & x_1& x_2&x_3\\ & 1& \lbd_1x_0&  & \lbd_2x_1\\ &&1&& \lbd_2x_0\\ &&&1&w_0\\&&&&1\end{pmatrix} ~|~ \lambda_1=\frac{b}{a} , \lambda_2=\frac{c}{a} ,x_i,w_0\in \Fq\right\}.$ 
   
   Let $X=\begin{pmatrix} 1& x_0 & x_1& x_2&x_3\\ & 1& \lambda_1x_0&  & \lambda_2x_1\\ &&1&& \lambda_2x_0\\ &&&1&w_0\\&&&&1\end{pmatrix}$ be an element of $Z_{UT_5(\Fq)}(A).$ Let $B=\begin{pmatrix} 1& a_0 & a_1& a_2&a_3\\ & 1& \lambda_1a_0&  & \lambda_2a_1\\ &&1&& \lambda_2a_0\\ &&&1&d_0\\&&&&1\end{pmatrix},$ and 
   $B^\prime= \begin{pmatrix} 1& a_0^\prime & a_1^\prime& a_2^\prime&a_3^\prime\\ & 1& \lambda_1a_0^\prime&  & \lambda_2a_1^\prime\\ &&1&& \lambda_2a_0^\prime\\ &&&1&d_0^\prime\\&&&&1\end{pmatrix} = XBX^{-1}$. Thus equating 
   $XB=B^\prime X$ gives us $a_0=a_0^\prime$, $a_1=a_1^\prime$, $a_2=a_2^\prime$, $d_0=d_0^\prime$, and the following equation:
   \begin{center}
  \begin{tabular}{ccc}
    $x_2d_0+a_3=w_0a_2^\prime+a_3^\prime$\\
  \end{tabular}
  \end{center}
  We look at two cases, when $(a_2,d_0)=(0,0)$ and $(a_2,d_0)\neq(0,0).$
  
  {\bf Case: $(a_2,d_0)=(0,0)$} In this case, we get $a_3=a_3^\prime.$ Therefore $Z_{UT_5(\Fq)}(A,B)=Z_{UT_5(\Fq)}(A).$ So $(A,B)$ is a branch of type $D_2,$ and there are 
  $q$ branches.

  {\bf Case: $(a_2,d_0)\neq(0,0)$} In this case, one of $d_0$ and $a_2$ is non-zero. We can choose $x_2$ or $w_0$ in such a way that we get $a_3=0.$ By routine check, we get 
  $Z_{UT_5(\Fq)}(A,B)$ is commutative group of size $q^4.$ Thus $(A,B)$ is of the regular type $R_3,$ and there are $q^2(q^2-1)$ branches of this type.   
  \end{proof}
 
 \begin{prop}\label{T5R}
 A matrix of the $R_1$ type has $q^6$ branches of type $R_1,$ a matrix of the $R_2$ type has $q^5$ branches of type $R_2,$ and
 a matrix of the $R_3$ type has $q^4$ branches of type $R_3.$
\end{prop}
\begin{proof}
 The type $R_1, R_2$ and $R_3$ are $\Reg$ types, hence the centralizer of matrices of such a type is a commutative.
\end{proof}

\subsection{Branching Rules for the New Types} While determining the branching rules for the types in $UT_5(\Fq)$, we observed that there are some commuting pairs of elements of $UT_5(\Fq)$, which are not isomorphic to the centralizers of any of the elements in $UT_5(\Fq)$. Thus, giving rise to what we call ``new types''. The new types, we have seen so far are $UNT_1$ (first observed in Proposition~\ref{T5A3}), $UNT_2$ (observed in Proposition~\ref{T5B1}) and $UNT_3$ (observed in Propositions~\ref{T5B1}). Now, we compute the branching for these cases and we see that no further new types occur.

\begin{prop}\label{T5UN1}
The new type $UNT_1$ has $q^3$ branches of type $UNT_1$, $q^2(q^2-1)$ branches of type $R_2$, and $q^4-q^3$ branches of type $R_3$.
\end{prop}
\begin{proof}For some pair $(A,B)$ of commuting elements in $UT_5(\Fq)$, of type $UNT_1$, the centralizer subgroup is $\ZU{5}(A,B) = \left\{\begin{psmallmatrix}1&x_0&x_1&\lambda z_0 &x_3\\&1&y_0&&y_2\\&&1&& \\&&&1&z_0\\&&&&1 \end{psmallmatrix}\right\}$, where $\lbd \neq 0$ is fixed. Let $C = \begin{pmatrix}1& a_0& a_1 & \lbd c_0& a_3\\ &1&b_0& & b_2\\ &&1&& \\&&&1& c_0\\&&&&1\end{pmatrix}$, and let $C' = \begin{pmatrix}1& a'_0& a'_1 & \lbd c'_0& a'_3\\ &1&b'_0& & b'_2\\ &&1&& \\&&&1&c_0\\&&&&1\end{pmatrix} = XCX^{-1}$, where $X =\begin{pmatrix}1&x_0&x_1&\lambda z_0 &x_3\\&1&y_0&&y_2\\&&1&& \\&&&1&z_0\\&&&&1 \end{pmatrix}$. On equation $XC = CX$, we get $a'_0 = a_0$, $b'_0 = b_0$, $b'_2 = b_2$, $c'_0 = c_0$, and the following equations:

\begin{eqnarray}
a_1 + x_0b_0 &=& a'_1 + a_0y_0 \label{EU11}\\
a_3 + x_0b_2  &=& a'_3 +y_2a_0 \label{EU12}
\end{eqnarray}

We look at two main cases: $(a_0,b_2) = (0,0)$, and $(a_0,b_2)\neq (0,0)$.

\noindent{\bfseries When $a_0 = b_2 = 0$:} Equation~\ref{EU12} becomes $a'_3 = a_3$, Equation~\ref{EU11} becomes $a'_1 = a_1 + x_0b_0$. We have two subcases here:

When $b_0 =0$, then we get $a'_1 = a_1$. Thus $C$ boils down to $\begin{pmatrix}1&&a_1&\lbd c_0&a_3\\&1&&&\\&&1&& \\ &&&1&c_0\\&&&&1\end{pmatrix}$, and $\ZU{5}(A,B,C) = \ZU{5}(A,B)$. $(A,B,C)$ is therefore of type $UNT_1$, and there are $q^3$ such branches. 

When $b_0 \neq 0$, in Equation~\ref{EU11}, we can choose $x_0$ such that $a'_1 = 0$. Hence $C$ is reduced to $\begin{pmatrix}1& & &\lbd c_0&a_3\\&1&b_0&&\\ & & 1&&\\ &&&1&c_0 \\ &&&&1\end{pmatrix}$, and $\ZU{5}(A,B,C) = \left\{\begin{pmatrix}1& & x_1 & \lbd z_0& x_3\\ &1&y_0& & y_2\\ &&1&&\\&&&1&z_0\\&&&&1\end{pmatrix} \right\}$. Easy to see that this is a commutative group of size $q^5$. $(A,B,C)$ is a branch of type $R_2$, and there are $q^2(q-1)$ such branches.

\noindent{\bfseries When $(a_0,b_2)\neq (0,0)$:} When $a_0 \neq 0$, in Equation~\ref{EU11}, we choose $y_0$ such that $a'_1 = 0$. Thus, on replacing $a_1$ with $a'_1 = 0$ in that equation, we get $y_0 = \frac{b_0}{a_0}x_0$. In Equation~\ref{EU12} choose $y_2$ so that $a'_3 = 0$. Thus $C$ is reduced to $\begin{pmatrix}1&a_0&&\lbd c_0&\\&1&b_0&&b_2\\&&1&&\\ & & & 1&c_0\\&&&&1\end{pmatrix}$, and thus $\ZU{5}(A,B,C)= \left\{\begin{pmatrix}1& x_0& x_1 & \lbd z_0& x_3\\ &1&\frac{b_0}{a_0}x_0 & & \frac{b_2}{a_0}x_0 \\ &&1&& y_1\\&&&1&z_0\\&&&&1\end{pmatrix} \right\}$. Easy to see that this subgroup is a commutative one of size $q^4$. Thus $(A,B,C)$ is a branch of type $R_3$, and there are $(q-1)q^3 = q^4-q^3$ such branches. 

When $a_0 = 0$, and $b_2 \neq 0$. Equation~\ref{EU12} becomes $a_3 + x_0b_2 = a'_3$, and  Choose $x_0$ such that $a'_3 = 0$. Then, on replacing $a_3$ with $a'_3 = 0$ in Equation~\ref{EU12}, we get $x_0 = 0$. With these, Equation~\ref{EU11} becomes $a'_1 = a_1$. $C$ thus boils down to $\begin{pmatrix}
1 & &a_1 &\lbd c_0 &\\ 
&1&b_0&&b_2\\
&&1&&\\
&&&1&c_0\\
&&&&1
\end{pmatrix}$, and $\ZU{5}(A,B,C) = \left\{\begin{pmatrix}1& & x_1 &\lbd z_0& x_3\\ &1&y_0& & y_2\\ &&1&&\\&&&1&z_0\\&&&&1\end{pmatrix} \right\}$. This branch too is of type $R_2$, and there are $q^3(q-1)$ such branches. So, in total there are $q^3(q-1) + q^2(q-1) = q^4 -q^2$ branches of type $R_2$.\end{proof}
\begin{prop}\label{T5UN2}
The new type $UNT_2$ has $q^3$ branches of type $UNT_2$, $q^5-q^2$ branches of type $R_1$, and $q^4-q^3$ branches of type $R_3$.
\end{prop}
\begin{proof}
A commuting pair $(A,B)$ of type $UNT_2$ has the centralizer $\left\{\begin{pmatrix}1&x_1&y_1&y_2&x_2\\&1&z_1&z_2&w_1\\&&1&&\lbd x_1\\&&&1&x_1\\&&&&1\end{pmatrix} \right\}$. 

Let $C = \begin{pmatrix}1&a_1&b_1&b_2&a_2\\&1&c_1&c_2&d_1\\&&1&&\lbd a_1\\&&&1&a_1\\&&&&1\end{pmatrix}$, $C' = \begin{pmatrix}1&a'_1&b'_1&b'_2&a'_2\\&1&c'_1&c'_2&d_1\\&&1&&\lbd a'_1\\&&&1&a'_1\\&&&&1\end{pmatrix}$ be a conjugate of $C$, and let $X = \begin{pmatrix}1&x_1&y_1&y_2&x_2\\&1&z_1&z_2&w_1\\&&1&&\lbd x_1\\&&&1&x_1\\&&&&1\end{pmatrix}$ such that $XC = C'X$. Equating $XC = C'X$ gives us $a'_1 = a_1$, $c'_1 = c_1$ and $c'_2 = c_2$, and the following bunch of equations:
\begin{eqnarray}
\begin{pmatrix}a_1 & b_1 + x_1c_1 & b_2 + x_1c_2\end{pmatrix} &=& \begin{pmatrix}a_1 & b'_1 + a_1z_1 & b'_2 + a_1z_2\end{pmatrix}\label{EU21}\\
d_1 + (\lbd z_1 + z_2)a_1 &=& d'_1 + (\lbd c_1 + c_2)x_1 \label{EU22}\\
a_2 + x_1d_1 + (\lbd y_1+y_2)a_1 &=& a'_2 + (\lbd b'_1 + b'_2)x_1+ w_1 a_1 \label{EU23}
\end{eqnarray}
There are two main cases here:

\noindent{\bfseries Case: $a_1 =0$}

When $c_1 = c_2 = 0$, Equation~\ref{EU21} leads us to $b'_1 = b_1$, $b'_2 = b_2$, and from Equation~\ref{EU22} $d'_1 = d_1$. With these, Equation~\ref{EU23} becomes $a_2 + x_1d_1 = a'_2 + (\lbd b_1 + b_2) x_1$. 

When $d_1 = \lbd b_1 + b_2$, we get $a'_2 =a_2$. Thus $C$ is reduced to $\begin{pmatrix}1&&b_1&b_2&a_2\\&1&&&\lbd b_1 + b_2 \\ &&1&&\\&&&1&\\&&&&1\end{pmatrix}$, and $\ZU{A,B,C} = \ZU{5}(A,B)$. Thus, $(A,B,C)$ is of type $UNT_2$, and there are $q^3$ such branches.

When $d_1 \neq \lbd b_1 + b_2$, we can choose $x_1$ such that $a'_2 = 0$. Thus $C$ is reduced to $\begin{pmatrix}
1& & b_1& b_2 &\\
&1&&&d_1\\
&&1&&\\
&&&1&\\
&&&&1
\end{pmatrix}$, and $\ZU{5}(A,B,C)=\begin{pmatrix}1&&y_1&y_2&x_2\\&1&z_1&z_2&w_1\\&&1&&\\&&&1&\\&&&&1\end{pmatrix}$. Thus $(A,B,C)$ is of type $R_1$, and there are $q^2(q-1)$ such branches.

When $c_1 \neq 0$, in Equation~\ref{EU21}, we can choose $x_1$ so that $b'_1 = 0$. Thus on replacing $b_1$ with $b'_1 = 0$, we get $x_1 = 0$, and thus $b'_2 = b_2$. And Equation~\ref{EU22} reduces to $d'_1 = d_1$, and Equation~\ref{EU23} boils down to $a'_2 = a_2$. $C$ is reduced to 
$\begin{pmatrix}
1& & & b_2 & a_2\\
&1&c_1&c_2&d_1\\
&&1&&\\
&&&1&\\
&&&&1
\end{pmatrix}$, and $\ZU{5}(A,B,C) = \begin{pmatrix}1&&y_1&y_2&x_2\\&1&z_1&z_2&w_1\\&&1&&\\&&&1&\\&&&&1\end{pmatrix}$. $(A,B,C)$ is thus of type $R_2$, and there are $(q-1)q^4$ such branches.

When $c_1 = 0$, and $c_2 \neq 0$. In Equation~\ref{EU21}, we get $b'_1 = b_1$, and choose $x_1$ such that $b'_2 = 0$. Hence on substituting $b_2$ with $b'_2 = 0$ and equating Equation~\ref{EU21}, we get $x_1 = 0$. With this Equation~\ref{EU22} boils down to $d'_1 = d_1$, and Equation~\ref{EU23} boils down to $a'_2 = a_2$. $C$ is reduced to
$\begin{pmatrix}
1&&b_1&&a_2\\
&1&&c_2&d_1\\
&&1&&\\
&&&1&\\
&&&&1
\end{pmatrix}$, and $Z(A,B,C) = \begin{pmatrix}1&&y_1&y_2&x_2\\&1&z_1&z_2&w_1\\&&1&&\\&&&1&\\&&&&1\end{pmatrix}$. $(A,B,C)$ is a branch of type $R_1$, and there are $q^3(q-1)$ such branches.

\noindent{\bfseries Case $a_1\neq 0$:} In this case, in Equation~\ref{EU21}, we choose $z_1$ and $z_2$ such that $b'_1 = 0$ and $b'_2 = 0$ respectively. Thus, on replacing $b_1$ by $b'_1 = 0$, and $b_2$ by $b'_2 = 0$ in Equation~\ref{EU21}, and equating, we get $z_1 = \frac{c_1}{a_1}x_1$ and $z_2 = \frac{c_2}{a_1}x_1$. Putting these in Equation~\ref{EU22} leads us to $d_1 +\left(\lbd\frac{c_1}{a_1}x_1 +\frac{c_2}{a_1}x_1 \right)a_1 = d'_2 + (\lbd c_1 + c_2)x_1$. Thus $d'_1 = d_1$.

With all this, Equation~\ref{EU23} boils down to $a_2 + x_1d_1 + (\lbd y_1 + y_2)a_1 = a'_2 + w_1a_1$. Choose $w_1$ so that $a'_2 = 0$. Hence $C$ is reduced to 
$\begin{pmatrix}
1&a_1&&&\\
&1&c_1&c_2&d_1\\
&&1&&\lbd a_1\\
&&&1&a_1\\
&&&&1
\end{pmatrix}$, and 
$$\ZU{5}(A,B) = \begin{pmatrix}1&x_1&y_1&y_2&x_2\\&1&\frac{c_1}{a_1}x_1&\frac{c_2}{a_1}x_1&\lbd y_1 + y_2 + \frac{d_1}{a_1}x_1\\&&1&&\lbd x_1\\&&&1&x_1\\&&&&1\end{pmatrix}.$$
Easy to check that the above centralizer subgroup is a commutative one, of size $q^4$. Thus $(A,B,C)$ is of type $R_3$, and there are $(q-1)q^3$ such branches. 

Adding up all the branches of type $R_1$ gives us $q^2(q-1) + q^3(q-1) + q^4(q-1) = q^5 - q^2$ branches of type $R_1$.
\end{proof}

\begin{prop}\label{T5UN3}
The new type $UNT_3$ has $q^3$ branches of type $UNT_3$, $q^4-q^2$ branches of type $R_2$, and $q^4 - q^3$ branches of type $R_3$.
\end{prop}
\begin{proof}
A commuting pair $(A,B)$ of matrices in $UT_5(\Fq)$ of type $UNT_3$ has as its centralizer:
$\left\{\begin{pmatrix}1&x_1&y_1&y_2&x_2\\&1&\lbd_1 x_1&z_2&w_1\\&&1&\lbd_2 x_1 & \frac{\lbd_2}{\lbd_1}y_1\\&&&1&x_1\\&&&&1\end{pmatrix} \right\}$. 

Let $C = \begin{pmatrix}1&a_1&b_1&b_2&a_2\\&1&\lbd_1 a_1&c_2&d_1\\&&1&\lbd_2 a_1 & \frac{\lbd_2}{\lbd_1}b_1\\&&&1&a_1\\&&&&1\end{pmatrix}$, and $C' = \begin{pmatrix}1&a'_1&b'_1&b'_2&a'_2\\&1&\lbd_1 a'_1&c'_2&d'_1\\&&1&\lbd_2 a'_1 & \frac{\lbd_2}{\lbd_1}b'_1\\&&&1&a'_1\\&&&&1\end{pmatrix} = XCX^{-1}$, where $X = \begin{pmatrix}1&x_1&y_1&y_2&x_2\\&1&\lbd_1 x_1&z_2&w_1\\&&1&\lbd_2 x_1 & \frac{\lbd_2}{\lbd_1}y_1\\&&&1&x_1\\&&&&1\end{pmatrix}$. From $XC = C'X$, we get $a'_1 = a_1$, $b'_1 = b_1$, $c'_2 = c_2$, and the following equations:
\begin{eqnarray}
b_2+ x_1c_2 + \lbd_2y_1a_1 &=& b'_2 + z_2a_1 + \lbd_2x_1b_1\label{EU31}\\
d_1 + \lbd_2x_1b_1 + z_2a_1 &=& d'_1 + \lbd_2y_1a_1 + x_1c_2\label{EU32}\\
a_2 + x_1d_1 + y_2a_1 &=& a'_2 + w_1a_1 + x_1b'_2. \label{EU33}
\end{eqnarray}
\noindent{\bfseries Case $a_1 = 0$:} Equation ~\ref{EU31} becomes $b_1 + x_1c_2 = b'_1 + x_1\lbd_2b_1$. When $c_2 = \lbd_2 b_1$, then $b'_2 = b_2$, and similarly in Equation~\ref{EU32}, $d'_1 = d_1$. Here, if $b_2 = d_1$, we get from Equation~\ref{EU33}, $a'_2 = a_2$. Hence $C$ is reduced to $\begin{pmatrix}1&&b_1&b_2&a_2\\&1&&\lbd_2b_1 & b_2\\ &&1&&\frac{\lbd_2}{\lbd_1}b_1\\ &&&1&\\&&&&1\end{pmatrix}$, and $\ZU{5}(A,B,C) = \ZU{5}(A,B)$. $(A,B,C)$ is a branch of type $UNT_3$, and there are $q^3$ such branches. 

When $b_2 \neq d_1$, choose $x_1$ such that $a'_2 = 0$. $C$ is reduced to 
$\begin{pmatrix}
1& &b_1 & b_2 &\\
&1&&\lbd_2b_1&d_1\\
&&1&&\frac{\lbd_2}{\lbd_1}b_1\\
&&&1&\\
&&&&1
\end{pmatrix}$, and $\ZU{5}(A,B,C) = \left\{\begin{pmatrix}1&&y_1&y_2&x_2\\&1&&z_2&w_1\\&&1&&\frac{\lbd_2}{\lbd_1}y_1\\&&&1&\\&&&&1\end{pmatrix}\right\}$. $(A,B,C)$ is thus of type $R_2$, and there are $q^2(q-1)$ such branches.

When $c_2 \neq \lbd_2b_1$. In this case, in equation~\ref{EU31} itself, we choose $x_1$ such that $b'_2 = 0$. And on substituting $b_2$ with $0$ in this equation and equating, we get $x_1 = 0$. Thus, Equation~\ref{EU32} becomes $d'_1 = d_1$, and from Equation~\ref{EU33}, we get $a'_2 = a_2$. Thus $C$ is reduced to 
$\begin{pmatrix}
1&&b_1&&a_2\\
&1&&c_2&d_1\\
&&1&&\frac{\lbd_2}{\lbd_1}b_1\\
&&&1&\\
&&&&1
\end{pmatrix}$, and $\ZU{5}(A,B,C) = \left\{\begin{pmatrix}1&&y_1&y_2&x_2\\&1&&z_2&w_1\\&&1&&\frac{\lbd_2}{\lbd_1}y_1\\&&&1&\\&&&&1\end{pmatrix}\right\}$. This too is a branch of type $R_2$, and there are $q^3(q-1)$.

\noindent{\bfseries Case $a_1\neq 0$:} In Equation~\ref{EU31} choose $z_2$ such that $b'_2 =0$. Thus, substituting $b_2$ with $b'_2 = 0$ in this equation, leads us to $z_2 = \lbd_2 y_1 + \frac{(c_2 - \lbd_2b_1)}{a_1}x_1$. With these Equation~\ref{EU32} becomes $d'_1 = d_1$. Thus Equation~\ref{EU33} becomes $a_2 + x_1d_1 + y_2a_1 = a'_2 + w_1a_1$. Choose $w_1$ such that $a'_2 = 0$. Thus $C$ is reduced to
$\begin{pmatrix}
1&a_1&b_1&&\\
&1&\lbd_1a_1&c_2&d_1\\
&&1&\lbd_2a_1&\frac{\lbd_2}{\lbd_1}b_1\\
&&&1&a_1\\
&&&&1
\end{pmatrix}$, and 
$$\ZU{5}(A,B,C) = \left\{\begin{pmatrix}1&x_1&y_1&y_2&x_2\\&1&\lbd_1x_1&\lbd_2 y_1 + \frac{(c_2 - \lbd_2b_1)}{a_1}x_1&y_2 + \frac{d_1}{a_1}x_1\\&&1&\lbd_2 x_1&\frac{\lbd_2}{\lbd_1}y_1\\&&&1&x_1\\&&&&1\end{pmatrix}\right\}.$$
By a routine check, one can see that this centralizer group is commutative. Thus we have a branch of type $R_3$, and there are $(q-1)q^3$ such branches. 

Adding up the branches of type $R_2$, there is a total of $q^2(q-1) + q^3(q-1) = q^4 -q^2$ branches of type $R_2$.
\end{proof}

\section{Commuting Probabilities}\label{SCPk}

The number of simultaneous conjugacy classes of commuting $k$-tuples in $UT_n(\Fq)$ is denoted by $c_{UT}(n, k, q)$ and  the same for $GT_n(\Fq)$ is denoted by $c_{GT}(n, k, q)$. From Lemma 7.1~\cite{SS}, it follows that $c_{GT}(n, k, q)= {\bmf 1}.B_{GT_n(\Fq)}^k.e_1 $ and $c_{UT}(n,k, q)= {\bmf 1}.B_{UT_n(\Fq)}^k.e_1$ where $\bmf{1} = \begin{pmatrix}1&1&\cdots&1\end{pmatrix}$, and $\bmf{e}_1 = {}^t\begin{pmatrix}1&0&0&\cdots &0\end{pmatrix}$. We note that all of the branching matrices computed in this paper for triangular and unitriangular groups have entries polynomial in $q$ with integer coefficients. Thus, $c_{UT}(n,k, q)$ for $n=3,4, 5$ and $c_{GT}(n, k,q)$ for $n=2,3,4$ are polynomials in $q$ with integer coefficients.

From Theorem 1.1 in~\cite{SS}, for $k \geq 2$, and any finite group $G$, the probability that a $k$-tuple commutes is $cp_k(G) = \displaystyle\frac{c_G(k-1)}{|G|^{k-1}} = \frac{\bmf{1}B_G^{k-1}.e_1}{|G|^{k-1}}$. Now, that we have determined the branching matrix for the groups $GT_i(\Fq)$ for $i = 2,3,4$, and $UT_j(\Fq)$ for $j = 3,4,5$, for each of the groups, we will mention the commuting probabilities for $k\leq 5$. This computation is done using Sage~\cite{Sagemath}.

For the triangular groups we have:
\begin{center}
\begin{tabular}{c|c||c|c}\hline
$k$ & $cp_k(GT_2(\Fq))$ & $k$ & $cp_k(GT_2(\Fq))$\\ \hline
$2$&$\frac{1}{q - 1}$& $4$&$\frac{q^{2} - 2 q + 4}{q^{5} - 3 q^{4} + 3 q^{3} - q^{2}}$\\ &&&\\
$3$& $\frac{q^{2} - q + 2}{q^{4} - 2 q^{3} + q^{2}}$ & $5$ & $\frac{q^{4} - 3 q^{3} + 7 q^{2} - 3 q + 2}{q^{8} - 4 q^{7} + 6 q^{6} - 4 q^{5} + q^{4}}$\\ \hline
\end{tabular}
\end{center}
\vskip2mm
\begin{center}
\begin{tabular}{c|c||c|c}\hline
$k$ & $cp_k(GT_3(\Fq))$ & $k$ & $cp_k(GT_3(\Fq))$\\ \hline
$2$&$\frac{q^{2} + q - 1}{q^3(q-1)^2}$& $4$&$\frac{q^{5} - 3 q^{4} + 7 q^{3} - 5 q^{2} + 11 q + 4}{q^8(q-1)^6}$\\&&&\\
$3$& $\frac{q^{3} - q^{2} + q + 5}{q^5(q-1)^4}$ & $5$ & $\frac{q^{7} - 5 q^{6} + 17 q^{5} - 32 q^{4} + 54 q^{3} - 34 q^{2} + 25 q + 2}{q^{11}(q-1)^8}$\\ \hline
\end{tabular}
\end{center}
\vskip2mm 
\begin{center}
\begin{tabular}{c|c}\hline
$k$ & $cp_k(GT_4(\Fq))$ \\ \hline
$2$ & $\frac{q^{3} + 3 q^{2} - 2 q - 1}{q^{10}(q-1)^3}$\\ &\\
$3$ & $\frac{12 q^{5} - 52 q^{4} + 116 q^{3} - 97 q^{2} + 63 q - 37}{q^{20}(q-1)^6}$ \\ &\\
$4$ & $\frac{6 q^{8} - 16 q^{7} + 3 q^{6} + 195 q^{5} - 593 q^{4} + 1105 q^{3} - 1129 q^{2} + 912 q - 477}{q^{30}(q-1)^9}$ \\&\\
$5$ & $\frac{7 q^{11} - 32 q^{10} + 122 q^{9} - 192 q^{8} + 342 q^{7} - 714 q^{6} + 2038 q^{5} - 3954 q^{4} + 6136 q^{3} - 6304 q^{2} + 4596 q - 2213}{q^{40}(q-1)^{12}}$ \\ \hline
\end{tabular}
\end{center}
\vskip2mm 
In the case of unitriangular group we have: 
\begin{center}
\begin{tabular}{c|c||c|c}\hline
$k$ & $cp_k(UT_3(\Fq))$ & $k$ & $cp_k(UT_3(\Fq))$ \\ \hline
$2$ & $\frac{q^{2} + q - 1}{q^{3}}$ & $3$ & $\frac{q^{3} + q^{2} - 1}{q^{5}}$\\ &&&\\
$4$ & $\frac{q^{4} + q^{3} - 1}{q^{7}}$ & $5$ & $\frac{q^{5} + q^{4} - 1}{q^{9}} $.\\  \hline
\end{tabular}
\end{center}
\vskip2mm 
\begin{center}
\begin{tabular}{c|c||c|c}\hline
$k$ & $cp_k(UT_4(\Fq))$ & $k$ & $cp_k(UT_4(\Fq))$\\ \hline
$2$ & $\frac{2 q^{3} - 1}{q^{6}}$& $4$& $\frac{q^{7} + 3 q^{6} - 3 q^{5} + 5 q^{4} - 4 q^{3} - 3 q + 2}{q^{15}}$\\&&&\\
$3$& $\frac{2 q^{4} + 3 q^{3} - 2 q^{2} - 3 q + 1}{q^{10}}$ & $5$ & $\frac{q^{10} + 2 q^{9} - 2 q^{8} + 3 q^{7} - q^{6} + q^{4} - 3 q^{3} - 2 q + 2}{q^{20}}$\\ \hline
\end{tabular}
\end{center}
\vskip2mm 
\begin{center}
\begin{tabular}{c|c}\hline
$k$ & $cp_k(UT_5(\Fq))$ \\ \hline
$2$ & $\frac{5 q^{4} - 4 q^{3} + 9 q^{2} - 14 q + 5}{q^{10}}$\\ &\\
$3$& $\frac{11 q^{8} - 7 q^{7} + 23 q^{6} - 41 q^{5} + 5 q^{4} + 11 q^{3} + 3 q^{2} - 7 q + 3}{q^{20}}$ \\ &\\
 $4$& $\frac{2 q^{13} + 3 q^{12} + 5 q^{11} + 10 q^{10} - 6 q^{9} - 20 q^{8} + 8 q^{7} - 27 q^{6} + 42 q^{5} - 24 q^{4} + 9 q^{3} + q^{2} - 5 q + 3}{q^{29}}$\\ &\\
 $5$ & $\frac{\begin{psmallmatrix}2 q^{18} + 5 q^{16} - 5 q^{15} + 23 q^{14} - 25 q^{13} + 28 q^{12} - 41 q^{11} + 23 q^{10} - 17 q^{9} + \\10 q^{8} - 25 q^{7} + 18 q^{6} + 23 q^{5} - 26 q^{4} + 7 q^{3} + 3 q^{2} - 5 q + 3\end{psmallmatrix}}{q^{38}}$\\ \hline
\end{tabular}
\end{center}
\appendix
\section{Conjugacy classes of $GT_4(\Fq)$}\label{CCGT4}
The conjugacy classes for upper triangular group can be algorathmically computed following Belitskii's algorithm as described in~\cite{Ko} and in the appendix of~\cite{Bh}. We list them here for the convenience of reader and also to set the notation for types.

\def\arraystretch{1}
\begin{longtable}{|c|c|c|c|} \hline
Class Representatives & $\begin{smallmatrix} \text{Number of}\\ \text{Classes} 
\end{smallmatrix}$ &$\begin{smallmatrix} \text{Order of}\\ \text{Centralizer} 
\end{smallmatrix}$ & $\begin{smallmatrix} \text{Name of}\\ \text{Type} 
\end{smallmatrix}$ \\ \hline
$a_0I_4, a_0 \neq 0 $ & $(q-1)$ & $(q-1)^4q^6$ & $C$ \\ \hline

$\begin{matrix} \begin{psmallmatrix}a&1& &\\&a&&\\&&a&\\&&&a\end{psmallmatrix},~\begin{psmallmatrix} a&&&\\&a&&\\&&a&1\\&&&a\end{psmallmatrix}  \\ a\neq 0 \end{matrix}$ & $2(q-1)$ & $(q-1)^3q^4$ & $A_1$\\ \hline

$\begin{matrix} \begin{psmallmatrix}a&& &\\&a&1&\\&&a& \\&&&a\end{psmallmatrix} \\ a\neq 0\end{matrix}$ & $q-1$ & $(q-1)^3q^4$ & $A'_1$ \\ \hline

$\begin{matrix} \begin{psmallmatrix}a&&1 &\\&a&&\\&&a&\\&&&a\end{psmallmatrix},~\begin{psmallmatrix} a&&&\\&a&&1\\ &&a& \\&&&a \end{psmallmatrix}  \\ a\neq 0 \end{matrix}$ & $2(q-1)$ & $(q-1)^3q^5$ & $A_2$ \\ \hline

$\begin{matrix} \begin{psmallmatrix}a&& &1\\&a&&\\&&a& \\&&&a\end{psmallmatrix} \\ a\neq 0\end{matrix}$  & $q-1$ & $(q-1)^3q^6$ & $A_3$ \\ \hline

$\begin{matrix} \begin{psmallmatrix}a&1& &\\&a&&\\&&a&1 \\&&&a\end{psmallmatrix} \\  a\neq 0\end{matrix}$  & $q-1$ & $(q-1)^2q^4$ & $A_4$ \\ \hline

$\begin{matrix} \begin{psmallmatrix}a&&&1 \\&a&1&\\&&a& \\&&&a\end{psmallmatrix} \\  a\neq 0\end{matrix}$  & $q-1$ & $(q-1)^2q^4$ & $A_5$ \\ \hline

$\begin{matrix} \begin{psmallmatrix}a&&1& \\&a&&1\\&&a& \\&&&a\end{psmallmatrix} \\  a\neq 0\end{matrix}$ & $q-1$ & $(q-1)^2q^5$ & $A_6$ \\ \hline

$\begin{matrix} \begin{psmallmatrix}a&1&& \\&a&1&\\&&a& \\&&&a \end{psmallmatrix}, \begin{psmallmatrix}a&&& \\&a&1&\\&&a&1 \\&&&a\end{psmallmatrix} \\  a\neq 0\end{matrix}$ & $2(q-1)$ & $(q-1)^2q^3$ & $A_7$ \\ \hline

$\begin{matrix} \begin{psmallmatrix}a&1&& \\&a&&1\\&&a& \\&&&a\end{psmallmatrix}, \begin{psmallmatrix}a&&1& \\&a&&\\&&a&1 \\&&&a\end{psmallmatrix}  \\  a\neq 0\end{matrix}$ & $q-1$ & $(q-1)^2q^4$ & $A_8$ \\ \hline

$\begin{matrix} \begin{psmallmatrix}a&1&1& \\&a&&\\&&a&1 \\&&&a\end{psmallmatrix} \\  a\neq 0\end{matrix}$& $q-1$ & $q(q-1)q^4$ & $A_9$ \\ \hline

$\begin{matrix} \begin{psmallmatrix}a&&& \\&a&&\\&&b& \\&&&b\end{psmallmatrix}, \begin{psmallmatrix}a&&& \\&b&&\\&&a& \\&&&b\end{psmallmatrix} \\ \begin{psmallmatrix}a&&& \\&b&&\\&&b& \\&&&a\end{psmallmatrix}; a\neq b\end{matrix}$ & $3(q-1)(q-2)$ & $(q-1)^4q^2$ & $B_1$ \\ \hline

$\begin{matrix} \begin{psmallmatrix}a&&& \\&a&&\\&&a& \\&&&b\end{psmallmatrix}, \begin{psmallmatrix}a&&& \\&a&&\\&&b& \\&&&a\end{psmallmatrix} \\ \begin{psmallmatrix} a&&& \\&b&&\\&&a& \\&&&a \end{psmallmatrix}, \begin{psmallmatrix} b&&& \\&a&&\\&&a& \\&&&a\end{psmallmatrix}; a\neq b\end{matrix}$ & $4(q-1)(q-2)$ & $(q-1)^4q^3$ & $B_2$ \\ \hline

$\begin{matrix} \begin{psmallmatrix}a&1&& \\&a&&\\&&a& \\&&&b\end{psmallmatrix}, \text{and  3 more}\\ \begin{psmallmatrix}a&&& \\&a&1&\\&&a& \\&&&b\end{psmallmatrix},\text{and  3 more}; \\a\neq b\end{matrix}$ & $8(q-1)(q-2)$ & $(q-1)^3q^2$ & $B_3$ \\ \hline

$\begin{matrix} \begin{psmallmatrix}a&&1& \\&a&&\\&&a& \\&&&b\end{psmallmatrix}, \begin{psmallmatrix}a&&&1 \\&a&&\\&&b& \\&&&a\end{psmallmatrix} \\ \begin{psmallmatrix}a&&&1 \\&b&&\\&&a& \\&&&a\end{psmallmatrix}, \begin{psmallmatrix} b&&& \\&a&&1\\&&a& \\&&&a\end{psmallmatrix}; a\neq b\end{matrix}$ & $4(q-1)(q-2)$ & $(q-1)^3q^3$ & $B_4$ \\ \hline

$\begin{matrix} \begin{psmallmatrix}a&1&& \\&a&&\\&&b& \\&&&b\end{psmallmatrix}, \text{and  5 more}; \\a\neq b\end{matrix}$ & $6(q-1)(q-2)$ & $(q-1)^3q^2$ & $B_5$ \\ \hline

$\begin{matrix} \begin{psmallmatrix}a&&& \\&a&&\\&&b& \\&&&c\end{psmallmatrix}, \text{and  5 more}; \\a\neq b\neq c\neq a\end{matrix}$ & $6(q-1)(q-2)(q-3)$ & $(q-1)^4q$ & $B_6$ \\ \hline
\multicolumn{4}{|c|}{The $\Reg$ types} \\\hline
$\begin{matrix}  \begin{psmallmatrix}a&1&& \\&a&1&\\&&a&1 \\&&&a\end{psmallmatrix}, \\a\neq 0\end{matrix}$ & $q-1$ & $(q-1)q^3$ & $R_1$ \\ \hline

$\begin{matrix} \begin{psmallmatrix}a&1&& \\&a&1&\\&&a& \\&&&b\end{psmallmatrix}, \text{and  3 more}; \\a\neq b\end{matrix}$ & $4(q-1)(q-2)$ & $(q-1)^2q^2$ & $R_2$ \\ \hline

$\begin{matrix} \begin{psmallmatrix}a&1&& \\&a&&\\&&b&1 \\&&&b\end{psmallmatrix}, \begin{psmallmatrix}a&&1& \\&b&&1\\&&a& \\&&&b\end{psmallmatrix} \\\begin{psmallmatrix}a&&&1 \\&b&1&\\&&b& \\&&&a\end{psmallmatrix} ; a\neq b\end{matrix}$ & $3(q-1)(q-2)$ & $(q-1)^2q^2$ & $R_3$ \\ \hline

$\begin{matrix} \begin{psmallmatrix}a&1&& \\&a&&\\&&b& \\&&&c\end{psmallmatrix}, \text{and  5 others}; \\a\neq b\neq c\neq a\end{matrix}$ & $6(q-1)(q-2)(q-3)$ & $(q-1)^3q$ & $R_4$ \\ \hline

$\begin{matrix} \begin{psmallmatrix}a&&& \\&b&&\\&&c& \\&&&d\end{psmallmatrix}, \\a\neq b\neq c\neq a\\ a,b,c \neq d\end{matrix}$ & $\begin{smallmatrix}(q-1).(q-2).\\(q-3).q-4)\end{smallmatrix}$ & $(q-1)^4$ & $R_5$ \\ \hline
\end{longtable}

\section{Conjugacy classes of $UT_4(\Fq)$ and $UT_5(\Fq)$}\label{CCUT45}
Understanding conjugacy classes in unitriangular group is a challenging problem. We refer a reader to~\cite{VA1,VA2} for the reference. We list down the same for $UT_4(\Fq)$ and $UT_5(\Fq)$, what we need for our purpose.

\begin{longtable}{|c|c|c|c|}\hline
 Class Representatives & Number of Classes & Centralizer size  & Name of  Type\\ 
 &&in $UT_4(\Fq)$&\\
 \hline
 
 $\begin{psmallmatrix} 1&&& a\\&1&&\\&&1&\\&&&1\end{psmallmatrix}, a \in \Fq $ & $q$ &$q^6$ & $C$ \\ \hline

 $\begin{matrix}\begin{psmallmatrix} 1&&a& \\&1&&\\&&1&\\&&&1\end{psmallmatrix}, \begin{psmallmatrix} 1&&& \\&1&&a\\&&1&\\&&&1\end{psmallmatrix} ,\\ a \in \Fq^* \end{matrix}$ & $(q-1),(q-1)$ &$q^5$ & $A_1$ \\ \hline
 $\begin{psmallmatrix}1&&a& \\&1&&b\\&&1&\\&&&1\end{psmallmatrix}, a,b \in \Fq^*$ & $(q-1)^2$ & $q^5$ & $A_2$ \\ \hline
 $\begin{matrix}\begin{psmallmatrix}1&a&& \\&1&&\\&&1&\\&&&1\end{psmallmatrix},\begin{psmallmatrix}1&&& \\&1&&\\&&1&a\\&&&1\end{psmallmatrix},\\ \begin{psmallmatrix}1&a&& \\&1&&\\&&1&b\\&&&1\end{psmallmatrix}, \begin{psmallmatrix}1&a&& \\&1&&b\\&&1&\\&&&1\end{psmallmatrix},\\\begin{psmallmatrix}1&&a& \\&1&&\\&&1&b\\&&&1\end{psmallmatrix}, \begin{psmallmatrix}1&a&b& \\&1&&\\&&1&c\\&&&1\end{psmallmatrix},\\a, b,c\in \Fq^* \end{matrix}$ & $\begin{matrix}(q-1),(q-1),\\(q-1)^2,(q-1)^2,\\(q-1)^2,(q-1)^3\end{matrix}$ & $q^4$ & $A_3$ \\ \hline
 $\begin{matrix}\begin{psmallmatrix} 1&&& \\&1&a&\\&&1&\\&&&1\end{psmallmatrix}, \begin{psmallmatrix}1&&&b \\&1&a&\\&&1&\\&&&1\end{psmallmatrix},\\a,b \in \Fq^*\end{matrix}$ & $(q-1),(q-1)^2$ & $q^4$ & $R_1$ \\ \hline  
 $\begin{matrix}\begin{psmallmatrix}1&a&& \\&1&b&\\&&1&\\&&&1\end{psmallmatrix},\begin{psmallmatrix}1&&& \\&1&b&\\&&1&a\\&&&1\end{psmallmatrix},\\ \begin{psmallmatrix}1&a&& \\&1&b&\\&&1&c\\&&&1\end{psmallmatrix}, a,b,c\in \Fq^* \end{matrix}$ & $\begin{matrix}(q-1)^2, (q-1)^2,\\ (q-1)^3\end{matrix}$ &$q^3$ & $R_2$ \\ \hline
\end{longtable}

\def\arraystretch{1}
\begin{longtable}{|c|c|c|c|} \hline
Class Representatives & $\begin{matrix} \text{Number of}\\ \text{Classes} \end{matrix}$ &$\begin{matrix} \text{Order of}\\ \text{Centralizer in $UT_{5}(\Fq)$} \end{matrix}$ & $\begin{matrix} \text{Name of}\\ \text{Type} \end{matrix}$ \\ \hline
$\begin{matrix}\left(\begin{smallmatrix} 1&&&& a\\&1&&&\\&&1&&\\&&&1&\\&&&&1\end{smallmatrix}\right), a \in \Fq \end{matrix}$ & $q$ & $q^{10}$ & $C$ \\ \hline

$\begin{matrix} \left(\begin{smallmatrix}1&&&a&\\&1&&&\\&&1&&\\&&&1&\\&&&&1\end{smallmatrix}\right),~\left(\begin{smallmatrix}1&&&&\\&1&&&a\\&&1&&\\&&&1&\\&&&&1\end{smallmatrix}\right)  \\ a\in \Fq^* \end{matrix}$ & $(q-1),(q-1)$ & $q^9$ & $A_1$ \\ \hline

$\begin{matrix} \left(\begin{smallmatrix}1&&&& \\&1&&a&\\&&1&&\\&&&1&\\&&&&1\end{smallmatrix}\right),~\left(\begin{smallmatrix}1&&&& b\\&1&&a&\\&&1&&\\&&&1&\\&&&&1\end{smallmatrix}\right)  \\ a,b\in \Fq^* \end{matrix}$ & $(q-1),(q-1)^2$ & $q^8$ & $A_2$ \\ \hline

 $\begin{matrix} \left(\begin{smallmatrix}1&&a&&\\&1&&&\\&&1&&\\&&&1&\\&&&&1\end{smallmatrix}\right),~\left(\begin{smallmatrix}1&&&&\\&1&&&\\&&1&&a\\&&&1&\\&&&&1\end{smallmatrix}\right) , \\
a,b\in \Fq^* \end{matrix}$  &$\begin{matrix} (q-1),(q-1)^2,\\(q-1),(q-1)^2 \end{matrix}$ & $q^8$ & $A_3$\\ \hline
$\begin{matrix} \left(\begin{smallmatrix}1&&&&\\&1&a&&\\&&1&&\\&&&1&\\&&&&1\end{smallmatrix}\right),~\left(\begin{smallmatrix}1&&&&b\\&1&a&&\\&&1&&\\&&&1&\\&&&&1\end{smallmatrix}\right) , \\
\left(\begin{smallmatrix}1&&&&\\&1&&&\\&&1&a&\\&&&1&\\&&&&1\end{smallmatrix}\right),~\left(\begin{smallmatrix}1&&&&b\\&1&&&\\&&1&a&\\&&&1&\\&&&&1\end{smallmatrix}\right) \\ 
a,b\in \Fq^* \end{matrix}$  &$\begin{matrix} (q-1),(q-1)^2,\\(q-1),(q-1)^2 \end{matrix}$ & $q^7$ & $A_4$\\ \hline

$\begin{matrix} \left(\begin{smallmatrix}1&a&&&\\&1&&&\\&&1&&\\&&&1&\\&&&&1\end{smallmatrix}\right),~\left(\begin{smallmatrix}1&a&&&\\&1&&&b\\&&1&&\\&&&1&\\&&&&1\end{smallmatrix}\right) , \\
\left(\begin{smallmatrix}1&&a&&\\&1&&&\\&&1&&b\\&&&1&\\&&&&1\end{smallmatrix}\right),~\left(\begin{smallmatrix}1&a&&&\\&1&&&\\&&1&&b\\&&&1&\\&&&&1\end{smallmatrix}\right) ,\\ 
\left(\begin{smallmatrix}1&a&c&&\\&1&&&\\&&1&&b\\&&&1&\\&&&&1\end{smallmatrix}\right),~\left(\begin{smallmatrix}1&&&&\\&1&&&\\&&1&&\\&&&1&a\\&&&&1\end{smallmatrix}\right) ,\\ 
\left(\begin{smallmatrix}1&&&b&\\&1&&&\\&&1&&\\&&&1&a\\&&&&1\end{smallmatrix}\right),~\left(\begin{smallmatrix}1&&b&&\\&1&&&\\&&1&&\\&&&1&a\\&&&&1\end{smallmatrix}\right) ,\\ 
\left(\begin{smallmatrix}1&&b&c&\\&1&&&\\&&1&&\\&&&1&a\\&&&&1\end{smallmatrix}\right),~
a,b,c\in \Fq^* \end{matrix}$  &$\begin{matrix} (q-1),(q-1)^2,\\(q-1)^2,(q-1)^2,\\(q-1)^3,(q-1), \\(q-1)^2,(q-1)^2,\\(q-1)^3\end{matrix}$ & $q^7$ & $A_5$\\ \hline

$\begin{matrix} \left(\begin{smallmatrix}1&&&a& \\&1&&&b\\&&1&&\\&&&1&\\&&&&1\end{smallmatrix}\right), a,b\in \Fq^*\end{matrix}$ & $(q-1)^2$ & $q^9$ & $B_1$ \\ \hline
$\begin{matrix} 
\left(\begin{smallmatrix}1&&a&&\\&1&&&b\\&&1&&\\&&&1&\\&&&&1\end{smallmatrix}\right),~\left(\begin{smallmatrix}1&&&b&\\&1&&&\\&&1&&a\\&&&1&\\&&&&1\end{smallmatrix}\right) \\ 
a,b\in \Fq^* \end{matrix}$  &$\begin{matrix} (q-1),(q-1)^2,\\(q-1),(q-1)^2 \end{matrix}$ & $q^8$ & $B_2$\\ \hline

$\begin{matrix} \left(\begin{smallmatrix}1&&a&&\\&1&&b&\\&&1&&\\&&&1&\\&&&&1\end{smallmatrix}\right),~\left(\begin{smallmatrix}1&&&&\\&1&&a&\\&&1&&b\\&&&1&\\&&&&1\end{smallmatrix}\right) , \\ 
a,b\in \Fq^* \end{matrix}$  &$ (q-1)^2,(q-1)^2$ & $q^7$ & $B_3$\\ \hline

$\begin{matrix} \left(\begin{smallmatrix}1&a&&&\\&1&&b&\\&&1&&\\&&&1&\\&&&&1\end{smallmatrix}\right),~\left(\begin{smallmatrix}1&a&&&\\&1&&b&\\&&1&&c\\&&&1&\\&&&&1\end{smallmatrix}\right) , \\
\left(\begin{smallmatrix}1&&&&\\&1&a&&\\&&1&&b\\&&&1&\\&&&&1\end{smallmatrix}\right),~\left(\begin{smallmatrix}1&&&c&\\&1&a&&\\&&1&&b\\&&&1&\\&&&&1\end{smallmatrix}\right) ,\\ 
\left(\begin{smallmatrix}1&&b&&\\&1&&&\\&&1&a&\\&&&1&\\&&&&1\end{smallmatrix}\right),~\left(\begin{smallmatrix}1&b&&&\\&1&&&\\&&1&a&\\&&&1&\\&&&&1\end{smallmatrix}\right), \\ 
\left(\begin{smallmatrix}1&b&c&&\\&1&&&\\&&1&a&\\&&&1&\\&&&&1\end{smallmatrix}\right),~\left(\begin{smallmatrix}1&&b&&\\&1&&&c\\&&1&a&\\&&&1&\\&&&&1\end{smallmatrix}\right), \\ 
\left(\begin{smallmatrix}1&b&&&\\&1&&&c\\&&1&a&\\&&&1&\\&&&&1\end{smallmatrix}\right),~\left(\begin{smallmatrix}1&b&c&&\\&1&&&d\\&&1&a&\\&&&1&\\&&&&1\end{smallmatrix}\right), \\ 
\left(\begin{smallmatrix}1&&&&\\&1&&b&\\&&1&&\\&&&1&a\\&&&&1\end{smallmatrix}\right),~\left(\begin{smallmatrix}1&&c&&\\&1&&b&\\&&1&&\\&&&1&a\\&&&&1\end{smallmatrix}\right) , \\
\left(\begin{smallmatrix}1&&&&\\&1&b&&\\&&1&&\\&&&1&a\\&&&&1\end{smallmatrix}\right),~\left(\begin{smallmatrix}1&&&c&\\&1&b&&\\&&1&&\\&&&1&a\\&&&&1\end{smallmatrix}\right) ,\\ 
\left(\begin{smallmatrix}1&&&&\\&1&b&c&\\&&1&&\\&&&1&a\\&&&&1\end{smallmatrix}\right),~\left(\begin{smallmatrix}1&&&d&\\&1&b&c&\\&&1&&\\&&&1&a\\&&&&1\end{smallmatrix}\right) \\ 
a,b,c,d\in \Fq^* \end{matrix}$  &$\begin{matrix} (q-1)^2,(q-1)^3,\\(q-1)^2,(q-1)^3,\\(q-1)^2,(q-1)^2, \\(q-1)^3,(q-1)^3,\\(q-1)^3, (q-1)^4\\(q-1)^2,(q-1)^3,\\(q-1)^2,(q-1)^3,\\(q-1)^3,(q-1)^4\\ \end{matrix}$ & $q^6$ & $B_4$\\ \hline

$\begin{matrix} \left(\begin{smallmatrix}1&b&&&\\&1&&&\\&&1&&\\&&&1&a\\&&&&1\end{smallmatrix}\right),~\left(\begin{smallmatrix}1&b&&c&\\&1&&&\\&&1&&\\&&&1&a\\&&&&1\end{smallmatrix}\right)  \\ 
a,b,c\in \Fq^* \end{matrix}$  &$ (q-1)^2,(q-1)^3$ & $q^6$ & $B_5$ \\ \hline

$\begin{matrix} \left(\begin{smallmatrix}1&a&&&\\&1&b&&\\&&1&&\\&&&1&\\&&&&1\end{smallmatrix}\right),~\left(\begin{smallmatrix}1&a&&&\\&1&b&&\\&&1&&\\&&&1&c\\&&&&1\end{smallmatrix}\right) , \\ 
\left(\begin{smallmatrix}1&&&&\\&1&&&\\&&1&a&\\&&&1&b\\&&&&1\end{smallmatrix}\right),~\left(\begin{smallmatrix}1&c&&&\\&1&&&\\&&1&a&\\&&&1&b\\&&&&1\end{smallmatrix}\right)  \\ 
a,b,c\in \Fq^* \end{matrix}$  &$\begin{matrix} (q-1)^2,(q-1)^3,\\ (q-1)^2,(q-1)^3\\\end{matrix}$ & $q^5$ & $B_6$\\ \hline

$\begin{matrix} \left(\begin{smallmatrix}1&&a&&\\&1&&b&\\&&1&&c\\&&&1&\\&&&&1\end{smallmatrix}\right),~ a,b,c\in \Fq^*\end{matrix}$  & $(q-1)^3$ & $q^7$ &  $D_1$\\ \hline

$\begin{matrix} \left(\begin{smallmatrix}1&a&&&\\&1&b&&\\&&1&&c\\&&&1&\\&&&&1\end{smallmatrix}\right),~\left(\begin{smallmatrix}1&a&&&\\&1&&b&\\&&1&&\\&&&1&c\\&&&&1\end{smallmatrix}\right) , \\ 
\left(\begin{smallmatrix}1&a&&&\\&1&b&c&\\&&1&&\\&&&1&d\\&&&&1\end{smallmatrix}\right),~\left(\begin{smallmatrix}1&&a&&\\&1&&&\\&&1&b&\\&&&1&c\\&&&&1\end{smallmatrix}\right) , \\ 
\left(\begin{smallmatrix}1&a&b&&\\&1&&&\\&&1&c&\\&&&1&d\\&&&&1\end{smallmatrix}\right),~
a,b,c,d\in \Fq^* \end{matrix}$  &$\begin{matrix} (q-1)^3,(q-1)^3,\\ (q-1)^4,(q-1)^3\\(q-1)^4\end{matrix}$ & $q^5$ & $D_2$\\ \hline

$\begin{matrix} \left(\begin{smallmatrix}1&&&a&\\&1&b&&\\&&1&&\\&&&1&\\&&&&1\end{smallmatrix}\right),~\left(\begin{smallmatrix}1&&&&\\&1&&&a\\&&1&b&\\&&&1&\\&&&&1\end{smallmatrix}\right) , \\ 
a,b\in \Fq^* \end{matrix}$  &$ (q-1)^2,(q-1)^2$ & $q^6$ & $R_1$\\ \hline

$\begin{matrix} \left(\begin{smallmatrix}1&&&&\\&1&a&&\\&&1&b&\\&&&1&\\&&&&1\end{smallmatrix}\right),~\left(\begin{smallmatrix}1&&&&c\\&1&a&&\\&&1&b&\\&&&1&\\&&&&1\end{smallmatrix}\right)  \\ 
a,b,c\in \Fq^* \end{matrix}$  &$ (q-1)^2,(q-1)^3$ & $q^5$ & $R_2$\\ \hline

$\begin{matrix} \left(\begin{smallmatrix}1&a&&&\\&1&b&&\\&&1&c&\\&&&1&\\&&&&1\end{smallmatrix}\right),~\left(\begin{smallmatrix}1&&&&\\&1&a&&\\&&1&b&\\&&&1&c\\&&&&1\end{smallmatrix}\right) , \\
\left(\begin{smallmatrix}1&a&&&\\&1&b&&\\&&1&c&\\&&&1&d\\&&&&1\end{smallmatrix}\right),~
a,b,c,d\in \Fq^* \end{matrix}$  &$\begin{matrix} (q-1)^3,(q-1)^3\\ (q-1)^4\end{matrix}$ & $q^4$ & $R_3$\\ \hline

\end{longtable}


\end{document}